\documentclass{amsart}[12pt, a4paper]

\usepackage{amsmath,amstext,amssymb,mathrsfs,amscd,amsthm}
\usepackage{amsfonts}
\usepackage[all]{xy}
\usepackage{booktabs}
\usepackage{verbatim}
\usepackage{enumerate}
\usepackage[singlelinecheck=false]{caption} 
\usepackage{color}

\usepackage{graphicx}                   
\usepackage{subfig}
\usepackage{xspace}

\newtheorem{lemma}{Lemma}[section]
\newtheorem{proposition}[lemma]{Proposition}
\newtheorem{theorem}[lemma]{Theorem}
\newtheorem{corollary}[lemma]{Corollary}
\newtheorem{criterion}[lemma]{Criterion}

\newtheorem{predf}[lemma]{Definition} 
\newenvironment{df}{\begin{predf}\rm}{\end{predf}}
\newtheorem{preremark}[lemma]{Remark}  
\newenvironment{remark}{\begin{preremark}\rm}{\end{preremark}}
\newtheorem*{prenotation}{Notation}
\newenvironment{notation}{\begin{prenotation}\rm}{\end{prenotation}}

\numberwithin{equation}{section}

\newcommand\mmnote[1]{}
\newcommand{\nin}{\not\in}

\newcommand{\bC}{\mathbb{C}}

\newcommand{\bF}{\mathbb{F}}

\newcommand{\bP}{\mathbb{P}}

\newcommand{\bR}{\mathbb{R}}

\newcommand{\bZ}{\mathbb{Z}}

\newcommand\lra{\longrightarrow}

\newcommand\Diff{\mathrm{Diff}}
\newcommand\Emb{\mathrm{Emb}}

\newcommand\Bun{\mathrm{Bun}}
\newcommand\Th{\mathrm{Th}}
\newcommand\colim{\operatorname*{colim}}

\newcommand{\Fib}{\mathrm{Fib}}

\newcommand{\map}{\mathrm{map}}

\def\vv{v}
\def\ww{w}
\def\ppartial{\pp}
\newcommand\mirror[1]{\reflectbox{$#1$}}

\def\resp{resp. \!}                         										
\def\Id{{\mathrm{Id}}}        																	
\newcommand\hofib{\mathrm{hofib}}
\def\Im{\operatorname*{Im}}       															
\newcommand\cl[1]{\operatorname{cl}{#1}}												
\def\rigtharrow{\rightarrow}

\newcommand\stminus[2]{#1(#2)}																		
\def\manifold{manifold\xspace}
\def\manifolds{manifolds\xspace}

\def\neat{neat\xspace}              														

\def\d{\delta}                    															
\def\pp{\partial^0}\def\dd{\delta^0}														
\def\ppp{\partial^1}\def\ddd{\delta^1}													
\newcommand\y[1]{\tilde{#1}}																		
\newcommand\z[1]{\overline{#1}}																			
\def\SS{\Sigma}              																		

 \def\Diffor{\Diff^+} 				    																
\def\E{\mathcal{E}}          																		
\newcommand{\e}[2]{\E^+_{#1}(#2)}    															
\newcommand{\JJ}[2]{J_{#1}(#2)}														

\newcommand{\N}{N}																			
\newcommand{\Ncl}{\overline{\N}}												
\def\Tub{\mathrm{\overline{Tub}}}
\def\TubEmb{{\mathrm{\overline{\mathrm{T}}\Emb}}}
\def\GL{\mathrm{GL}}

\def\op{\mathrm{op}}
\def\inj{\mathrm{inj}}                                					

\newcommand{\OO}[3]{{\mathcal O}_{#1}(#2)_{#3}}       					
\newcommand{\h}[3]{{\mathcal O}^1_{#1}(#2)_{#3}}      					
\renewcommand{\b}[3]{{\mathcal O}^2_{#1}(#2)_{#3}}    					
\newcommand\A[2]{A_{#1}(#2)}

\def\aa{\mathfrak{a}}                  													
\def\ab{\mathfrak{b}}																						
\def\ii{\mathfrak{i}} 																					

\newcommand{\p}[3]{\mathcal{P}_{#1}(#2)_{#3}}     												

\def\gg{\mathfrak g}																						

\def\Ss{{\mathcal S}_2^\perp}																    

\def\Sss{{\mathcal S}_{n-2}}																		

\def\s{\mathscr{S}^\tau} 																						
\def\sss{\mathscr{S}^\nu}																				
\def\Gr{{\rm Gr}}                   														
\def\Top{{\rm\bf Top}}
\def\arctg{\operatorname*{arctan}}			
\def\exp{\operatorname*{exp}}																		
\def\fib{\mathrm{fib}}																					
\newcommand{\Dp}[2]{D_\partial(#1;#2)_{\bullet}}
		
\newcommand{\D}[2]{D(#1;#2)_\bullet}             			
\def\DDn{D^\natural}
\newcommand{\Dn}[2]{\DDn(#1;#2)_\bullet}               
\newcommand{\Dpn}[2]{D^\natural_\partial(#1;#2)_\bullet} 
\def\H{\mathcal H}      																				

\newcommand{\Gmc}[1]{\Gamma_c(\Sss(T{#1}) \to #1)}

\newcommand{\Gmg}[2]{\Gamma_c(\Sss(T#1) \to #1)_{#2}}

\def\Sq{\mathrm{Sq}}																						
\newcommand{\en}[2]{\E^\nu_{#1}(#2)}														
\newcommand{\f}[2]{\mathcal{F}_{#1}(#2)}     									  
\newcommand{\g}[2]{\mathcal{G}_{#1}(#2)}     									  
\newcommand{\supp}{\operatorname*{supp}}

\SelectTips{cm}{10}

\title{Homological stability for spaces of embedded surfaces}
\author{Federico Cantero Mor\'an}
\thanks{F.\ Cantero Mor\'an was funded through FPI Grant BES-2008-002642 and by Michael Weiss Humboldt professor grant. He was partially supported by project MTM2013-42178-P funded by the Spanish Ministry of Economy. He specially thanks Nathalie Wahl and the Centre for Symmetry and Deformation for their hospitality and the organisers of the 2010 West Coast Algebraic Topology Summer School, where this collaboration began.}

\email{fcant\_01@uni-muenster.de}
\address{Mathematisches Institut \\
Universit\"at M\"unster \\
Einsteinstr. 62 \\
D-48149 M\"unster}

\author{Oscar Randal-Williams}
\thanks{O.\ Randal-Williams was supported by ERC Advanced Grant No.\228082, the Danish National Research Foundation through the Centre for Symmetry and Deformation, and the Herchel Smith Fund.}
\email{o.randal-williams@dpmms.cam.ac.uk}
\address{DPMMS\\
Wilberforce Road\\
Cambridge CB3 0WB\\
UK}
\date{\today}

\begin{document}
\begin{abstract}
We study the space of oriented genus $g$ subsurfaces of a fixed manifold $M$, and in particular its homological properties. We construct a ``scanning map'' which compares this space to the space of sections of a certain fibre bundle over $M$ associated to its tangent bundle, and show that this map induces an isomorphism on homology in a range of degrees. 

Our results are analogous to McDuff's theorem on configuration spaces, extended from 0-dimensional submanifolds to 2-dimensional submanifolds.
\end{abstract}
\maketitle
\section{Introduction}

Let $M$ be a smooth manifold, not necessarily compact and possibly with boun\-da\-ry. Our object of study will be certain spaces of oriented surfaces in $M$, which we define as follows. Let $\Sigma_g$ denote a connected closed oriented smooth surface of genus $g$, and let $\Emb(\Sigma_g, M)$ denote the space of all smooth embeddings of this surface into the interior of $M$, equipped with the $C^\infty$ topology. The topological group $\Diff^+(\Sigma_g)$ of orientation preserving diffeomorphisms 
acts continuously and freely on $\Emb(\Sigma_g, M)$, and we define
$$\E^+(\Sigma_g, M) := \Emb(\Sigma_g, M) / \Diff^+(\Sigma_g)$$
to be the quotient space. As a set, $\E^+(\Sigma_g, M)$ is in bijection with the set of all subsets of $M$ which are smooth manifolds diffeomorphic to $\Sigma_g$, equipped with an orientation: hence we refer to $\E^+(\Sigma_g, M)$ as the \emph{moduli space of genus $g$ oriented surfaces in $M$}.

We will study the space $\E^+(\Sigma_g, M)$ using a technique called \emph{scanning}, which compares this space of surfaces in $M$ with a certain space of ``formal surfaces in $M$''. In order to introduce this space, we define, for an inner product space $(V, \langle -, - \rangle)$, the space \label{def:ssv}
$$\Ss(V):= \Th(\gamma_2^\perp \to \Gr_2^+(V)).$$
That is, we take the Grassmannian $\Gr_2^+(V)$ of oriented $2$-planes in $V$, consider the tautological $2$-plane bundle $\gamma_2 \subset V \times \Gr_2^+(V)$, and form its orthogonal complement $\gamma_2^\perp$ using the inner product on $V$. Then we take the Thom space of this vector bundle. We will denote the point at infinity by $\infty \in \Ss(V)$.

If $V \to B$ is a vector bundle with metric, we let $\Ss(V) \to B$ be the fibre bundle obtained by performing this construction fibrewise to $V$. The constant section with value $\infty$ in each fibre gives a canonical section of this bundle.


We fix a Riemannian metric $\gg$ on $M$. The \emph{space of formal surfaces} in $M$ is defined to be
$$\Gamma_c(\Ss(TM) \to M;\infty),$$
the space of sections of $\Ss(TM) \to M$ which are compactly supported, i.e., agree with the canonical section $\infty$ outside of a compact set and on $\partial M$. Every such section chooses for each point $x \in M$ either an oriented affine 2\nobreakdash-dimensional subset of $T_x M$ or the empty subset. The scanning construction associates to each oriented surface $\Sigma \subset M$ such a section by---loosely speaking---assigning to each $x \in M$ the best approximation to $\Sigma$ by an affine subset of $T_x M$.

To make this precise, we let $\E^\tau(\Sigma_g, M) \subset (0, \infty) \times \E^+(\Sigma_g, M)$ be the set of pairs $(\epsilon, W)$ such that the exponential map $\mathrm{exp} \colon NW \to M$ from the normal bundle of $W$ to $M$ (defined using the metric $\gg$ on $M$) restricts to an embedding of the subspace $N^\epsilon W \subset NW$ of vectors of length $< \epsilon$. We also fix a smooth family of diffeomorphisms $\varphi_\epsilon\colon [0,\infty)\cong [0,\epsilon)$.

We then define a map $M \times \E^\tau(\Sigma_g, M) \to  \Ss(TM)$ by 
\begin{align*}
(p, (\epsilon, W)) & \longmapsto 
\begin{cases}
\infty \in \Ss(T_p M) & \text{if } \begin{array}{l} p \nin \exp(N^\epsilon W),\end{array}\\
(D_v\exp)(T_q W -\varphi_\epsilon^{-1}(\|v\|)\cdot v) \subset T_pM & \text{if } p = \exp(v \in N_q^\epsilon W),
\end{cases}
\end{align*}
where we consider the affine oriented 2-plane $T_q W -\varphi_\epsilon^{-1}(\|v\|)\cdot v $ in $T_q W \oplus N_q W$ as lying inside $T_v(N^\epsilon W)$ using the canonical linear isomorphism between these last two vector spaces.
The adjoint to this map,
\begin{equation*}\label{eq:scanning}
\mathscr{S}_g^\tau \colon \E^\tau(\Sigma_g, M) \lra \Gamma_c(\Ss(TM) \to M;\infty),
\end{equation*}
is the \emph{scanning map} (cf.\ Section \ref{section:closed2}). As the forgetful map $\E^\tau(\Sigma_g, M) \to \E^+(\Sigma_g, M)$ is a weak homotopy equivalence, we often consider $\mathscr{S}_g^\tau$ as a map from $\E^+(\Sigma_g, M)$.

In Section \ref{sec:PathComp} we construct a function $\chi \colon \Gamma_c(\Ss(TM) \to M;\infty) \to \bZ$ such that $\chi \circ \mathscr{S}_g^\tau$ takes constant value $2-2g$; we think of $\chi$ as sending a formal surface to its Euler characteristic, and write $\Gamma_c(\Ss(TM) \to M;\infty)_g = \chi^{-1}(2-2g)$. The simplest form of our theorem is then as follows.

\begin{theorem}\label{thm:Main}
If $M$ is simply-connected and of dimension at least $5$, then the scanning map $\mathscr{S}_g^\tau \colon \E(\Sigma_g, M) \to \Gamma_c(\Ss(TM) \to M;\infty)_g$ induces an isomorphism in integral homology in degrees smaller than or equal to $\tfrac{1}{3}(2g-2)$.
\end{theorem}

This should be compared with the theorem of McDuff \cite[Theorem 1.1]{McDuff}, which can be viewed as a similar result for spaces of embedded $0$-manifolds, i.e.\ configuration spaces.

\begin{remark}
In principle the homology of spaces of embeddings of a surface into Euclidean space modulo diffeomorphisms can be computed as follows: by \cite{ALV} the rational homology of the fibre of the map from embeddings in Euclidean space to immersions can been computed; by Smale--Hirsch theory the homology of the space of immersions can be computed; then one can study the spectral sequence for the action of the diffeomorphism group of the surface on the space of embeddings. In practice, there are many difficulties in following this programme. By contrast, Theorem \ref{thm:Main} can be used to compute these rational homology groups in one step. Such calculations will appear in forthcoming work of the first author.
\end{remark}

Theorem \ref{thm:Main} follows from two rather more technical results. First, a homology stability theorem analogous to Harer stability \cite{H}. To make sense of such a result it is essential to discuss surfaces with boundary, and we will shortly define spaces of surfaces with boundary inside a manifold $M$ with non-empty boundary. A large part of our work is devoted to proving this homology stability theorem, which requires new techniques. The second technical result is analogous to the Madsen--Weiss theorem \cite{MW}, and identifies the stable homology of these spaces of surfaces. To describe these results we must introduce more terminology.

\subsection{Surfaces with boundary}

Now we suppose that $M$ has non-empty boundary $\partial M$ and that we are given a collar $C \colon \partial M \times [0,1) \hookrightarrow M$. Let $\Sigma_{g,b}$ be a fixed smooth oriented surface of genus $g$ with $b$ boundary components, and let $c \colon \partial \Sigma_{g,b} \times [0,1) \hookrightarrow \Sigma_{g,b}$ be a collar. We also fix an embedding $\delta \colon \partial \Sigma_{g,b} \hookrightarrow \partial M$, which we call a \emph{boundary condition}.

Let $\Emb(\Sigma_{g,b}, M;\delta)$ be the set of those embeddings $f$ that extend to an embedding $F\colon \SS_{g,b}\cup (\partial \SS_{g,b}\times I)\to M\cup (\partial M\times I)$ such that for all $x\in \partial \SS_{g,b}$ and $t\in I$,
\begin{align*}
f(x) = \delta(x), \quad  F(x,t) = (f(x),t).
\end{align*}

Similarly, let $\Diff^+(\SS_{g,b})$ be the group of those diffeomorphisms that extend to a diffeomorphism of $\SS_{g,b}\cup (\partial \SS_{g,b}\times I)$ that is the identity on $\SS_{g,b}\times I$. We endow both sets with the Whitney topology and define
$$\E^+(\Sigma_{g,b}, M;\delta) = \Emb(\Sigma_{g,b}, M;\delta) / \Diff^+(\Sigma_{g,b}).$$

Let $Q\colon \partial M \leadsto N$ be a cobordism which is collared at both boundaries. Then we can glue $Q$ to $M$ along $\partial M$ to obtain a new manifold $M \circ Q$ (using the collars to obtain a smooth structure). Similarly, if $e \colon \Sigma_{b + b'} \hookrightarrow Q$ is an embedding of a surface with $b$ boundary components in $\partial M$ and $b'$ in $N$ (which we call the \emph{incoming} and \emph{outgoing} boundaries $\partial_{\mathrm{in}}$, $\partial_{\mathrm{out}}$ respectively) such that every component of the surface intersects the incoming boundary, and if $e(\partial_{\mathrm{in}}\Sigma_{b+b'}) = \mathrm{Im}(\delta)$, we obtain a \emph{gluing map}
\begin{align*}
\E^+(\Sigma_{g,b}, M;\delta) &\lra \E^+(\Sigma_{h, b'}, M \circ Q;e\vert_{\partial_{\mathrm{out}} \Sigma_{b+b'}})\\
(W \subset M) & \longmapsto (W \cup e(\Sigma_{b+b'}) \subset M \circ Q)
\end{align*}
where the value of $h$ depends on the combinatorics of the topology of $\Sigma_{b + b'}$ (note that, as $\Sigma_{g,b}$ is connected and every component of $\Sigma_{b+b'}$ intersects the incoming boundary, $\Sigma_{g, b} \cup \Sigma_{b+b'}$ is connected).

In particular, if we let $Q = \partial M \times [0,1]$ and choose a diffeomorphism $M \circ Q \cong M$ (for example by reparametrising the collar in M), we obtain gluing maps $\E^+(\Sigma_{g,b}, M;\delta) \to \E^+(\Sigma_{h, b'}, M;\delta')$.

\subsection{Homological stability}

There are three basic types of gluing maps which suffice to generate all general gluing maps under composition. These are when $\Sigma_{b+b'}$ is
\begin{enumerate}
\item the disjoint union of a pair of pants with the legs as incoming boundary and a collection of cylinders;
\item the disjoint union of a pair of pants with the waist as incoming boundary and a collection of cylinders;
\item the disjoint union of a disc with its boundary incoming and a collection of cylinders. 
\end{enumerate}
When these surfaces are embedded in $\partial M \times [0,1]$, we call them \emph{stabilisation maps}, and we denote them by
\begin{eqnarray*}
\alpha_{g,b} = \alpha_{g,b}(M;\delta,\delta')\colon \E^+(\SS_{g,b},M;\delta)&\longrightarrow &\E^+(\SS_{g+1,b-1},M;\delta')\\[0.2cm]
\beta_{g,b} = \beta_{g,b}(M;\delta,\delta')\colon \E^+(\SS_{g,b},M;\delta)&\longrightarrow& \E^+(\SS_{g,b+1},M;\delta') \\[0.2cm]
\gamma_{g,b} = \gamma_{g,b}(M;\delta,\delta')\colon \E^+(\SS_{g,b},M;\delta)&\longrightarrow& \E^+(\SS_{g,b-1},M;\delta').
\end{eqnarray*}
As a warning to the reader, we remark that \emph{the notation does not determine the map}: we will often write, for example, $\beta_{g,b}$ to denote \emph{any} gluing map of this type. There are many because there can be many non-isotopic embeddings of $\Sigma_{b+b'}$ into $\partial M \times [0,1]$.
\vspace{-0.5cm}
\begin{figure}[h]
\centering
\subfloat[$\alpha_{2,2}(D^3)$]{\includegraphics[trim=0mm 0mm 0mm 90mm, clip, width=0.33\textwidth]{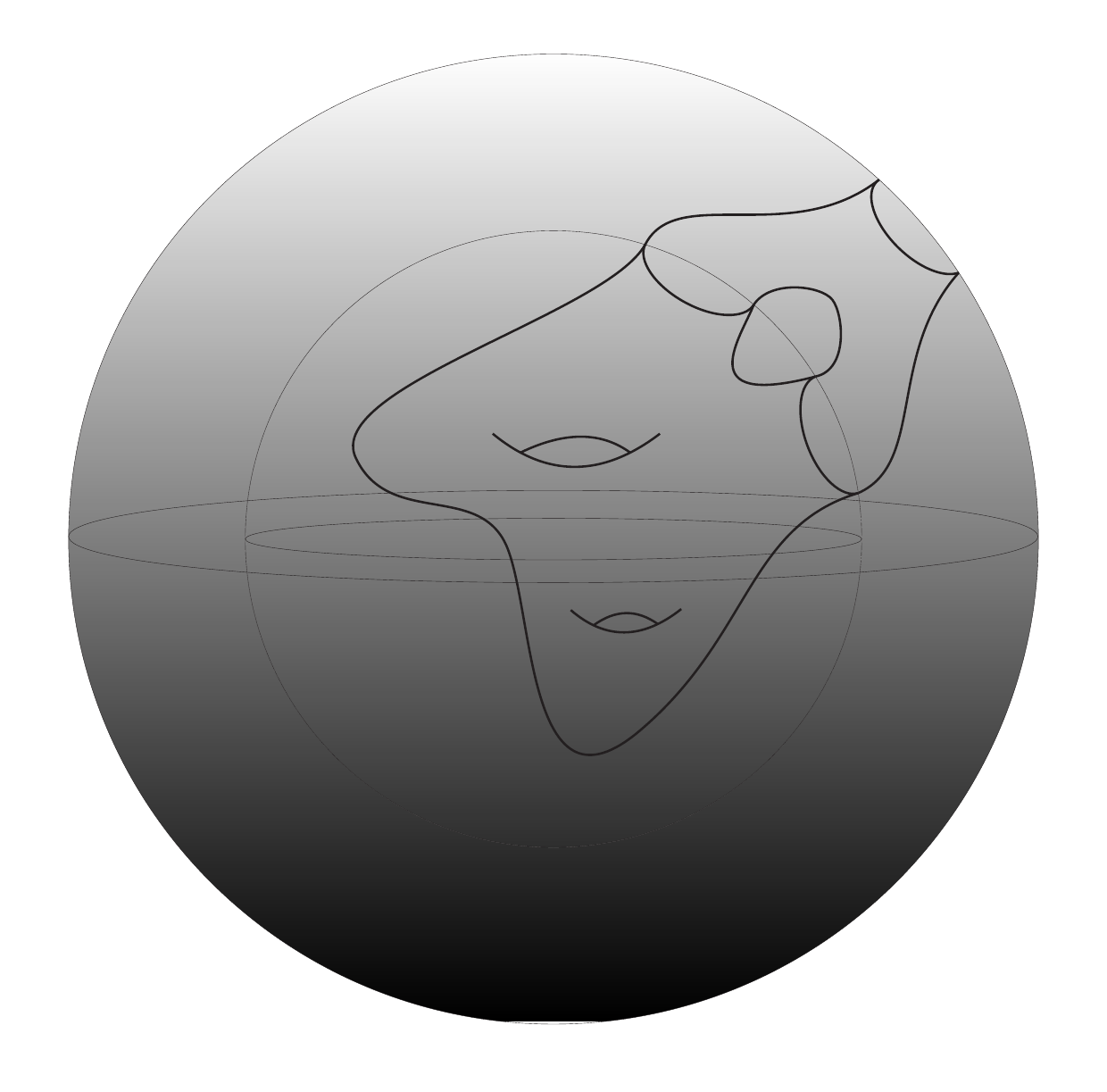}}
\subfloat[$\beta_{2,2}(D^3)$]{\includegraphics[trim=0mm 0mm 0mm 90mm, clip, width=0.33\textwidth]{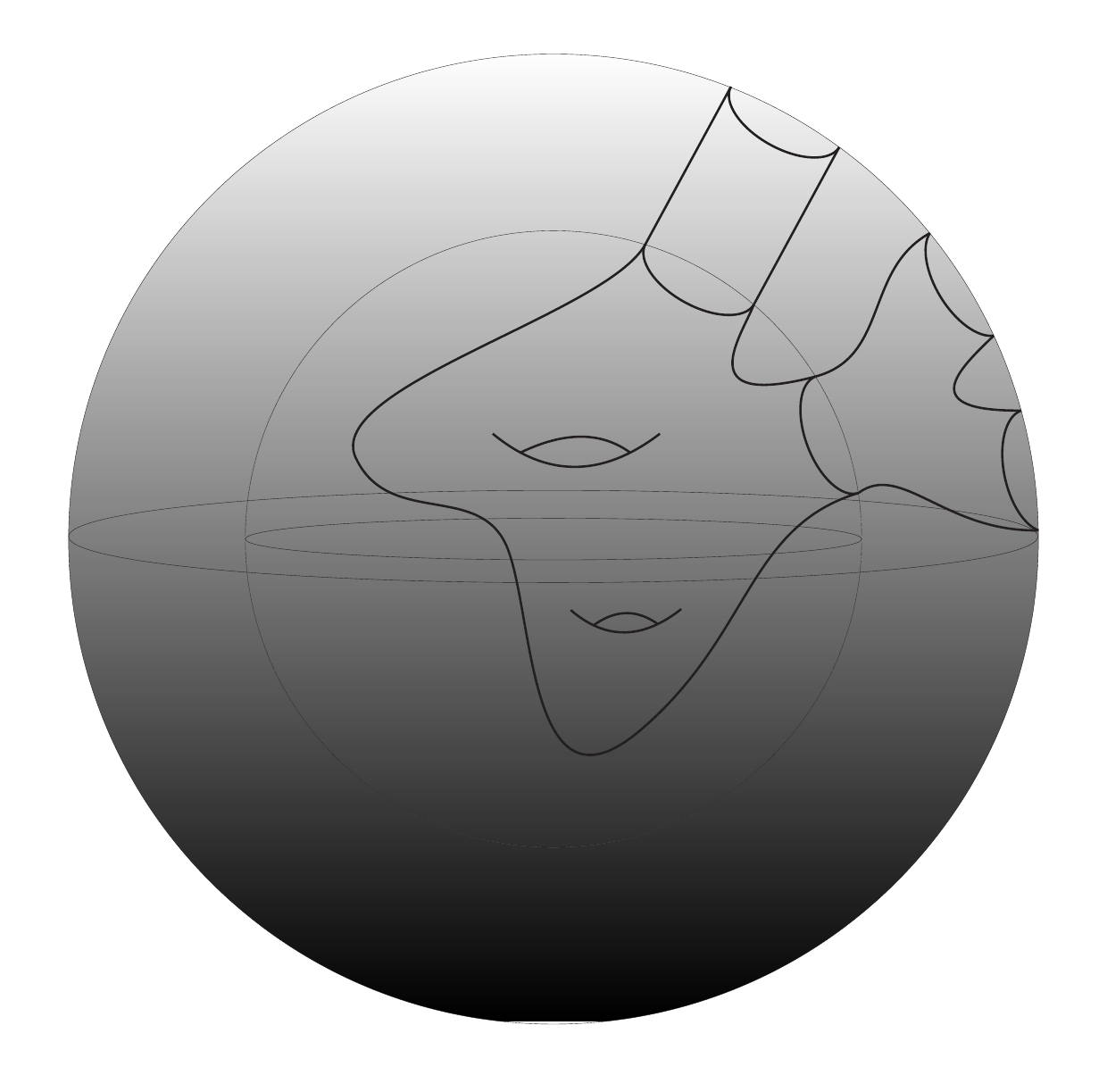}}
\subfloat[$\gamma_{2,2}(D^3)$]{\includegraphics[trim=0mm 0mm 0mm 90mm, clip, width=0.33\textwidth]{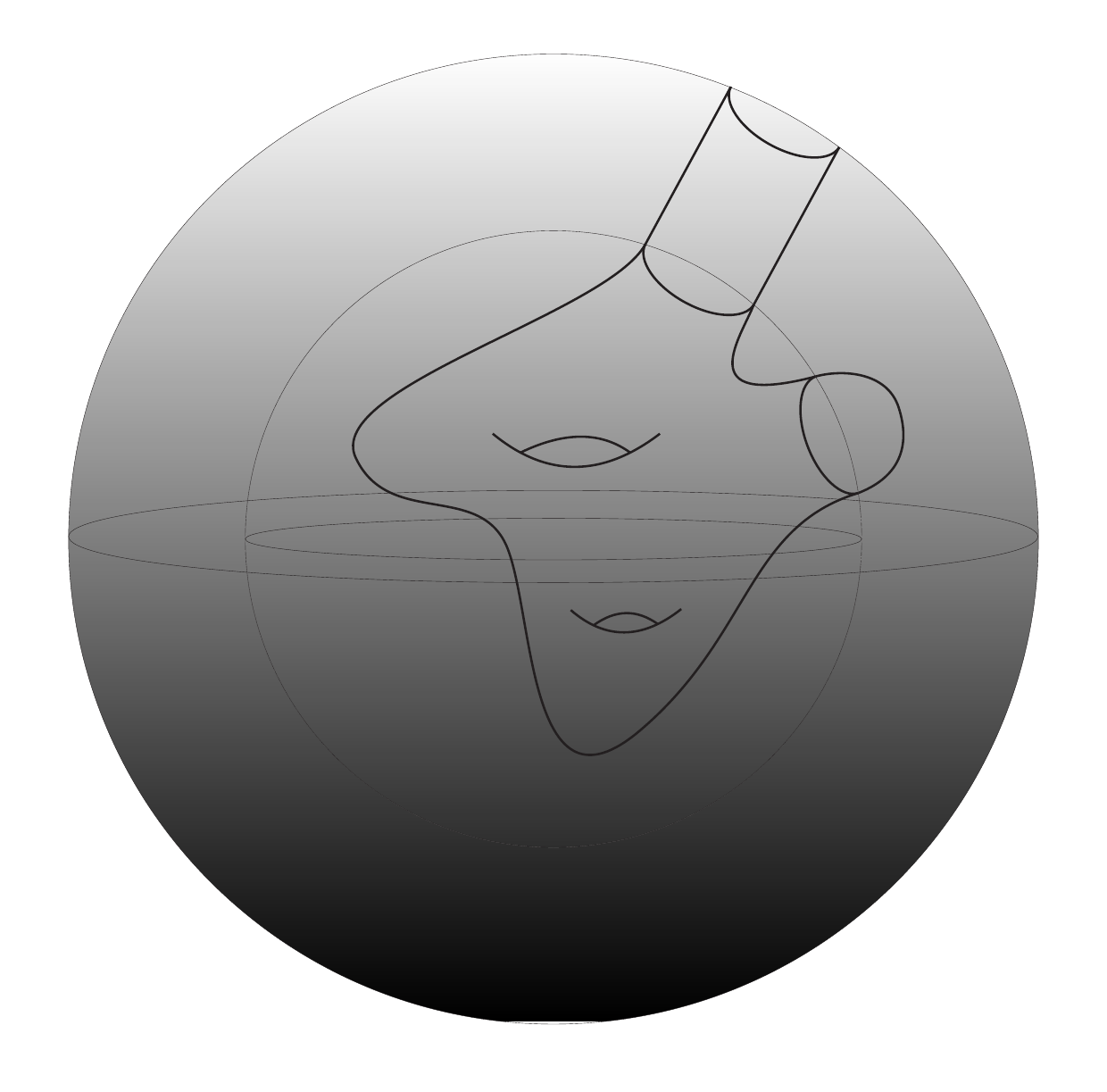}}
\caption{The three basic types of stabilisation maps acting on a surface in the space $\E(\SS_{2,2},D^3;\delta)$ of surfaces in the disc of dimension $3$.}
\end{figure}

The following is our main result concerning homological stability. 

\begin{theorem}\label{thm:stab} 
Let $M$ be a simply-connected manifold of dimension at least $5$. If the dimension of $M$ is $5$, we assume that the pairs of pants defining stabilisation maps are contained in a ball.
\begin{enumerate}
	\item Every map $\alpha_{g,b}$ induces an isomorphism in homology in degrees less than or equal to $\frac{1}{3}(2g-2)$, and an epimorphism in the next degree. 
	\item\label{thm:stab:it:2}  Every map $\beta_{g,b}$ induces an isomorphism in homology in degrees less than or equal to $\frac{1}{3}(2g-3)$, and an epimorphism in the next degree. If one of the newly created boundaries of the pair of pants is contractible in $\partial M$, then the map $\beta_{g,b}$ is also a monomorphism in all degrees.
	\item Every map $\gamma_{g,b}(M;\delta,\delta')$ induces an isomorphism in homology in degrees less than or equal to $\frac{2}{3}g$, and an epimorphism in the next degree. If $b\geq 2$, then it is always an epimorphism.
\end{enumerate}
\end{theorem}
Note that we do \emph{not} require that $\partial M$ is simply-connected, only that $M$ is. Thus there can be many non-isotopic boundary conditions.

These stability ranges are essentially optimal. The space $\E^+(\Sigma_{g,b}, \bR^\infty;\delta)$ is a model for the classifying space of the mapping class group of $\Gamma_{g,b}$, for which the stability ranges of Theorem \ref{thm:stab} are known to be essentially optimal (cf.\ \cite{Boldsen}). 

Another generalisation of McDuff's stability theorem has been developed by Martin Palmer \cite{Palmer-thesis}. He considers the space of all embedded submanifolds of a chosen diffeomorphism type in a background manifold satisfying certain hypotheses, and stabilises by repeatedly adding disjoint copies of a chosen submanifold near the boundary of the background manifold. The case of $1$-dimensional manifolds in $\bR^3$ is not covered by Palmer's result, but is nonetheless true \cite{Kupers:unlinked}.

\subsubsection{Remarks on the proof of Theorem \ref{thm:stab}}

The proof follows the general structure of \cite{R-WResolution2015}, where it is proven that the moduli space of surfaces with tangential structures satisfies homological stability. It is important to observe that embeddings are not determined by a local condition, hence they do not arise as a tangential structure.

As in \cite{R-WResolution2015}, we approximate the moduli space $\E^+(\Sigma_{g,b},M;\delta)$ by a semi-simplicial space whose spaces of $i$-simplices are themselves moduli spaces of surfaces $W$ together with a decomposition of the surface $W$ into $(i+1)$ 1-handles and a surface of smaller genus. In our case, each handle comes with a preferred isotopy to the boundary of $M$. We have found two ways of dealing with this extra data: the first one breaks the problem into two steps, by first considering semi-simplicial spaces $X_\bullet$ that do not incorporate this extra information, and then further taking a semi-simplicial approximation $Y_\bullet$ of the space $X_0$ of $0$-simplices which exclusively contains information about the preferred isotopies to the boundary. The second way uses a new simplicial technique of Galatius and the second author. An earlier draft of this paper used the first method, and the interested reader can find this argument in the second arXiv version of this paper. Here we present the second method, which is somewhat shorter.

The other main difference from \cite{R-WResolution2015} arises in Section \ref{section:triviality}, and concerns the notion called ``1-triviality''. Whereas in \cite{R-WResolution2015} this amounted to (in the case of oriented surfaces) the classification of surfaces, here it involves classifying surfaces embedded in a non-simply-connected manifold (a collar of $\partial M$), and is far more complicated. A final technical difference, not necessary but enlightening, is the technical Lemma \ref{lemma:factor}. This is an alternative to \cite[Lemma~8.2]{R-WResolution2015} and \cite[Lemma~6.2]{Palmer:oriented}, and allows us to work at the level spaces in the whole Section \ref{section:zero-in-homology}, instead of at the level of homology as in those papers.

\subsection{Stable homology}

To identify the stable homology groups resulting from Theorem \ref{thm:stab}, we require a version of the scanning map for surfaces with boundary. We will not describe it in full detail here, but just say that it is a map
$$\mathscr{S}_{g,b} \colon \E^+(\Sigma_{g,b}, M;\delta) \lra \Gamma_c(\Ss(TM) \to M;\bar{\delta})_g$$
to the space of sections of the bundle $\Ss(TM) \to M$ which are compactly supported and which in addition are equal to a fixed section $\bar{\delta} \colon \partial M \to \Ss(TM)|_{\partial M}$ on the boundary.

\begin{theorem}\label{thm:MainH}
If $M$ is simply-connected and of dimension at least $5$, then the map
$$\mathscr{S}_{g, b} \colon \E^+(\Sigma_{g,b}, M;\delta) \lra \Gamma_c(\Ss(TM) \to M;\bar{\delta})_g$$
induces an isomorphism in homology in degrees less than or equal to $\frac{1}{3}(2g-2)$.
\end{theorem}

This theorem will be proved in Sections \ref{section:stablehomology} and \ref{section:surgery} for manifolds $M$ with non-empty boundary. An additional argument in Section \ref{section:closed} establishes this theorem in full generality. Theorem \ref{thm:Main} follows as the particular case when $b=0$.

\subsubsection{Remarks on the proof of Theorem \ref{thm:MainH}}

The proof of this theorem relies on the results of \cite{GMTW} and  \cite{R-WEmbedded}, where the homotopy type of the cobordism category of submanifolds in a background manifold was found, using an $h$-principle argument. This homotopy type in the case $M=\bR^n$ was previously found by Galatius at the end of \cite{galatius-2006}.

The idea of the proof in the case $\partial M \neq \emptyset$ is to choose an embedding $\bR^{n-1} \subset \partial M$ and then consider the space of all (possibly disconnected) embedded surfaces in $M$ as a module over the cobordism category of surfaces in $\bR^{n-1}\times [0,1]$, and apply a group completion argument as in \cite{GMTW}. This is done in Section \ref{section:stablehomology}. As in \cite{GMTW} we are required to perform a parametrised surgery move in order to pass to a subcategory of ``positive boundary'' cobordisms. However, we then require an additional surgery step to pass to the ``connected surfaces'' submodule. In practice we do these two moves simultaneously. 

While the surgery moves, which are done in Section \ref{section:surgery}, broadly follow \cite{GR-W2}, we give a new technical development of this method. For example, we construct the ``do surgery'' maps as semi-simplicial maps, rather than as \emph{ad hoc} maps defined on the geometric realisation as in \cite{GR-W2}. This is a simplification which will be useful elsewhere.

When $\partial M=\emptyset$ we can not use this argument. One might hope then that the method used by McDuff \cite{McDuff} to solve the analogous problem for configuration spaces in closed manifolds could be adapted, but this is not the case (she uses a long exact sequence
\[\cdots\lra H_i(C_k(M\setminus \{\mathrm{pt}\}))\lra H_i(C_k(M))\lra H_{i-\dim M}(C_{k-1}(M\setminus\{\mathrm{pt}\}))\lra \cdots \]
which does not have a useful generalisation in the case of embedded surfaces).

Therefore, when $\partial M=\emptyset$ we introduce a new technique, which we develop in Section \ref{section:closed3}. We approximate the moduli space of surfaces in $M$ by a semi-simplicial space whose space of $i$-simplices is the moduli space of pairs $(N,W)$, where $N\subset M$ is the complement of $(i+1)$ balls in $M$ and $W$ is a surface in $N$. We construct a similar approximation of the target of the scanning map, and we observe that the scanning map extends to a semi-simplicial map between these semi-simplicial spaces. Finally, as $\partial N \neq \emptyset$ we can apply the case of the theorem already proved to show that this extension of the scanning map induces a levelwise homology isomorphism in a range of degrees, from which we deduce the same for the scanning map.

This technique should also be useful elsewhere: for example, it can be adapted to give a new proof of McDuff's theorem for configuration spaces in closed manifolds.

\section{Manifolds, submanifolds, and fibre bundles}
\subsection{Manifolds with corners} 
Let $A \subset \bR^n$, $B \subset \bR^m$, and let $f: A \to B$ be a continuous map. We say that $f$ is  \emph{smooth} if it extends to a smooth map from an open neighbourhood of $A$ to $\bR^m$. If $f$ is a homeomorphism, we say that it is a \emph{diffeomorphism} if it is smooth and has a smooth inverse. In particular, this defines the notion of diffeomorphism between subsets of $\bR^m_+ := [0,\infty)^m$.

A \emph{manifold with corners} $M$ of dimension $m$ is a Hausdorff, second countable topological space locally modeled on the space $\bR^m_+$ and its diffeomorphisms \cite{Cerf,Laures}.

In detail, a \emph{smooth atlas of $M$} is a family of topological embeddings (called \emph{charts}) $\{\varphi_i\colon U_i\rightarrow [0,\infty)^m\}_{i\in I}$, where each $U_i$ is an open subset of $M$ and if $U_i\cap U_j\neq \emptyset$ then the composite $\varphi_j\varphi_i^{-1}$ is a diffeomorphism from $\varphi_i(U_i\cap U_j)$ to $\varphi_j(U_i\cap U_j)$. A \emph{smooth structure} on a Hausdorff, second countable topological space is a smooth atlas which is maximal with respect to inclusion.
 
A \emph{$k$-submanifold} of a manifold with corners is a subset $W\subset M$ such that for each point $p\in W$ there is a chart $(U,\varphi)$ of $M$ with $p\in U$ such that 
$W\cap U = \varphi^{-1}(v+\bR^{k}_+)$, where $v\in \bR^m_+$ (see \cite[Definition~ 1.3.1]{Cerf}).

Let $N$ and $M$ be manifolds with corners. A continuous function $f\colon M\rightarrow \bR$ is \emph{smooth} if the composition $f\circ\varphi\colon U\rightarrow \bR$ is smooth for any chart $(U,\varphi)$ of $M$.  A continuous map $g\colon N\rightarrow M$ is \emph{smooth} if for each smooth function $f\colon M\rightarrow \bR$, the composite $f\circ g$ is smooth. A \emph{diffeomorphism} is a smooth map with smooth left and right inverse. 
An \emph{embedding} is a smooth map whose image is a submanifold and that induces a diffeomorphism onto its image.

The \emph{tangent space $T_p\bR^m_+$} at a point $p\in \bR^m_+$ is defined to be $T_p\bR^m$. Since diffeomorphisms of $\bR^m_+$ induce isomorphisms between tangent spaces of $\bR^m_+$, we can define the tangent bundle of a manifold with corners following the usual procedure. A smooth map $f\colon N\rightarrow M$ induces a fibrewise linear map $Df\colon TN\rightarrow TM$ between tangent bundles. 

If $(U,\varphi)$ is a chart and $p\in U$, then the number $c(p)$ of zero coordinates in $\varphi(p)$ is independent of the chart. The boundary of $M$ is the subspace 
\[\partial M = \{p\in M\mid c(p)>0\}.\] 
A \emph{connected $k$-face} of $M$ is the closure of a component of the subspace \[\{p\in M\mid c(p) = k\}.\] A manifold with corners $M$ is a \emph{manifold with faces} if each $p\in M$ belongs to $c(p)$ connected $1$-faces. A \emph{face} is a (possibly empty) union of pairwise disjoint connected $k$-faces, for some $k$, and is itself a manifold with faces. Observe that the boundary of a $k$-face is a union of $(k-1)$-faces.
\begin{df} If $M$ is a \manifold and $\pp M$ is a $1$-face in $M$, a \emph{collar} of $\pp M$ is an embedding $c$ of the \manifold $\pp M\times[0,1)$ into $M$ such that $c(x,0) = x$ and such that $c_{|(F\cap \pp M)\times [0,1)}$ is a collar of $\pp M\cap F$ in $F$ for any other $1$-face $F$. A \manifold is \emph{collared} if a $1$-face $\pp M$ and a collar of $\pp M$ are given. 
\end{df}
Note that if $M$ is collared, then we can endow $M\cup (\partial M\times I)$ with a canonical smooth structure.

\begin{df} A map $e\colon B\rightarrow M$ between collared manifolds $B$ and $M$ is a map of $B$ into $M$ that extends to a map $$F\colon B\cup (\pp B\times I)\to M\cup (\pp M\times I)$$ such that for all $x\in \pp B$ and $t\in I$, $F(x,t) = (f(x),t)$.
\end{df}

\begin{notation} During the first seven sections, the word \manifold will be used sy\-no\-ny\-mously with collared manifold with faces, and the word map (or embedding, or diffeomorphism) will always refer to a map between collared manifolds.
\end{notation}

A smooth map $f\colon B\rightarrow M$ between \manifolds is \emph{transverse} to a submanifold $W\subset M$ if for each $p\in B$ such that $f(p)\in W$ we have $Df(T_{p}B) + T_{f(p)}W = T_{f(p)}M$. A smooth map $f\colon B\rightarrow M$ to a \manifold $M$ is \emph{transverse to $\partial M$} if it is transverse to any connected face. A \emph{\neat} embedding $f$ is an embedding that is transverse to the boundary and that is collared if $B$ is collared. A \neat sub\-ma\-ni\-fold is the image of a \neat embedding.

If $A$ and $W$ are submanifolds of manifolds $B$ and $M$, a \emph{\neat embedding of the pair $(B,A)$ into the pair $(W,M)$} is an embedding $e\colon B\rightarrow M$ such that $e^{-1}(W)=A$ 
and such that $\dim (e(T_p B) + T_{e(p)} W) = \min\{\dim B + \dim W, \dim M\} $.

\subsection{Boundary conditions for spaces of embeddings}\label{ss:spaces-of-manifolds}


Let $B$ and $M$ be collared manifolds, and $f \colon B \to M$ be an embedding. Following Cerf, we denote by
\[\Emb(B,M;[f]) \subset \Emb(B,M)\]
the subspace of \neat embeddings $g \colon B \to M$ such that for each $x \in B$ the point $g(x)$ lies in the same face of $M$ as $f(x)$, equipped with Whitney $C^\infty$ topology. Similarly, if $f\colon (B,A)\rightarrow (M,W)$ is an embedding of a pair, we denote by 
\[\Emb((B,A),(M,W);[f])\subset \Emb(B,M;[f])\]
the subspace consisting of \neat embeddings of pairs.

We call a function
$$q : \{\text{connected faces of $B$}\} \lra \{\text{subspaces of $M$}\}$$
a \emph{face constraint}, and we let
\[\Emb(B,M;q) \subset \Emb(B,M)\]
denote the subspace of those embeddings $g : B \to M$ such that for each face $F \subset B$, $g(F) \subset q(F)$. Similarly we denote by 
\[\Emb((B,A),(M,W);q)\subset \Emb(B,M;q)\]
the subspace consisting of \neat embeddings of pairs.

We denote by $\Diff(M)$ (\resp $\Diff_c(M)$) the space of diffeomorphisms (\resp compactly supported diffeomorphisms) of $M$ which restrict to diffeomorphisms of each connected face, endowed with the Whitney $C^\infty$ topology. If $W_1,\ldots,W_k$ is a set of submanifolds of $M$ then we denote by \[\Diff(M;W_1,\ldots,W_k)\subset \Diff(M)\] the subspace of those diffeomorphisms of $M$ which restrict to a diffeomorphism of each $W_j$, which is orientation-preserving if $W_j$ is oriented. We denote by $\Diff_\partial(M)$ the space of diffeomorphisms of $M$ which restrict to the identity on the boundary.

For a boundary condition $q$, we let $\Diff(M;q)$ be $\Diff(M;W_1,\ldots,W_k)$ where $W_1,\ldots,W_k$ are the values of the function $q$.

If $f,g\colon B\rightarrow M$ are \neat embeddings, we say that $f$ and $g$ \emph{have the same jet along $\partial M$} if $f^{-1}(\partial M) = g^{-1}(\partial M)$ and all the partial derivatives of $f$ and $g$ at all points in $f^{-1}(\partial M)$ agree. This defines an equivalence relation on the set of neat embeddings, and we let
$$J : \Emb(B,M) \lra \JJ{\partial}{B,M}$$
be the quotient map. If $d\in \JJ{\partial}{B,M}$, we write
\[\Emb(B,M;d) := J^{-1}(d)\]
for the subspace of all \neat embeddings of $B$ into $M$ whose jet is $d$.

Let us now specialise to the case where the manifold $B$ is a connected orientable surface. If $\SS_{g,b}$ is an oriented surface of genus $g$ with $b$ boundary components, and $M$ is a manifold, then we define
\begin{align*}
\E^+(\SS_{g,b},M) &:= \Emb(\SS_{g,b},M)/\Diff^+(\SS_{g,b})\\
\Delta_{b}(M):= \Delta(\SS_{g,b},M) &:= \JJ{\partial}{\SS_{g,b},M} / \Diff^+(\SS_{g,b}).
\end{align*}
Write
$$\text{\dh} \colon \E(\SS_{g,b},M) \lra \Delta_b(M)$$
for the map induced by $J$ between quotient spaces and define, for $\delta\in \Delta_b(M)$, the subspace
\[\e{g,b}{M;\delta}:=\E^+(\SS_{g,b},M;\delta) := j^{-1}(\delta).\]
Note that if $\delta=[d]$ there is a bijection $\e{g,b}{M;\delta} \cong \Emb(\SS_{g,b},M;d)/\Diff_\partial^+(\SS_{g,b})$. We write $\dd$ for the class of $\delta_{|\pp \SS_{g,b}}$, and $\Diff(M;\delta)$ for the space of those diffeomorphisms that preserve $\delta$, i.e. the stabiliser of $\delta$ for the action of $\Diff(M)$ on $\Delta_b(M)$.


\subsection{Retractile spaces and fibrations} \label{section:retractile}

During the proof of Theorem \ref{thm:stab}, we will need to prove that certain restriction maps onto spaces of embeddings or spaces of embedded surfaces are Serre fibrations. We approach this problem in this section, following Palais and Cerf, through Lemma \ref{lemma:retractil-target}, wich says that if $E$ and $X$ are spaces with an action of a group $G$, then an equivariant map $f\colon E\to X$ between them is a locally trivial fibration (in particular a Serre fibration) provided that $X$ is ``$G$-locally retractile''. We point out that the role of the group $G$ here is merely auxiliary to prove that a map is a fibration. The rest of the section is devoted to verify that the spaces of embeddings and embedded surfaces that will concern us are $G$-locally retractile when $G$ is the diffeomorphism group of $M$.

\begin{df}(\cite{Cerf,Palais}) Let $G$ be a topological group. A $G$-space $X$ is \emph{$G$-locally retractile} if for any $x\in X$ there is a neighbourhood $U$ of $x$ and a continuous map $\xi\colon U\rightarrow G$ (called the \emph{$G$-local retraction around $x$}) such that $y = \xi(y)\cdot x$ for all $y\in U$.
\end{df}
\begin{lemma}\label{lemma:retractil-identity} If $X$ is a $G$-locally retractile locally path connected space and $G_0\subset G$ denotes the path-connected component of the identity, then $X$ is also $G_0$-locally retractile.
\end{lemma}
\begin{proof} If $\xi\colon U\rigtharrow G$ is a local retraction around $x\in X$, and $x\in U_0\subset U$ is a path-connected neighbourhood of $x$, then, since $\xi(x) = \Id$, we deduce that $\xi_{|U}$ factors through $G_0$ and defines a $G_0$-local retraction around $x$.
\end{proof}

\begin{lemma}[\cite{Cerf}]\label{lemma:retractil-target} A $G$-equivariant map $f\colon E\rightarrow X$ onto a $G$-locally retractile space is a locally trivial fibration.
\end{lemma}
\begin{proof}
For each $x\in X$ there is a neighbourhood $U$ and a $G$-local retraction $\xi$ that gives a homeomorphism
\begin{align*}
f^{-1}(x)\times U &\longrightarrow f^{-1}(U)\\
(z,y) &\longmapsto \xi(y)\cdot z.\qedhere
\end{align*}
\end{proof}
\begin{lemma}\label{lemma:retractil-source} If $f\colon X\rightarrow Y$ is a $G$-equivariant map that has local sections and $X$ is $G$-locally retractile, then $Y$ is also $G$-locally retractile. In particular, $f$ is a locally trivial fibration.
\end{lemma}
\begin{proof}
The composite of a local section for $f$ and a $G$-local retraction for $X$ gives a $G$-local retraction for $Y$.
\end{proof}

Let $P\to B$ be a principal $G$-bundle and let $F$ be a left $G$-space. There is an associated bundle of groups $P \times_G G^{ad} \to B$, and we let $\Gamma(P \times_G G^{ad} \to B)$ denote the space of its sections, equipped with the compact-open topology. This space of sections is a topological group, and it acts on the space $\Gamma(P \times_G F \to B)$ of sections of the associated $F$-bundle.

\begin{lemma} Suppose $B$ is compact and locally compact, $F$ is a $G$-space that is also $G$-locally retractile and for each neighbourhood $V'$ of the diagonal $D\subset F\times F$ there is a neighbourhood $V$ of $D$ contained in $V'$ that deformation retracts onto $D$. Then $\map(B,F)$ is $\map(B,G)$-locally retractile. Similarly, for a principal $G$-bundle $P \to B$, the space of sections $\Gamma(P \times_G F \rigtharrow B)$ is $\Gamma(P\times_G G^{ad}\to B)$-locally retractile.
\end{lemma}
\begin{proof} Here is a proof of the first part. A proof of the second part is similar to this one, working with spaces over $B$.

Let $\Phi\colon G\times F\to F\times F$ be defined as $\Phi(x,g) = (x,g\cdot x)$. If we prove that there is a neighbourhood $V$ of the diagonal $D = f(B)\times f(B)\subset F\times F$ and a global section $\varphi$ of the restriction $\Phi_{|V}\colon \Phi^{-1}(V)\to V$, then we obtain a $\map(B,G)$-local retraction $\psi$ around any point $f_0\in \map(B,F)$ as the composition
\[
A = \{f\mid (f_0\times f)(B)\subset V\} \stackrel{f\mapsto f_0\times f}{\longrightarrow} \map(B,V)\stackrel{\varphi\circ-}{\longrightarrow} \map(B,G\times F)\stackrel{\pi}{\to} \map(B,G),\]
where the last map is induced by the projection onto the second factor. To see that this is a local retraction, we first notice that $A$ is an open neighbourhood of $f_0$ because $B$ is compact and locally compact. 
Second, write $\Psi\colon \map(B,G)\times \{f_0\}\to \map(B,F)$, and 
\begin{align*}(\Psi\psi(f))(x) &= (\Psi\pi\varphi (f_0\times f))(x) = \Psi\pi\varphi(f_0(x),f(x)) \\&= (\pi\varphi(f_0(x),f(x)))\cdot f_0(x) = f(x).\end{align*}
So we only need to find $V$ and $\varphi$. Let us denote $G_{x,y} = \{g\in G\mid gx = y\}$. Let $x_0\in  F$ and let $\xi\colon U_{x_0}\to G\times \{x_0\}$ be a local retraction around $x_0$. Then there is a homeomorphism 
\[ U_{x_0}\times U_{x_0} \times \Phi^{-1}(x_0,x_0) \cong U_{x_0}\times U_{x_0}\times G_{x_0,x_0}\]
 that sends a triple $(x,y,g)$ to the triple $(x,y,\xi(y)^{-1}g\xi(x))$. Define $V'$ to be the open set $\bigcup_{x\in F} U_x\times U_x$. The restriction $\Phi_{|V'}\colon \Phi^{-1}(V')\to V'$ is locally trivial, hence a fibre bundle.

Let $V$ be a neighbourhood of $D$ that is contained in $V'$ and deformation retracts to $D$. Then in the following pullback square of fibre bundles
\[\xymatrix{
\Phi^{-1}(D)\ar[r]\ar[d] & \Phi^{-1}(V)\ar[d] \\
D\ar[r]& V
}\] 
the bottom map is a homotopy equivalence and the left vertical map has a global section that sends a pair $(x,x)$ to the pair $(x,e)$, where $e\in G$ is the identity element. Hence the right vertical map has a global section $\phi$ too.
\end{proof}

If $p\colon E\rightarrow B$ is a rank $n$ vector bundle over a manifold $B$, we denote by 
\[{\rm Vect}_k(E):=\Gamma(\Gr_k(E)\to B)\] the space of rank $k$ vector subbundles of $E$. If $P\to B$ is the principal $GL_n(\bR)$-bundle associated to $p$, let 
\[\GL(E):=\Gamma(P\times_{\GL_n(\bR)}\GL_n(\bR)^{ad}\to B).\]
be the space of sections of its adjoint bundle. If $L_\partial\in \mathrm{Vect}_k(E_{|\partial B})$ is a vector subbundle, we denote by ${\rm Vect}_k(E;L_\partial)\subset {\rm Vect}_k(E)$ the subspace of those vector subbundles whose restriction to $\partial B$ is $L_\partial$. Similarly, we denote by $\GL_\partial(E)$ the group of bundle automorphisms of $E$ that restrict to the identity over $\partial B$. The Grassmannian $\Gr_k(\bR^n)$ is $\GL_n(\bR)$-locally retractile and the inclusion of the diagonal $D\subset \Gr_k(\bR^n)\times \Gr_k(\bR^n)$ is a smooth submanifold, therefore one can find a neighbourhood $V$ of $D$ contained in $V'$ that deformation retracts onto $D$ by shrinking a tubular neighbourhood of $D$. As a consequence,

\begin{corollary} If $B$ is compact and locally compact and $E$ is a vector bundle over $B$, then the space $\mathrm{Vect}_k(E)$ is $\GL(E)$-locally retractile and the space ${\rm Vect}_k(E;L_\partial)$ is $\GL_\partial(E)$-locally retractile.  
\end{corollary}

The following lemma and its corollary are a consequence of a more general theorem proved by Cerf \cite[2.2.1~Th\'eor\`eme~5, 2.4.1~Th\'eor\`eme~5$^\prime$]{Cerf} in full generality for \manifolds with faces. Palais \cite{Palais} proved it for manifolds without corners and Lima \cite{Lima} gave later a shorter proof. 

\begin{proposition}[\cite{Cerf}]\label{prop:retractil-embeddings} If $f\colon B\rightarrow M$ is an embedding of a compact manifold $B$ into a manifold $M$, then $\Emb(B,M;[f])$ is $\Diff(M)$-locally retractile. If $d$ is a jet of an embedding of $B$ into $M$, then $\Emb(B,M;d)$ is $\Diff_\partial(M)$-locally retractile.
\end{proposition}

Applying Lemma \ref{lemma:retractil-target} to the restriction map between spaces of embeddings, we obtain

\begin{corollary}[\cite{Cerf}]\label{Cerf-fibration} If $A\subset B$ is a compact submanifold and $f\colon B\rightarrow M$ is an embedding, then the restriction map
\[\Emb(B,M;[f])\longrightarrow \Emb(A,M;[f_{|A}])\]
is a locally trivial fibration.
\end{corollary}

We will need local retractibility for the space of surfaces in a manifold. The following theorem is stated for closed submanifolds in the reference, but its proof adapts to give the following:
\begin{proposition}[\cite{MR613004,Michor}]\label{prop:binz-fisher} If $B$ and $M$ are manifolds and $d$ is a jet of $B$ in $M$, then the quotient map \[\Emb(B,M;d)\longrightarrow \Emb(B,M;d)/\Diff_\partial(B)\] is a fibre bundle. $\Diff_\partial(B)$ may be replaced by $\Diffor_\partial(B)$ if $B$ is oriented. 
\end{proposition}
By definition the space $\e{g,b}{M;\delta}$ is the quotient $\Emb(\SS_{g,b},M;d)/\Diff_\partial(\SS_{g,b})$ with $d \in \text{\dh}^{-1}(\delta)$. Hence from Lemma \ref{lemma:retractil-source} and Proposition \ref{prop:binz-fisher} we deduce that
\begin{corollary}\label{cor:retractil-embedded} The space $\e{g,b}{M;\delta}$ is $\Diff_\partial(M)$-locally retractile.
\end{corollary}

We will also need local retractibility for two more spaces.

\begin{proposition}\label{prop:retractil-resolutions}\label{cor:retractil-resolutions} Let $W\subset M$ be a submanifold and let $A\subset B$ be \manifolds.
Then the space of embeddings of pairs $\Emb((B,A),(M,W))$ is $\Diff(M;W)$-locally retractile.

More generally, the space $\Emb(B,M;q)$ is $\Diff(M;q)$-locally retractile and the space $\Emb((B,A),(M,W);q)$ is $\Diff(M;W,q)$-locally retractile.
\end{proposition}
\begin{proof} 
Let $e_0$ be one such embedding, and consider the diagram
\[\xymatrix@C=-.3cm{
& \Diff(M;W)\times \{e_0\}\ar[rr]\ar[dl]\ar[dd] && \Diff(M)\times \{e_0\}\ar[dl]\ar[dd] \\
\Diff(W)\times \{e_{0|A}\} \ar[rr]\ar[dd] && \Emb(W,M) \ar[dd]_<<<<<<<<{\circ e_{0|A}} & \\
& X \ar[dl]\ar @{} [dr] \ar[rr] && \Emb(B,M)\ar[dl] \\
\Emb(A,W) \ar[rr] && \Emb(A,M)&
}\]
where the space $X\subset \Emb(B,M)$ is the subspace of those embeddings $e$ such that $e(A)\subset W$. Hence the bottom square is a pullback square and has a natural action of $\Diff(M;W)$. The space $\Emb((B,A),(M,W))$ is a subspace of $X$ and is invariant under the action of $\Diff(M;W)$. All the vertical maps except $\Emb(W,M)\rightarrow \Emb(A,M)$ are orbit maps. All the vertical maps except possibly $h\colon \Diff(M;W)\times \{e_0\}\rightarrow X$ are locally trivial fibrations by Lemmas \ref{lemma:retractil-target} and Proposition \ref{prop:retractil-embeddings}. Moreover, $h$ is the pullback of the other three vertical maps, hence is also a locally trivial fibration, so it has local sections. As the subspace $\Emb((B,A),(M,W))\subset X$ is $\Diff(M;W)$-invariant, any $\Diff(M;W)$-local retraction around $e_0$ in $X$ gives, by restriction, a local retraction around $e_0$ in $\Emb((B,A),(M,W))$.

Recall that preserving a face constraint means preserving a sequnce of submanifolds, so the general case is obtained by iterating the above argument.
\end{proof}

\subsection{Tubular neighbourhoods} \label{section:tubular}

Let $V\subset M$ be a \neat submanifold and denote by $\N_MV = TM_{|V}/TV$ the normal bundle of $V$ in $M$. A \emph{tubular neighbourhood} of $V$ in $M$ is a \neat embedding
\[f\colon \N_M V\longrightarrow M\]
such that 
\begin{enumerate}[(i)]
\item the restriction $f_{|V}$ is the inclusion $V\subset M$, and

\item the composition
\[T V\oplus \N_M V \overset{\sim}\lra T(\N_M V)_{|V}\overset{Df}\lra TM_{|V} \overset{\text{proj.}}\lra \N_M V\]
agrees with the projection onto the second factor.
\end{enumerate}

If $W\subset M$ is a submanifold and the inclusion map defines an embedding of pairs of $(V,V\cap W)$ into $(M,W)$, then we define a \emph{tubular neighbourhood of $V$ in the pair $(M,W)$} to be a tubular neighbourhood $f \colon \N_M V\longrightarrow M$ of $V$ in $M$ such that $f_{|\N_W (V\cap W)}$ is a tubular neighbourhood of $V\cap W$.

We may compactify $\N_M V$ fibrewise by adding a sphere at infinity to each fibre, obtaining the \emph{closed normal bundle} $\Ncl_M V$ of $V$ in $M$. We denote by $S(\Ncl_MV) \subset \Ncl_MV$ the subbundle of spheres at infinity. We define a \emph{closed tubular neighbourhood} of a collared submanifold $V$ to be an embedding of $\Ncl_M V$ into $M$ whose restriction to $\N_M V$ is a tubular neighbourhood.

Note that $V$ determines a face constraint $q$ for $\Ncl_M V$ in $M$, by assigning to each connected face of $\Ncl_M V$ the minimal face to which its projection to $V$ belongs. We denote by 
\[\Tub(V,M)\subset \Emb((\Ncl_M V,V),(M,V);f)\]
 the subspace of tubular neighbourhoods. A boundary condition $q_N$ for a tubular neighbourhood of $V$ is a boundary condition for $V$ in $M$, and we denote by 
\[\Tub(V,M;q_N)\subset \Tub(V,M)\] the subspace of those tubular neighbourhoods $t$ such that $t(x,v) \subset q_N(x)$. We denote by 
\[\Tub(V,(M,W);q_N)\subset \Tub(V,M;q_N)\]
 the subspace of tubular neighbourhoods of $V$ in the pair $(M,W)$. Finally, if $q\subset q_N$ is a pair of boundary conditions, we denote by \[\Tub(V,M;(q_N,q))\] the space of tubular neighbourhoods of $V$ in $M$ such that the restriction to each face $F$ is a tubular neighbourhood in the pair $(q_N(x),q(x))$, where $x$ is any point in $F$.

The following lemma follows from the proof of \cite[Proposition 31]{godin-2007}, where it is stated for the space of all tubular neighbourhoods of compact submanifolds.

\begin{lemma}\label{lemma:tubular-contractible}
If $V$ and $W$ are compact submanifolds of $M$ and $q_N$ is a boundary condition for $V$ in $M$ such that $q_N(x)$ is a neighbourhood of $x$ in the face $[V\subset M](x)$, then the spaces $\Tub(V,M;q_N)$ and $\Tub(V,(M,W);q_N)$ are contractible. 
\end{lemma}
\begin{proof}
Let us denote by $\mathrm{Tub}(V,M;q_N)$ the space of all non-closed, collared tubular neighbourhoods of $V$ in $M$. The proof in \cite{godin-2007} has two steps. In the first, a tubular neighbourhood $f$ is fixed and a weak deformation retraction $H$ of $\mathrm{Tub}(V,M;q_N)$ is constructed into the subspace $T_f$ of all tubular neighbourhoods whose image is contained in $\Im f$. The second step provides a contraction of $T_f$ to the point $\{f\}$. It is easy to see that if $f$ is a closed, collared tubular neighbourhood of a pair, then both homotopies define homotopies for $\Tub(V,(M,W);q_N)$, and the argument there applies verbatim.
\end{proof}

\subsection{\boldmath The space of thickened embeddings} \label{ss:thickened-embeddings}

Using the notion of tubular neighbourhood we have just defined, we introduce now a space of embeddings equipped with a choice of tubular neighbourhood of their image. We will show that the space can be topologised as a quotient of the space of thickened embeddings. We will consider more generally that the image of the embeddings are endowed with a field of subplanes of the tangent bundle of the background manifold.

\subsubsection{The set $\TubEmb_{k,C}(B,M;q,q_N,q_C)$} Let $B$ and $M$ be manifolds and let $C\subset B$ be a submanifold. Let $q$ and $q_N$ be face constraints for $B$ in $M$ and let $q_C$ be a face constraint for $C$. The set $\TubEmb_{k,C}(B,M;q,q_N,q_C)$ is defined as the set of triples $(e,t,L)$, where
\begin{enumerate}
\item $e\in \Emb(B,M;q)$ is an embedding;
\item $t\in \Tub(e(B),M;(q_N,q))$;
\item $L \in \Gamma(\Gr_k(\N_M e(B))_{|e(C)}\to e(C);\N_{q_C(\pp C)} e(\pp C))$.
\end{enumerate}
If $k= 0$, the subset $C$ is irrelevant and we will write $\TubEmb(B,W;q,q_N)$ for $\TubEmb_{0,C}(B,M;q,q_N)$. Note that, if $\pp C \neq \emptyset$, then the last condition forces $k = \dim q_C(\pp C) - \dim \pp C$.

\subsubsection{The action}The set $\TubEmb_{k,C}(B,M;q,q_N,q_C)$ has a natural action of the discretisation of the group $\Diff(M;q,q_N,q_C)$ given as follows:
 If $g$ is a diffeomorphism of $M$ and $t$ is a tubular neighbourhood as above, then $g$ induces isomorphisms $TM_{|e(B)}\rightarrow TM_{|ge(B)}$ and $Te(B)\rightarrow Tge(B)$, hence an isomorphism $g_*\colon \N_M e(B)\rightarrow \N_M ge(B)$. We define $g(t)$ as the composite
\[\xymatrix{
\N_Mge(B)\ar[r]^-{g_*^{-1}}& \N_M e(B) \ar[r]^-t& M \ar[r]^-g& M. 
}\]
We define $g(L)$ as $g_*\circ L\circ g^{-1}$.

In what follows we omit the face constraints from the notation, for the sake of clarity.

\subsubsection{The topology} 
If $V \to B$ is a vector bundle of dimension $\mathrm{dim}(M)-\mathrm{dim}(B)$, then there is a map
\[\phi_V\colon \Emb(V,M)\times \Gamma(\Gr_k(V_{\vert C})\to C) \lra \TubEmb_{k,C}(B,M)\]
given by sending a pair $(f,s)$ to the triple
$$(B \overset{f_{|B}}\to M, N_M f(B) \overset{f_*^{-1}}\to V \overset{f}\to M,  f(C) \overset{f_{\vert C}^{-1}}\to C \overset{s}\to \Gr_k(V_{\vert C}) \overset{f_*}\to \Gr_k(N_M f(C))),$$
where $f_* : V \to N_V B \to N_M f(B)$ is the map induced by $f$ on normal bundles.

If $\varphi : V \to V$ is a bundle map, then $\phi_V(f \circ \varphi, \varphi(s)) = \phi_V(f,s)$, and furthermore if $\phi_V(f,s) = \phi_V(h,r)$, then 
\begin{enumerate}[(i)]
\item $f_{|B} = h_{|B}$,
\item $ff_*^{-1} = hh_*^{-1}$, hence $f=h \circ (h_*^{-1}f_*)$ and
\item $f_*sf_{|B}^{-1} = h_*rh_{|B}^{-1}$, hence $(h_*^{-1}f_*)(s) = r$.
\end{enumerate}
(The first condition follows from the second, and the third follows from the first two conditions.) Thus that part of the set $\TubEmb_{k,C}(B,M)$ which is the image of $\phi_V$ may be described as the quotient of $\Emb(V,M)\times \Gamma(\Gr_k(V_{\vert C})\to C)$ by the action of $\Bun(V,V)$ which we have described. We thus give $\TubEmb_{k,C}(B,M)$ the finest topology making the maps $\phi_V$ continuous for every choice of $V$.

\begin{lemma}\label{lemma:111}\mbox{}
\begin{enumerate}[(i)]
\item The map $\phi_V$ is a (locally trivial) principal $\Bun(V,V)$-bundle onto those components which it hits.

\item The action of $\Diff(M)$ on $\TubEmb_{k,C}(B,M)$ is continuous.
\end{enumerate}
\end{lemma}
\begin{proof}
For part (i), choose a $\Diff(M)$-local retraction of $\Emb(B,M)$ around an embedding $e_0$,
\[\xi\colon U\lra \Diff(M),\]
then we obtain a local trivialisation of $\phi_V$ as follows: Let $\hat{U}\subset \TubEmb_{k,C}(B,M)$ be the preimage of $U$ under the forgetful map to $\Emb(B,M)$. Let $(f_0,s_0)\in \phi_V^{-1}(e_0,t_0,L_0)$. Then define $\sigma : \phi_V^{-1}(\hat{U}) \to \hat{U}\times \Bun(V, V)$ by $\phi_V\times \psi$, with
\[\psi(f,s)\colon V \overset{f_*}{\lra} \N_M f(B) \overset{\xi(f_{|B})_*^{-1}}{\lra} \N_Mf_0(B) \overset{f_{0*}^{-1}}{\lra} V.\]
Define $\tau : \hat{U}\times \Bun(V, V) \to \phi_V^{-1}(\hat{U})$ by sending $((e,t,L), \alpha)$ to the pair $(f,s)$ given by
\begin{align*}
f\colon V\overset{\alpha}{\lra}  V \overset{f_{0*}}{\lra} \N_{e_0(B)} M \overset{\xi(e)_*}{\lra} \N_{e(B)} M \overset{t}{\lra} M \quad\quad s = f_*^{-1}Lf.
\end{align*}
The maps $\sigma$ and $\tau$ are both continuous, and if $\tau((e,t,L), \alpha)=(f,s)$ then $f_* = \xi(e)_*f_{0*}\alpha$, from what it follows that $\sigma$ and $\tau$ are mutual inverses.

For part (ii), note that there is an action of $\Diff(M)$ on $\Emb(V,M)\times \Gamma(\Gr_k(V_{\vert C})\to C)$, by postcomposition on the first factor. The map $\phi_V$ is equivariant with respect to these actions, so there is a commutative square
\begin{equation*}
\xymatrix{
\Diff(M) \times \Emb(V,M)\times \Gamma(\Gr_k(V_{\vert C})\to C) \ar[d]^-{\mathrm{Id} \times \phi_V}\ar[r] & \Emb(V,M)\times \Gamma(\Gr_k(V_{\vert C})\to C) \ar[d]^-{\phi_V}\\
\Diff(M) \times \TubEmb_{k,C}(B,M) \ar[r]& \TubEmb_{k,C}(B,M),
}
\end{equation*}
where the horizontal maps are given by the action. The top map is continuous, as $\Diff(M)$ acts on $\Emb(V,M)$ continuously, and the vertical maps are open since $\phi_V$ is a locally trivial fibre bundle. It follows that the action of $\Diff(M)$ on the image of $\phi_V$ is continuous, and this holds for all $V$, so the action of $\Diff(M)$ on $\TubEmb_{k,C}(B,M)$ is continuous.
\end{proof}

This discussion also applies to the subspace $\TubEmb_{k,C}((B,A),(M,W);q,q_N,q_C)$ of tubular neighbourhoods of embeddings of a pair $(B,A)$ in a pair $(M,W)$, and this latter space has a natural action of $\Diff(M;W,q,q_N,q_C)$.  

\begin{proposition}\label{prop:retractil-fibrations} The space $\TubEmb_{k,C}(B,M;q,q_N,q_C)$ is $\Diff(M;q,q_N,q_C)$-locally retractile. 
\end{proposition}

\begin{proof}
Since the source of $\phi_V$ is $\Diff(M)$-locally retractile by Proposition \ref{prop:retractil-embeddings} and $\phi_V$ has local sections over its image, it follows that its image is also $\Diff(M)$-locally retractile by Lemma \ref{lemma:retractil-source}. 
\end{proof}

Using the proof of Proposition \ref{prop:retractil-resolutions}, and the previous lemma we obtain

\begin{corollary}\label{cor:retractil-fibrations} The spaces $\TubEmb_{k,C}((B,A),(M,W);q,q_N)$ are $\Diff(M;q,q_N)$-locally retractile. 
\end{corollary}

\section{Maps between spaces of subsurfaces}\label{ss:gluingmaps}

For a collared \manifold $M$, we will use two kinds of maps between spaces of surfaces of the form $\e{g,b}{M;\delta}$. 
The first map glues a collar $\pp M\times I$ to $M$ and a surface $P\subset \pp M\times I$ to the surfaces in $\e{g,b}{M;\delta}$. For the second map, we remove a submanifold $u'\subset M$ from $M$. If a surface $u''\subset u'$ is given, we may glue $u''$ to the surfaces in $M\setminus u'$, obtaining a map from $\e{h,c}{M\setminus u';\delta'}$ to $\e{g,b}{M;\delta}$, where $h,c,\delta'$ depend on the surfaces $u''$. In the following paragraphs these constructions are explained in detail.

The \manifold $M_1$ is defined as the union of the \manifold $\pp M\times [0,1]$ and the \manifold $M$ along $\pp M\times \{0\}$ using the collar of $M$. The collar of $M$ gives a canonical collar both to $\pp M_1 := \pp M\times \{1\}$ and to $\pp(\pp M\times I) := \pp M\times \{0,1\}$. The boundary condition $\delta$ also gives boundary conditions 
\begin{align*}
\y{\d} &= \text{\dh}(\dd\times I)\in \Delta_2(\pp M\times I)\\
\z{\d} &= \text{\dh}(W\cup(\dd\times I))\in \Delta_2(M_1)
\end{align*}
where $W$ is any surface in $\e{g,b}{M;\delta}$. Let $\SS'$ be another collared surface with $\partial \SS' = \partial(\pp\SS\times I)$. For each $P\in \E^+(\SS',\pp M\times [0,1];\y{\d})$, there is a continuous map
\begin{align*}
-\cup P\colon \E^+(\SS,M;\delta)&\lra \E^+(\SS\cup\SS',M_1;\z{\delta}) 
\end{align*} 
that sends a submanifold $W$ to the union $W\cup P$. These are \emph{maps of type $I$}.

Our main theorem will prove that the maps of type I have certain homology stability properties. If $P$ and $P'$ are isotopic, then the induced maps $-\cup P$ and $-\cup P'$ will be isotopic. The following follows easily:

\begin{lemma} Let $M$ be of dimension at least $5$. If $P\subset \pp M\times [0,1]$ then there exist $P_1,\ldots,P_k$, such that:
\begin{enumerate}
\item $P_1\cup \ldots\cup P_k \simeq k\cdot P\subset \pp M\times [0,k]$ (the product means taking a dilation in the second coordinate).
\item $P_i$ a submanifold all of whose components but one are of the form $\delta^0\times [i,i+1]\subset \pp M\times [i,i+1]$,
\item the projection to the second coordinate of $\pp M\times [i,i+1]$ restricted to the remaining component is a Morse function with at most one critical point.
\item if there is a critical point in this remaining component, then there exists a pair of points $q_0,q_1 \in \delta^0$ such that $q_j\times [i,i+1]\subset P$ and $\{q_0\times \{i\}, q_0\times \{i+1\}, q_1\times \{i\}, q_1\times \{i+1\}\}$ hit all the components of the boundary of the remaining component of $P$.
\end{enumerate}
\end{lemma}

\begin{df}\label{df:relevant} If $P$ is as in the previous lemma, it is called a basic cobordism, and the remaining component will be called the \emph{relevant} cobordism and the boundary of it will be called the \emph{relevant} boundary condition.
\end{df}

If $\SS$ is a compact connected oriented surface of genus $g$ and $b$ boundary components, and $P$ is a basic cobordism, we will denote the map $-\cup P$ by
\begin{align*}
\alpha_{g,b}(M;\delta,\bar{\delta})\colon &\e{g,b}{M;\delta}\lra \e{g+1,b-1}{M;\bar{\delta}}\\  
\beta_{g,b}(M;\delta,\bar{\delta})\colon &\e{g,b}{M;\delta}\lra \e{g,b+1}{M;\bar{\delta}}\\
\gamma_{g,b}(M;\delta,\bar{\delta})\colon &\e{g,b}{M;\delta}\lra \e{g,b-1}{M;\bar{\delta}}
\end{align*}
 depending on the genus and the number of boundary components of the surfaces in the target. Note that, if $\SS$ has no corners, then $P$ will be a disjoint union of connected surfaces, one of them a pair of pants or a disc, and the rest diffeomorphic to cylinders. Now we define the second type of gluing map. Let $s\subset \SS'$ be either empty or a closed tubular neighbourhood of an arc or a point in $\SS$. The complement $\SS\setminus s$ is again a collared \manifold, but if $s$ is an arc and $M$ is a manifold with boundary, its complement is no longer a manifold with boundary. This justifies working with manifolds with corners.

Consider a tuple $u= (u',u'',u''')$ consisting on a \neat embedding $u'''\in \Emb(B,M;q)$, a closed tubular neighbourhood $u'$ of $u'''$ in $M$, and a (possibly empty) surface $u''\in \E(s,u';\delta[u])$ such that $\delta[u]\cap \delta = \delta[u]^0$. Then $\cl (M\setminus u')$ is a collared \manifold with $\pp \cl (M\setminus u') = \cl(\pp M\setminus \pp u')$. The boundary conditions $\delta$ and $\delta[u]$ give rise to a boundary condition 
\[\delta(u) = \text{\dh}(\cl(W\setminus u''))\in \Delta_2(\cl M\setminus u')\]
where $W\in\e{g,b}{M;\delta}$ is any surface that contains $u''$.

The triple $u=(u',u'',u''')$ defines a map
\[\E(\cl \SS\setminus s,\cl M\setminus u';\delta(u))\longrightarrow \E(\SS,M;\d)\]
that sends a submanifold $W$ to the union $W\cup u''$. These are \emph{maps of type II}.
\begin{notation} First, since the map defined above is completely determined by the tuple $u$, we will use the notation $\stminus{M}{u}$ for $\cl {M}\setminus{u'}$. Second, for maps of type I, we denote with a tilde $~\y{}~$ the objects that we glue to the space of surfaces and with a dash $~\z{\phantom{u}}~$ the objects obtained by removing or gluing surfaces to $\e{g,b}{M;\d}$. For maps of type II, we denote with brackets $[~]$ the objects that we remove from the space of surfaces and with parentheses $(~)$ the result of removing those objects. In addition, we denote with $'''$ the submanifold, with $'$ the tubular neighbourhood and with $''$ the surface in the tubular neighbourhood. We will be consistent with these notations. Third, for maps of type I, the triple $u$ defines triples $\y{u}$ and $\z{u}$ in the manifolds $\pp M\times I$ and $M_1$ given by $\y{u} = \pp u\times I$ and $\z{u} = u\cup \y{u}$. If we assume in addition that $P\cap \y{u}' = \y{u}''$ and that $(\pp s)\times I \subset \SS'$, then in the diagram 
\[
\xymatrix{
\E^+(\SS(s),M(u);\delta(u)) \ar[d]\ar@{-->}[rr] && \E^+((\SS\cup\SS')(\z{s})
,M_1(\z{u});\z{\delta}(\z{u}))\ar[d] \\
\E^+(\SS,M;\delta) \ar[rr] && \E^+(\SS\cup \SS',M_1;\z{\delta})
}
\]
we may construct the upper horizontal arrow as $-\cup P\setminus \y{u}''$. As before, we will use the notation $P(\y{u}) = P\setminus \y{u}''$. 
\end{notation}

\section{Homotopy resolutions}

A \emph{semi-simplicial space}, also called \emph{$\Delta$-space}, is a contravariant functor
\[X_\bullet\colon \Delta^{\op}_{\inj}\longrightarrow \Top\]
from the category $\Delta_\inj$ whose objects are non-empty finite ordinals and whose morphisms are injective order-preserving inclusions to the category $\Top$ of topological spaces. The image of the ordinal $n$ is written $X_n$ and we denote by $\partial_j\colon X_{n+1}\rightarrow X_n$ the image of the inclusion $ n = \{0,1,\ldots,n-1\}\hookrightarrow \{0,1,\ldots,n\} = n+1$ that misses the element $j\in \{0,1,\ldots,n\}$. These are called \emph{face maps} and the whole structure of $X_\bullet$ is determined by specifying the spaces $X_n$ for each $n$ together with the face maps in each level.

A \emph{semi-simplicial space augmented over a topological space $X$} is a semi-simplicial space $X_\bullet$ together with a continuous map $\epsilon\colon X_0\rightarrow X$ (the \emph{augmentation}) such that $\epsilon \partial_0 = \epsilon \partial_1\colon X_1\rightarrow X$. Alternatively, an \emph{augmented semi-simplicial space} is a contravariant functor
\[X_\bullet\colon \Delta^{\op}_{\inj,0}\longrightarrow \Top\]
from the category $\Delta_{\inj,0}$ whose objects are (possibly empty) ordinals and whose morphisms are injective order-preserving inclusions to the category $\Top$. As above, we denote by $\partial_j\colon X_n\rightarrow X_{n+1}$ the image of the inclusion that misses $j$ and we denote by $\epsilon$ the image of the unique inclusion $\emptyset\rightarrow 0$. We denote by $\epsilon_i\colon X_i\to X$ the unique composition of face maps and the augmentation map.

A \emph{semi-simplicial map} between (augmented) semi-simplicial spaces is a natural transformation of functors. If $\epsilon_\bullet\colon X_\bullet \to X$ and $\epsilon'_\bullet\colon Y_\bullet \to Y$ are augmented semi-simplicial spaces, a semi-simplicial map $f_\bullet$ is equivalent to a sequence of maps $f_n \colon X_n \to Y_n$ such that $d_i \circ f_n = f_{n-1} \circ d_i$ for all $i$ and all $n$, together with a map $f \colon X \to Y$ such that $\epsilon'\circ f_0 = f\circ \epsilon$.

There is a realisation functor \cite{SegalCatCohom}
\[|\,\cdot\,|\colon \text{Semi-simplicial~spaces}\longrightarrow \Top\]
and we say that an augmented semi-simplicial space $X_\bullet$ over a space $X$ is a \emph{resolution} of $X$ if the map induced by the augmentation
\[|\epsilon_\bullet|\colon |X_\bullet|\longrightarrow X\]
is a weak homotopy equivalence. It is an \emph{$n$-resolution} if the map induced by the augmentation is $n$-connected (i.e., the relative homotopy groups $\pi_i(X,|X_\bullet|)$ vanish when $i\leq n$).

We will use the spectral sequences given by the skeletal filtration associated with augmented semi-simplicial spaces as they appear in \cite{R-WResolution2015}. For each augmented semi-simplicial space $\epsilon_\bullet\colon X_\bullet \to X$ there is a spectral sequence defined for $t\geq 0$ and $s\geq -1$:
\[E^1_{s,t} = H_t(X_s)\Longrightarrow H_{s+t+1}(X,|X_\bullet|),\]
and for each map between augmented semi-simplicial spaces $f_\bullet\colon X_\bullet\rightarrow Y_\bullet$ there is a spectral sequence defined for $t\geq 0$ and $s\geq -1$:
\[E^1_{s,t} = H_t(Y_s,X_s)\Longrightarrow H_{s+t+1}((|\epsilon_{\bullet}^Y|),(|\epsilon_{\bullet}^X|)),\]
where, for a continuous map $f\colon A\rightarrow B$, we denote by $M_f$ the mapping cylinder of $f$ and by $(f)$ the pair $(M_f,A)$. 

The following criteria will be widely used throughout the paper.

\begin{criterion}[{\cite[Lemma~2.1]{R-WResolution}}]\label{criterion1} Let $\epsilon_\bullet\colon X_\bullet \to X$ be an augmented semi-simplicial space
. If each $\epsilon_i$ is a fibration and $\Fib_x(\epsilon_i)$ denotes its fibre at $x$, then the realisation of the semi-simplicial space $\Fib_x(\epsilon_\bullet)$ is weakly homotopy equivalent to the homotopy fibre of $\vert\epsilon_\bullet\vert$ at $x$.
\end{criterion}

An {\em augmented topological flag complex} is an augmented semi-simplicial space $\epsilon_\bullet\colon X_\bullet \to X$ such that
\begin{enumerate}
    \item the product map $X_i \rightarrow X_0\times_{X}\cdots\times_{X} X_0$ is an open embedding;
    \item a tuple $(x_0,\ldots,x_i)$ is in $X_i$ if and only if for each $0 \leq j< k\leq i$ the pair $(x_j,x_k)\in X_0\times_{X}X_0$ is in $X_1$.
\end{enumerate}

\begin{criterion}[{\cite[Theorem~ 6.2]{GR-W2}}]\label{cor18}\label{criterion2}
Let $\epsilon_\bullet\colon X_\bullet \rightarrow X$ be an augmented topological flag complex. Suppose that
\begin{enumerate}
\item $X_0 \rightarrow X$ has local sections, that is, $\epsilon$ is surjective and for each $x_0\in X_0$ such that $\epsilon(x_0) = x$ there is a neighbourhood $U$ of $x$ and a map $s\colon U\rightarrow X_0$ such that $\epsilon (s(y)) = y$ and $s(x) = x_{0}$;
\item given any finite collection $\{x_1,\ldots,x_n\} \subset X_0$ in a single fibre of $\epsilon$ over some $x\in X$, there is an $x_\infty$ in that fibre such that each $(x_i, x_\infty)$ is a $1$-simplex.
\end{enumerate}
Then $\vert\epsilon_\bullet\vert: |X_\bullet| \rightarrow X$ is a weak homotopy equivalence.
\end{criterion}

\begin{remark}\label{rem:criterion2}
For the second condition, we could also ask that there is an $x_0$ such that each $(x_0, x_i)$ is a 1-simplex, and the conclusion still holds.
\end{remark}

\subsection{A criterion for homological stability}

\begin{criterion}\label{criterion3} 
Let $f_\bullet : X_\bullet \to Y_\bullet$ be a map of augmented semi-simplicial spaces such that $\vert\epsilon_\bullet^X\vert : \vert X_\bullet\vert \to X$ is $(l-1)$-connected and $\vert\epsilon_\bullet^Y\vert : \vert Y_\bullet\vert \to Y$ is $l$-connected. Suppose there is a sequence of path connected based spaces $(B_i, b_i)$ and maps $p_i \colon Y_i \to B_i$, and form the map
$$g_i \colon \mathrm{hofib}_{b_i}(p_i \circ f_i) \lra \mathrm{hofib}_{b_i}(p_i)$$
induced by composition with $f_i$. Suppose that there is a $k \leq l+1$ such that
$$\text{$H_q(g_i) = 0$\, when $q+i\leq k$, except if $(q,i)=(k,0)$.}$$

Then the map induced in homology by the composition of the inclusion of the fibre and the augmentation map
\[H_q(g_0)\longrightarrow H_q(f_0)\overset{\epsilon}\longrightarrow H_q(f)\]
is an epimorphism in degrees $q\leq k$. 

If in addition $H_k(g_0)=0$, then $H_q(f) = 0$ in degrees $q\leq k$.
\end{criterion}

\begin{proof}
We have a homotopy fibre sequence of pairs $(g_i) \to (f_i) \to B_i$, and so a relative Serre spectral sequence
\[
\widetilde{E}^2_{p,q} = H_p(B_i;\mathcal{H}_q(g_i))  \Longrightarrow H_{p+q}(f_i),
\]
where ${\mathcal H}_q(g_i)$ denotes homology with twisted coefficients. Since $H_q(g_i) = 0$ for all $q\leq k-i$ except $(q,i)=(k,0)$, we have that $H_q(f_i) = 0$ for all $q+i\leq k$ except $(q,i)=(k,0)$. Moreover, if $i=0$, all differentials with target or source $H_0(B;\mathcal{H}_q(g_0))$ for $q\leq k$ are trivial, and these are the only possibly non-trivial groups with total degree $p+q\leq k$. Hence 
\[H_q(g_0){\longrightarrow}H_0(B;\mathcal{H}_q(g_0))\longrightarrow H_q(f_0)\]
is the composition of two epimorphisms if $q\leq k$.

The first page of the spectral sequence for the resolution $(f_\bullet)\rightarrow (f)$ is
\[
E_{p,q}^1 = H_q(f_p), \quad p \geq -1,
\]
and it converges to zero in total degrees $p+q\leq l$. Since $H_q(f_p) = 0$ for all $p+q\leq k$ except $(p,q) = (0,k)$, any differential with target $E^r_{-1,q}$ for $q\leq k$ and $r\geq 2$,
\begin{equation*}
d_r: E^r_{r-1,q-r+1} \longrightarrow  E^r_{-1,q},
\end{equation*}
has source a quotient of $H_{q-r+1}(f_{r-1})$, which is trivial. As $k-1\leq l$, and the spectral sequence converges to zero in total degrees $p+q\leq l$, we have that for each $q\leq k$ there is an $r\geq 1$ such that $E^r_{-1,q} = 0$, hence the homomorphisms induced by the augmentation map $d_1\colon H_q(f_0)\rightarrow H_q(f)$ are epimorphisms in degrees $q\leq k$.

For the second part, note that in that case all epimorphisms $H_q(g_0)\rightarrow H_q(f_0)\rightarrow H_q(f)$ have trivial source when $q\leq k$, hence the target is also trivial in those degrees.
\end{proof}

There is one final concept that we will use rather often to describe the homotopy type of the fibre of a fibration. We say that a pair of maps
$$A \overset{g}\lra B \overset{f}\lra C$$
is a \emph{homotopy fibre sequence} if $f \circ g$ is homotopic to the constant map to a point $c \in C$, and the induced map $A \to \mathrm{hofib}_{c}(f)$ is a weak homotopy equivalence. For our purposes, such data can be treated as if $f$ were a fibration and $f$ were the inclusion of the fibre over $c$.

\section{Resolutions of spaces of surfaces}\label{section:resolutions}
In this section we construct two $(g-1)$-resolutions of the space $\e{g,b}{M;\delta}$, where $M$ is a collared \manifold with non-empty boundary, and $\delta\in \Delta_{2}^+(M)$ is a non-empty boundary condition (in particular $b\geq 1$). We will also characterize the space of $i$-simplices of each resolution as the total space of a certain homotopy fibration. Afterwards we will explain how these $(g-1)$-resolutions give rise to a $(g-1)$-resolution or a $g$-resolution of the stabilisation maps (the connectivity of each resolution depends on the stabilisation map), and how to characterize their spaces of $i$-simplices. 

\subsection{Resolutions of a single surface}\label{ss-resolutions1} In the proof of \cite[Proposition~5.3]{R-WResolution2015} the following semi-simplicial space was introduced: If $W$ is a compact connected oriented surface with non-empty boundary, and $\ell_0,\ell_1$ are embedded intervals in $\partial W$, the semi-simplicial space $O(W; k_0,k_1)_\bullet$ (denoted $A'(X)$ in the cited reference) is defined as follows: An $i$-simplex is a tuple $(\vv_0,\ldots,\vv_i)$ of pairwise disjoint embeddings of the interval $[0,1]$ in $W$ such that
\begin{enumerate}
\item $\vv_j(0)\in k_0$ and $\vv_j(1)\in k_1$;
\item the complement of $\vv_0\cup\ldots\cup \vv_i$ in $W$ is connected;
\item the ordering at the endpoints of the arcs is $(\vv_0(0),\ldots,\vv_i(0))$ in $k_0$ and $(\vv_i(1),\ldots,\vv_0(1))$ in $k_1$, where $k_0$ and $k_1$ are ordered according to the orientation of $\delta$.
\end{enumerate}
The $j$th face map forgets the $j$th arc. In order to simplify the notation we will write $[i]$ for $\{0,\ldots,i\}$. The set of $i$-simplices is topologized as a union of components of the space $\Emb(I\times [i],W;q)$ with $q(x) = k_j$ if $x\in \{j\}\times [i]$ and $q(x) = W$ otherwise. Let $\pi_0O(W,k_0,k_1)_\bullet$ be the semi-simplicial set obtained by applying levelwise the functor $\pi_0$.
\begin{proposition}[{\cite[Theorem 5.3]{R-WResolution2015}}]\label{proposition:RW1-4.1} The realisation $|\pi_0 O(W;k_0,k_1)_\bullet|$ is $(g-2)$-connected, where $g$ is the genus of $W$.
\end{proposition}
This proposition is the hardest step in the proof of homological stability for diffeomorphism groups of surfaces, and the change in the definition of this complex of curves yielded various improvements \cite{Ivanov, Boldsen, R-WResolution} of the original complex of Harer \cite{H}. A detailed proof of the above proposition may also be found in \cite{Wahl-survey}.

\subsection{Resolutions of spaces of surfaces}

We will use the following notation:
\begin{enumerate}
\item $D^+_{r} = \{(x_1,x_2)\in \bR^2\mid x_2\geq 0, \|(x_1,x_2)\|\leq r\}$;
\item $\pp D^+_{r} = \{(x_1,x_2)\in D^+_{r}\mid x_2 =  0\}$;
\item $\ppp D^+_{r} = \{(x_1,x_2)\in D^+_{r}\mid \|(x_1,x_2)\| = r\}$;
\end{enumerate}
We sometimes will identify $\ppp D^+_\frac{1}{2}$ with $[0,1]$ using the map $(x,y)\mapsto x+\frac{1}{2}$
\begin{df}\label{df:ssresolution} Let $\ell\subset \pp M$ be an open ball that intersects $\dd$ in two intervals denoted $\ell_0$ and $\ell_1$. There is a semi-simplicial space $\OO{g,b}{M;\delta, \ell}{\bullet}$ (for which we write $\OO{g,b}{M;\delta}{\bullet}$ for brevity) whose $i$-simplices are tuples $(W,u_0,\ldots,u_i)$ with $u_j = (u_j',u_j'',u_j''')$, where
\begin{enumerate}
	\item $W\in \e{g,b}{M;\delta}$;
	\item $u_j'''\colon (D^+_1,\ppp D^+_\frac{1}{2})\to (M,W)$ is an embedding of pairs, and the restriction to $\pp D^+_{1}$ has image in $\ell$;
	\item\label{c3} $(u'_j,u''_j)$ is a closed tubular neighbourhood of $u'''_j$ in the pair $(M,W)$;
\end{enumerate}
are such that
\begin{enumerate}
\setcounter{enumi}{3}
	\item if $\vv_j$ denotes the restriction of $u'''_j$ to $\ppp D^+_\frac{1}{2}$, then $(\vv_0,\ldots,\vv_i)$ is an $i$-simplex in $O(W, \ell_0, \ell_1)$.
	\item\label{c5} $u'_0,\ldots,u'_i$ are pairwise disjoint.
\end{enumerate}
The $j$th face map forgets $u_j$, that is, it sends an $i$-simplex $(W,u_0,\ldots,u_i)$ to the $(i-1)$-simplex $(W,u_0,\ldots,\hat{u}_j,\ldots,u_i)$. There is an augmentation map $\epsilon_\bullet$ to the space $\e{g,b}{M;\delta}$ that forgets everything but $W$. This defines a semi-simplicial set, and we topologise the set of $i$-simplices as a subspace of \[\e{g,b}{M;\delta}\times \TubEmb(I\times [i],M).\] 
\end{df}
If we want to emphasise that $\ell$ intersects a single component of $\delta$, we denote the semi-simplicial space by $\h{g,b}{M;\delta,\ell}{\bullet}$; if we want to emphasise that $\ell$ intersects two disctinct components of $\delta$, we denote the semi-simplicial space by $\b{g,b}{M;\delta,\ell}{\bullet}$. 

We denote by $\ww_j$ the image of the restriction of $u''_j$ to $\vv_j$, which is a piece of surface.

\begin{proposition}\label{resolution-arcs} $\OO{g,b}{M;\delta,\ell}{\bullet}$ is a $(g-1)$-resolution of $\e{g,b}{M;\delta}$.
\end{proposition}
\begin{proof} In order to find the connectivity of the homotopy fibre of $\epsilon_\bullet$, we use Criterion \ref{criterion1} to assure that the semi-simplicial fibre $\Fib_W(\epsilon_\bullet)$ of $\epsilon_\bullet$ over a surface $W$ is homotopy equivalent to the homotopy fibre of $|\epsilon_\bullet|$: the space $\e{g,b}{M;\delta}$ is $\Diff_\partial(M)$-locally retractile by Corollary \ref{cor:retractil-embedded}, and, as the group $\Diff(M;\delta,\ell)$ also acts on this space, any local retraction for $\Diff_\partial(M)$ gives also a local retraction for $\Diff(M;\delta,\ell)$. In addition, the augmentation maps $\epsilon_i$ are $\Diff(M;\delta,\ell)$-equivariant for all $i$. Therefore, by Lemma \ref{lemma:retractil-target}, they are locally trivial fibrations.

For this proof only, we define $\OO{g,b}{M;\delta,\ell}{\bullet}^{\star}$ to be the semi-simplicial space that results from forgetting the data of the tubular neighbourhoods in the definition of $\OO{g,b}{M;\delta,\ell}{\bullet}$. That is, we take spaces of tuples $(W,u_0''',\ldots,u_i''')$ as above, remove condition \ref{c3} and replace condition \ref{c5} by requiring that $u_0''',\ldots,u_i'''$ must be pairwise disjoint. Again, it is augmented over $\e{g,b}{M;\delta}$, and we denote the augmentation by $\epsilon^{\star}$.

The $i$-simplices of $\Fib_W(\epsilon_\bullet)$ are tuples $(u_0,\ldots,u_i)$ with $u_j = (u'_j,u''_j,u'''_j)$ where $u'''_j$ are embeddings of a half disc in $W$ and $(u'_j,u''_j)$ are pairwise disjoint closed tubular neighbourhoods of $u'''_j$ in the pair $(W,M)$. Forgetting the closed tubular neighbourhoods gives a levelwise $\Diff(M;W,\delta,\ell)$-equivariant semi-simplicial map
\[r_\bullet\colon \Fib_W(\epsilon_\bullet)\longrightarrow \Fib_W(\epsilon_\bullet^{\star}),\]
and the space of $i$-simplices of $O(W;\ell_0,\ell_1)_\bullet$ is $\Diff(W)$-locally retractile, and also $\Diff(M;W,\delta,\ell)$-locally retractile by Corollary \ref{cor:retractil-resolutions}. Therefore, by Lemma \ref{lemma:retractil-target}, $r_\bullet$ is a levelwise locally trivial fibration. The fibre of $r_\bullet$ over an $i$-simplex is a space of closed tubular neighbourhoods of discs in the pair $(M,W)$, which is contractible by Lemma \ref{lemma:tubular-contractible}, so $r_\bullet$ is a homotopy equivalence. In the next paragraph we will show that $\|\Fib_W(\epsilon_\bullet^{\star})\|$ is $(g-2)$-connected, finishing the proof of the proposition.

Taking the discrete topology on the semi-simplicial space $\Fib_W(\epsilon_\bullet^{\star})$, we get a semi-simplicial \emph{set} $\Fib_W(\epsilon_\bullet^{\star})^{\rm disc}$, and a pair of maps:
\[\Fib_W(\epsilon_\bullet^{\star})\longleftarrow \Fib_W(\epsilon_\bullet^{\star})^{\rm disc}\overset{g_\bullet}{\longrightarrow} \pi_0O(W;\ell_0,\ell_1)_\bullet.\]
The first map is the identity on points, and the second map sends each tuple $u_j$ to its restriction $\vv_j$ to $\ppp D^+_{\frac{1}{2}}$, and then takes isotopy classes. Now we use the techniques in \cite{GR-W} as summarised in \cite[Theorem~A.7]{Kupers:unlinked}. The theorem says that if 
\begin{enumerate}
\item $\pi_0 O(W;\ell_0,\ell_1)_\bullet$ is weakly Cohen-Macaulay of dimension $g-1$ (i.e., it is $(g-2)$-connected and the link of each $i$-simplex is $(g-i-3)$-connected).
\item $\Fib_W(\epsilon_\bullet^{\star})$ is a Hausdorff ordered flag space.
\item The map $\vert g_\bullet \vert \colon \vert\Fib_W(\epsilon_\bullet^{\star})^{\rm disc} \vert \to \vert\pi_0 O(W;\ell_0,\ell_1)_\bullet \vert$ is simplexwise injective.
\item If $u_1''',\ldots,u_i'''$ is a collection of $0$-simplices in $\Fib_W(\epsilon_\bullet^{\star})$, and $[\vv]\in \pi_0O(W;\ell_0,\ell_1)_0$ so that $([\vv],[\vv_j'''])$ is an $1$-simplex for all $j$, then there exists a $u_0'''$ with $g_0(u_0''') = [\vv]$, such that $u_0'''\neq u_j'''$ for all $j$ and such that $(u_0''',u_j''')$ is a $1$-simplex for all $j$.
\end{enumerate}
then $\Fib_W(\epsilon_\bullet^{\star})$ is $(g-2)$-connected. 

For the first condition, Proposition \ref{proposition:RW1-4.1} above says that this is indeed $(g-2)$-connected. The link of an $i$-simplex $([\vv_0],\ldots,[\vv_i])$ is canonically isomorphic to the complex $\pi_0 O(W\setminus (\vv_0\cup \ldots \cup \vv_i);k_0,k_1)$ (for certain embedded intervals $k_0,k_1$), and the surface $W\setminus (\vv_0\cup \ldots \cup \vv_i)$ has genus $g-i-1$ or $g-i$, depending on whether $\ell\cap \delta$ is connected or not (see Proposition \ref{prop:fibrations-arcs}). In both cases, the link is $(g-i-3)$-connected, and so it is weakly Cohen-Macaulay of dimension $g-1$. 

The second condition holds by inspection, and the third condition is automatic as $\vert g_\bullet \vert$ is the geometric realisation of a semi-simplicial map. Finally, the fourth condition holds by general position: Since $M$ is simply-connected, it is possible to find a map $u'''_0\colon D^+_{1}\to M$ such that its restriction to $\pp D^+_{1}$ lies in $\ell$ and its restriction to $\ppp D^+_{\frac{1}{2}}$ is $\vv$. Because the dimension of $M$ is at least $5$, a small perturbation of $u'''_0$ makes it to be embedded and also disjoint from each of the other $u'''_j$. Then $(u'''_0,u'''_j)$ is a $1$-simplex for all $j$.
\end{proof}

We now wish to understand the homotopy type of the spaces $\OO{g,b}{M;\delta,\ell}{i}$ of $i$-simplices in these semi-simplicial resolutions. We will show that they can be described in terms of a certain space $A_i(M;\delta,\ell)$ and a space of surfaces $\e{g',b'}{{M}';\delta'}$ in a different manifold $M'$, with different boundary conditions and different genus and number of boundaries. This description will allow us to deduce homological information about $\e{g,b}{{M};\delta}$ from similar information about spaces of surfaces of smaller genus.

\begin{df} Let $A_i(M;\delta,\ell)$ be the set of tuples $(a_0,\ldots,a_i)$ with $a_j = (u'''_j,u'_j,L_j)$ and 
\begin{enumerate}
\item the $u'''_j\colon D^+_{1} \rightarrow M$ are pairwise disjoint embeddings with $\partial_0u'''_j \subset \ell$, and both $\ell_0\cap u'''_j>\ell_0\cap u'''_k$ and $\ell_1\cap u'''_j<\ell_1\cap u'''_k$ if $j>k$;
\item $u'_j$ is a closed tubular neighbourhood of $u'''_j$ in $M$ disjoint from $u'_k$ if $j\neq k$ whose restriction to $\pp D^+_{1}$ is a closed tubular neighbourhood in $\ell\subset \pp M$;
\item $L_j\subset \N_M \ppp u'''_j$  is an oriented $1$-dimensional subbundle such that $L_{j|\partial\ppp u'''_j} = N_{\delta\cap\ell} (\partial \ppp u'''_j)$, i.e., $L_j\in \Gr_1^+(\N_M \ppp u'''_j;\N_{\delta\cap\ell} (\partial \ppp u'''_j))$.
\end{enumerate}
This space is a union of components of the space $\TubEmb_{1,\ppp D^+_{\frac{1}{2}}\times[i]}( D^+_{1}\times [i], M;q,q_N,q_C)$, where $C=\ppp D^+_{\frac{1}{2}}$ and 
\begin{align*}
q_C(\{(-\tfrac{1}{2},0)\}) &= \ell_0 & q(\pp D^+_1) &= \ell \\
q_C(\{(\tfrac{1}{2},0)\}) &= \ell_1 & q_N(\pp D^+_1) &= \ell
\end{align*}
and $q(x)=M$ otherwise. We use this to give a topology to $\A{i}{M;\delta,\ell}$.
\end{df}

There are restriction maps
\begin{equation}\label{eq:834}
\OO{g,b}{M;\delta,\ell}{i} \longrightarrow \A{i}{M;\delta,\ell}
\end{equation}
that send $(W,u_0,\ldots,u_i)$ to $(a_0,\ldots,a_i)$, where $a_j = (u'''_j,u'_j,N_W \vv_j)$. 

\begin{proposition}\label{fibrations-arcs}\label{prop:fibrations-arcs} The restriction maps \eqref{eq:834} are fibrations, and, using the notation of Section \emph{\ref{ss:gluingmaps}}, their fibres over a point ${\bf u} = (u_0,\ldots,u_i)$ in $\A{i}{M;\delta,\ell}$ are given by
\[\begin{array}{rcccc}
\e{g-i-1,b+i+1}{\stminus{M}{\bf u};\stminus{\delta}{\bf u}}&\longrightarrow& \h{g,b}{M;\delta,\ell}{i} &\longrightarrow &\A{i}{M;\delta,\ell} \\[0.2cm]
\e{g-i,b+i-1}{\stminus{M}{\bf u};\stminus{\delta}{\bf u}}&\longrightarrow& \b{g,b}{M;\delta,\ell}{i} &\longrightarrow& \A{i}{M;\delta,\ell},
\end{array}\]
depending on how many components of $\delta$ intersect $\ell$. 
\end{proposition}
The manifold $M({\bf u})$ is obtained from $M$ by cutting out an open half disc, which does not change the manifold up to diffeomorphism. It is for this reason that we have cut out a neighbourhood of a half disc rather than a neighbourhood of an arc, and we will use this property in Section \ref{section:triviality}.
\begin{df}\label{def:appaug1} When $i=0$, the composition of the inclusion of the fibre in \eqref{eq:834} with the augmentation $\OO{g,b}{M;\delta,\ell}{0} \to \e{g,b}{M;\delta}$ is called the \emph{approximate augmentation} of the resolution $\OO{g,b}{M;\delta}{\bullet}$ over the $i$-simplex ${\bf u}$.
\end{df}
\begin{proof}[Proof of Proposition \ref{fibrations-arcs}]
The restriction maps are $\Diff(M;\delta,\ell)$-equivariant and, by Proposition \ref{prop:retractil-fibrations}, the space $A_i(M;\delta,\ell)$ is $\Diff(M;\delta,\ell)$-locally retractile, hence the restriction maps are locally trivial fibrations by Lemma \ref{lemma:retractil-target}.

The fibre over a point ${\bf u}$ is the space of surfaces $W$ in $M$ that contain the strips $(u_0'',\ldots,u_i'')$ and such that $W\backslash (u_0'',\ldots,u_i'')$ lies outside $u'_0\cup\ldots\cup u'_i$. If we take a parametrisation $f\colon \SS\rightarrow W$ of any surface and write ${\bf s} = (s_0,\ldots,s_i) = (f^{-1}\circ \ww_0,\ldots,f^{-1}\circ \ww_i)$, then this space is canonically homeomorphic to the space $\E(\stminus{\SS}{{\bf s}},\stminus{M}{{\bf u}};\stminus{\delta}{\bf u})$, so we just need to classify $\stminus{\SS}{\bf s}$.

Removing a strip from $\SS$ is the same as removing a $1$-cell, up to homotopy equivalence, hence $\chi(\stminus{\SS}{{\bf s}}) = \chi(\SS)+i+1$. Now, let us say that a strip $s_j$ in $\SS$ is of type I if both components of $\partial s_j$ are contained in a single component of $\partial \stminus{\SS}{(s_0,\ldots,s_{j-1})}$, and that it is of type II otherwise. 
\begin{itemize}
\item If $s_j$ is of type I, then $\partial \stminus{\SS}{s_0,\ldots,s_j}$ has one more boundary component than $\delta$ and, as a consequence of the last condition of the definition of $O(\SS)_\bullet$, the strip $s_{j+1}$ is again of type I.
\item If $s_j''$ is of type II, then $\partial \stminus{\SS}{s_0,\ldots,s_j}$ has one less boundary component than $\delta$ and, as a consequence of the last condition of the definition of $O(\SS)_\bullet$, the strip $s_{j+1}$ is of type I. 
\end{itemize}
Hence, the only strip of type II that may occur in the construction of $\partial \stminus{\SS}{{\bf s}}$ is the one given by $s_0$ in $O^2(\SS)_{\bullet}$. Hence $\partial \stminus{\SS}{{\bf s}}$ has $b+i+1$ components if ${\bf s}\in O^1(\SS)_\bullet$ and $b+i-1$ components if ${\bf s}\in O^2(\SS)_{\bullet}$. Finally, we obtain the genus of $\stminus{\SS}{\bf{s}}$ from the formula $g=\frac{1}{2}(2-\chi-b)$.
\end{proof}

\begin{figure}[h]
\centering
\subfloat[]{\includegraphics[width=5.5cm,angle=0]{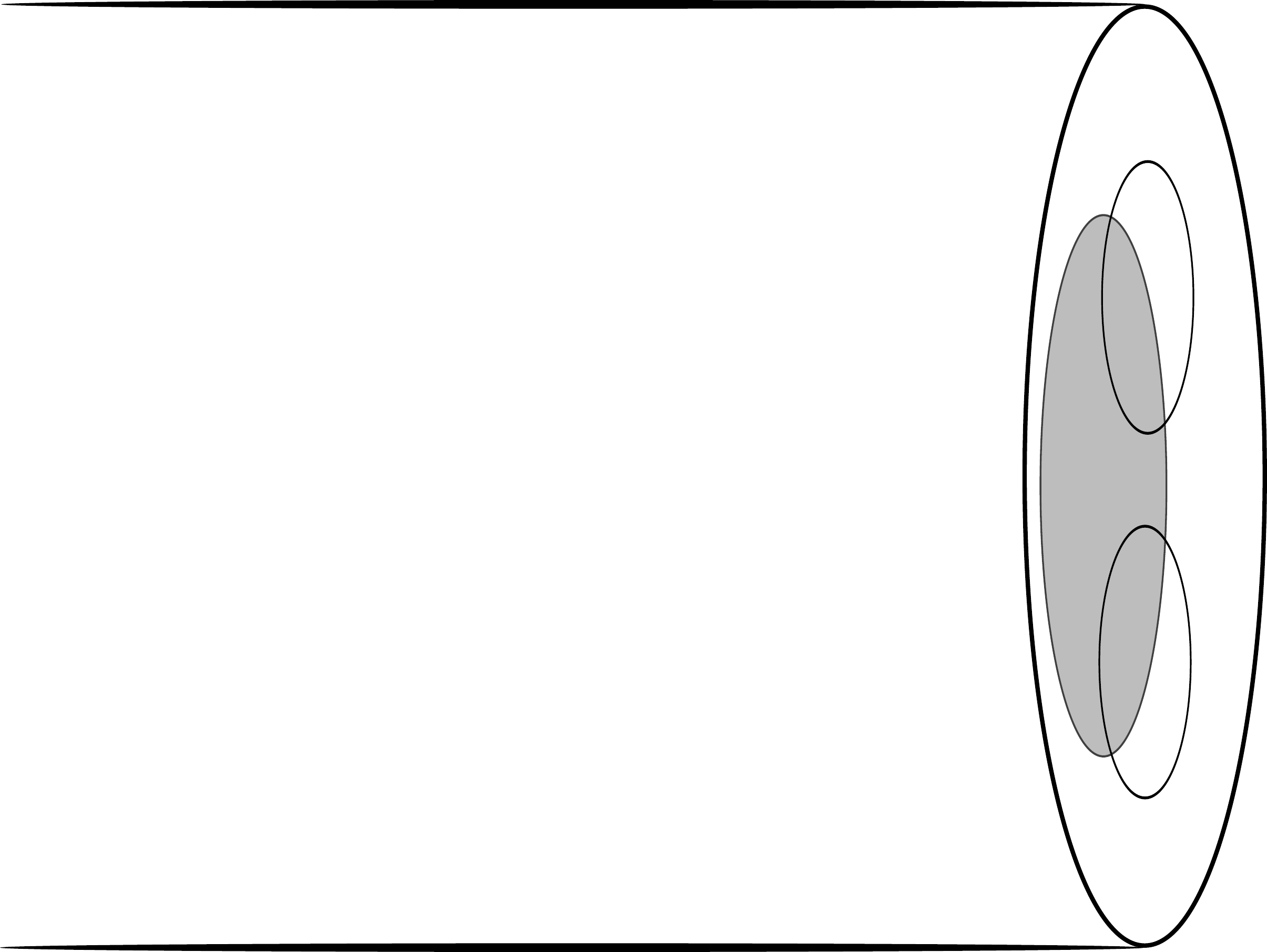}}
\hspace{1cm}
\subfloat[]{\includegraphics[width=5.5cm,angle=0]{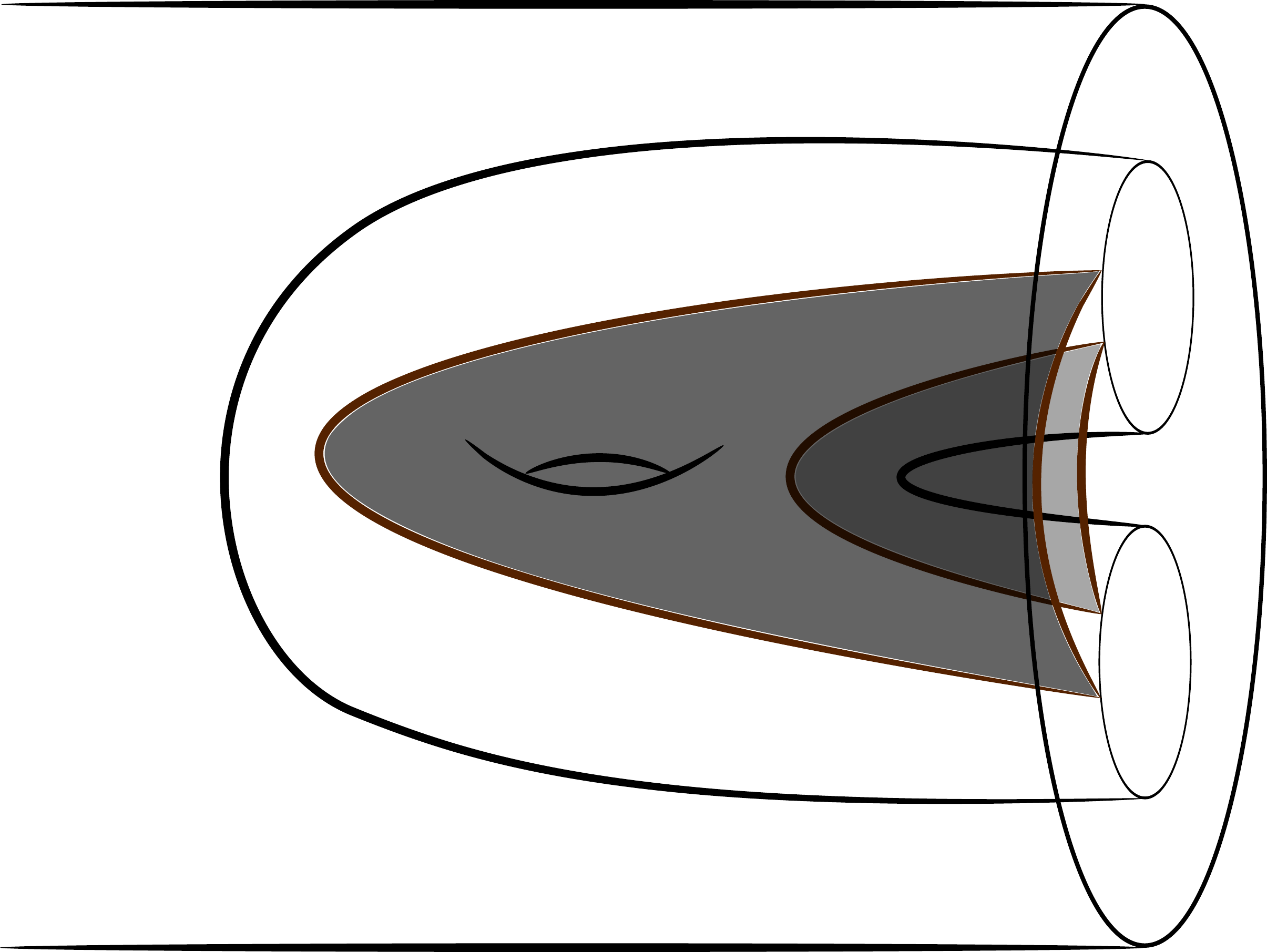}}
\caption{(a) A background manifold $M = D^2\times [0,1)$, a boundary condition consisting of two circles, and an embedded ball $\ell$ (in grey) that intersects the circles in two intervals. (b) A $1$-simplex in the boundary resolution.}
\end{figure}


\subsection{Stabilisation maps between resolutions}\label{stab-maps-resolution-arcs} 

In this subsection we show how to extend the stabilisation maps defined in Section \ref{ss:gluingmaps} to maps between the resolutions we have constructed.
\footnotesize
\begin{equation}\label{eq:stab-resolutions}
\begin{gathered}
\begin{array}{cc}
\xymatrix@C=.9cm{
\b{g,b}{M;\delta,\ell}{i}\ar@{-->}[r] \ar[d]^{\epsilon_i} &
\h{g+1,b-1}{M_1;\z{\delta}, \z{\ell}}{i} \ar[d]^{\epsilon_i} \\
\e{g,b}{M,\delta} \ar[r]^-{\alpha_{g,b}(M;\delta,\z{\delta})} & 
\e{g+1,b-1}{M_1;\z{\delta}}
}
&
\xymatrix@C=.9cm{\h{g,b}{M;\delta,\ell}{i}\ar@{-->}[r] \ar[d]^{\epsilon_i} &
\b{g,b+1}{M_1;\z{\delta}, \z{\ell}}{i} \ar[d]^{\epsilon_i} \\
\e{g,b}{M,\delta} \ar[r]^-{\beta_{g,b}(M;\delta,\z{\delta})} & 
\e{g,b+1}{M_1;\z{\delta}}.
}
\end{array}
\end{gathered}
\end{equation} 
\normalsize
In Section \ref{ss:gluingmaps}, we defined the maps $\alpha_{g,b}(M;\delta,\z{\delta})$ by gluing a cobordism $P \subset\pp M\times I$ to each surface in $\e{g,b}{M;\delta}$. As we did there, in the following constructions we will assume, without loss of generality, that 
\begin{enumerate}
\item $\z{\ell} = \ell\times \{1\}$. 
\item $P \cap (\ell\times [0,1]) = (\ell\cap \delta)\times [0,1]$, in particular $\z{\ell}\cap \z{\delta} = (\ell\cap \delta)\times \{1\}$.
\item $\ell_0\cup \ell_1$ hit all components of the relevant part of $\dd$ (see Definition \ref{df:relevant}). 
\end{enumerate}
These assumptions make the extension of the stabilisation map canonical: Let us define $\y{u}_j = \pp u_j\times I$. Then, joining each term in the tuple $(u_0,\ldots,u_i)$ ---that are subsets of $M$--- 
 to the products $(\y{u}_0,\ldots,\y{u}_i)$ ---that are subsets of $\pp M\times I$---, we obtain new triples $(\z{u}_0,\ldots,\z{u}_i)$ ---that are subsets of $M_1$---
, defined as $\z{u}_j = u\cup \y{u}_j$. This yields the dashed maps $\alpha_{g,b}(M;\delta,\z{\delta})_i$ in the first diagram. These maps commute with the face maps and with the augmentation maps, so they define a map of semi-simplicial spaces 
\[\alpha_{g,b}(M;\delta,\z{\delta})_\bullet\colon \b{g,b}{M;\delta,\ell}{\bullet}\longrightarrow \h{g+1,b-1}{M_1;\z{\delta},\z{\ell}}{\bullet}\]
which is augmented over $(\alpha_{g,b}(M;\delta,\z{\delta}))$. Analogously, we can define a map
\[\beta_{g,b}(M;\delta,\z{\delta})_\bullet\colon \h{g,b}{M;\delta,\ell}{\bullet}\longrightarrow \b{g,b+1}{M_1;\z{\delta},\z{\ell}}{\bullet}\]
which is augmented over $(\beta_{g,b}(M;\delta,\z{\delta}))$.

\begin{corollary}[To Proposition \ref{resolution-arcs}]\label{cor:relative-resolutions-arcs} The semi-simplicial pair $(\alpha_{g,b}(M;\delta,\z{\delta})_\bullet)$ together with the natural augmentation map to $(\alpha_{g,b}(M;\delta,\z{\delta}))$ is a $g$-resolution. The semi-simplicial pair $(\beta_{g,b}(M;\delta,\z{\delta})_\bullet)$ together with the natural augmentation map to $(\beta_{g,b}(M;\delta,\z{\delta}))$ is a $(g-1)$-resolution.
\end{corollary}

There is a commutative square
\begin{equation}\label{eq:98}
\begin{gathered}
\xymatrix{
\b{g,b}{M;\delta,\ell}{i} \ar[d]\ar[rr]^-{\alpha_{g,b}(M;\delta,\z{\delta})_i}& & \h{g+1,b-1}{M_1;\z{\delta,},\z{\ell}}{i} \ar[d]\\
A_i(M;\delta,\ell) \ar[rr]^-{{\bf u}\mapsto {\bf\z{u}}} & & A_i(M_1;\z{\delta},\z{\ell})
}
\end{gathered}
\end{equation}
where the lower map is a homotopy equivalence. Hence we obtain a map between the fibres over the points ${\bf u}$ and ${\bf \z{u}}$ of the fibrations of Proposition \ref{fibrations-arcs},
\begin{equation}\label{eq:65}\e{g-i,b+i-1}{\stminus{M}{{\bf u}};\stminus{\delta}{{\bf u}}}\longrightarrow \e{g-i,b+i}{\stminus{M_1}{{\bf \z{u}}};\stminus{\z{\delta}}{{\bf \z{u}}}}.\end{equation}
More concretely, this is a map of type I given by the cobordism $P({\bf u}):= P\setminus{\bf\tilde{u}''}\subset \stminus{\pp M}{{\bf u'}}\times I$, which is denoted $\beta_{g-i,b+i-1}(\stminus{M}{{\bf u}};\stminus{\delta}{{\bf u}},\stminus{\z{\delta}}{{\bf \z{u}}})$, as can be seen from the difference between the genus of the surfaces in the source and target spaces. 

As the map $A_i(M;\delta,\ell)\rightarrow A_i(M_1;\z{\delta},\z{\ell})$ is a homotopy equivalence, the space $\e{g-i,b+i-1}{\stminus{M}{{\bf u}};\stminus{\delta}{{\bf u}}}$ is homotopy equivalent to the homotopy fibre of the composition of the augmentation map of $\b{g,b}{M;\delta,\ell}{\bullet}$ with this map. Moreover, we have shown that the map between the fibres of the locally trivial fibrations of diagram (\ref{eq:98}) is a stabilisation map $\beta_{g-i,b+i-1}(\stminus{M}{{\bf u}};\stminus{\delta}{{\bf u}},\stminus{\z{\delta}}{{\bf \z{u}}})$.

As a consequence, we have a diagram
\begin{equation}\label{eq:900}
\begin{gathered}
\xymatrix{
\e{g-i,b+i-1}{\stminus{M}{{\bf u}};\stminus{\delta}{{\bf u}}}\ar[rr]^{\beta_{g-i,b+i-1}(\stminus{M}{{\bf u}};\stminus{\delta}{{\bf u}},\stminus{\z{\delta}}{{\bf \z{u}}})}\ar[d]^{\simeq}&& \e{g-i,b+i}{\stminus{M_1}{{\bf \z{u}}};\stminus{\z{\delta}}{{\bf \z{u}}}}\ar[d]^{\simeq} \\
\hofib_{{\bf \z{u}}}(\rho)\ar[rr]\ar[d] && \hofib_{{\bf \z{u}}}(\rho') \ar[d]\\
\b{g,b}{M;\delta,\ell}{i} \ar[dr]^{\rho}\ar[rr]^-{\alpha_{g,b}(M;\delta,\z{\delta})_i}& & \h{g+1,b-1}{M_1;\z{\delta},\z{\ell}}{i} \ar[dl]_{\rho'}\\
 & A_i(M_1;\z{\delta},\z{\ell}).&
}
\end{gathered}
\end{equation}
This gives that the pair $(\hofib_{{\bf \z{u}}}(\rho'),\hofib_{{\bf \z{u}}}(\rho))$ is homotopy equivalent to the pair $(\beta_{g-i,b+i-1}(\stminus{M}{{\bf u}};\stminus{\delta}{{\bf u}},\stminus{\z{\delta}}{{\bf \z{u}}}))$.

Following the same procedure with the map $\beta_{g,b}(M;\delta,\z{\delta})$, we obtain that the pair given by the map from the homotopy fibre of 
\[\h{g,b}{M;\delta,\ell}{i}\lra A_i(M_1;\z{\delta},\z{\ell})\]
 to the fibre of the composition
\[\h{g,b}{M;\delta}{i}\lra A_i(M;\delta,\ell)\lra A_i(M_1;\z{\delta},\z{\ell})\]
 is homotopy equivalent to the pair given by
\[\e{g-i-1,b+i+1}{\stminus{M}{{\bf u}};\stminus{\delta}{{\bf u}}}\longrightarrow \e{g-i,b+i}{M_1({\bf \z{u}});\stminus{\z{\delta}}{{\bf \z{u}}}},\]
which is a map of type $\alpha_{g-i-1,b+i+1}(\stminus{M}{{\bf u}};\stminus{\delta}{{\bf u}},\stminus{\z{\delta}}{{\bf \z{u}}})$. 
\begin{corollary}[To Proposition \ref{fibrations-arcs}]\label{cor:relative-fibrations-arcs} There are homotopy fibre sequences
\[\begin{array}{rcccl}
(\beta_{g-i,b+i-1}(\stminus{M}{{\bf u}};\stminus{\delta}{{\bf u}}))&\longrightarrow& (\alpha_{g,b}(M;\delta)_i) &\longrightarrow& A_i(M_1;\z{\delta},\z{\ell}), \\[0.2cm]
(\alpha_{g-i-1,b+i+1}(\stminus{M}{{\bf u}};\stminus{\delta}{{\bf u}}))&\longrightarrow& (\beta_{g,b}(M;\delta)_i) &\longrightarrow& A_i(M_1;\z{\delta},\z{\ell}),
\end{array}\]
that is, the homotopy fibre over ${\bf \z{u}}$ is homotopy equivalent to the pair shown.
\end{corollary}
\begin{df} When $i=0$, the composition of the inclusion of the fibre in the fibre sequences of the last corollary with the relative augmentations 
\[\begin{array}{rcccl}
(\beta_{g,b-1}(\stminus{M}{{\bf u}};\stminus{\delta}{{\bf u}}))&\longrightarrow& (\alpha_{g,b}(M;\delta)_0) &\longrightarrow& (\alpha_{g,b}(M;\delta)), \\[0.2cm]
(\alpha_{g-1,b+1}(\stminus{M}{{\bf u}};\stminus{\delta}{{\bf u}}))&\longrightarrow& (\beta_{g,b}(M;\delta)_0) &\longrightarrow& (\beta_{g,b}(M;\delta)),
\end{array}\]
are called the \emph{approximate augmentations} of the resolutions $(\alpha_{g,b}(M;\delta)_\bullet)$ and $(\beta_{g,b}(M;\delta)_\bullet)$ over the $i$-simplex ${\bf u}$.
\end{df}

\section{Homological stability for surfaces with boundary}\label{section:homologicalstability}
In this section we prove the first two assertions of Theorem \ref{thm:stab}, leaving some details until Section \ref{section:triviality}. The proof of the last assertion will be deferred to Section \ref{section:lastboundary}. 

\begin{proposition}\label{prop:71}
Let $M$ be a simply-connected manifold of dimension at least $5$. If the dimension of $M$ is $5$, we assume in addition that the pairs of pants defining the stabilisation maps are contractible in $\pp M\times [0,1]$. Then
\begin{enumerate}
\item $H_k(\alpha_{g,b}(M)) = 0$ for $k\leq \frac{1}{3}(2g+1)$;
\item $H_k(\beta_{g,b}(M)) = 0$ for $k\leq \frac{2}{3}g$.
\end{enumerate} 
\end{proposition}

This proposition will be proven by induction: Lemma \ref{lemma:stabpi0} gives the starting step and Lemma \ref{lemma:induction} gives the inductive step. The proof of Proposition \ref{prop:71} broadly follows the proof of Theorem 7.1 in \cite{R-WResolution2015}. In the language in that paper, stabilisation on $\pi_0$ is covered by Lemma \ref{lemma:stabpi0} and $1$-triviality will be the subject of Section \ref{section:triviality}.

\subsection{The inductive step}

Let $\mathcal M$ be the class of simply-connected manifolds of dimension at least $5$ with non-empty boundary. We define the following statements, which we will prove by simultaneous induction: Firstly, for the stabilisation maps,

\begin{itemize}
	\item[$F_g:$]  $H_{k}(\alpha_{h,b}(M))=0$ for all $M\in \mathcal M$, all $h\leq g$ and all $k\leq \frac{1}{3}(2h+1)$;
	\item[$G_g:$]  $H_{k}(\beta_{h,b}(M))=0$ for all $M\in \mathcal M$, all $h\leq g$ and all $k\leq \frac{2}{3}h$.
\end{itemize}

\noindent Secondly, for the approximated augmentations for the $g$-resolution $\alpha_{g,b}(M)_\bullet$ and the $(g-1)$-resolution $\beta_{g,b}(M)_\bullet$ over any $0$-simplex $u$,

\begin{itemize}
	\item[$X_g:$]  $H_{k}(\beta_{h,b-1}(M(u)))\rightarrow H_k(\alpha_{h,b}(M))$ is an epimorphism for all $M\in \mathcal M$, all $h\leq g$ and all $k\leq \frac{1}{3}(2h+1)$; 
	\item[$A_g:$]  $H_{k}(\beta_{h,b-1}(M(u)))\rightarrow H_k(\alpha_{h,b}(M))$ is zero for all $M\in \mathcal M$, all $h\leq g$ and all $k\leq \frac{1}{3}(2h+2)$; 
	\item[$Y_g:$]  $H_{k}(\alpha_{h-1,b+1}(M(u)))\rightarrow H_k(\beta_{h,b}(M))$ is an epimorphism for all $M\in \mathcal M$, all $h\leq g$ and all $k\leq \frac{2}{3}h$; 
	\item[$B_g:$]  $H_{k}(\alpha_{h-1,b+1}(M(u)))\rightarrow H_k(\beta_{h,b}(M))$ is zero for all $M\in \mathcal M$, all $h\leq g$ and all $k\leq \frac{1}{3}(2h+1)$.
\end{itemize}

\begin{lemma}\label{lemma:induction} If $M$ satisfies the hypotheses of Proposition \emph{\ref{prop:71}}, then
\[\begin{array}{ccc}
	 {\rm (i)}~X_g \text{ and } A_g \Rightarrow F_g; & {\rm (iii)}~G_{g} \Rightarrow X_g; & 
	{\rm (v)}~G_{g} \text{ and } X_{g-1}\Rightarrow A_g; \\
	 {\rm (ii)}~Y_g \text{ and } B_g \Rightarrow G_g; & {\rm (iv)}~F_{g-1}\Rightarrow Y_g; & 
	 {\rm (vi)}~F_{g-1} \text{ and } Y_{g-1}\Rightarrow B_g.
\end{array}\]
\end{lemma}

\begin{proof}

\mbox{}
 (i) The morphism induced in homology by the approximate augmentation
$$H_{k}(\beta_{g,b-1}(M(u)))\lra H_k(\alpha_{g,b}(M))$$
is both zero and an epimorphism in all degrees $k\leq \frac{1}{3}(2g+1)$ (since $X_g$ and $A_g$ hold), so $H_k(\alpha_{g,b}(M)) = 0$ in these degrees. Similarly for (ii).

 (iii) Consider the $g$-resolution $\alpha_{g,b}(M;\delta)_\bullet$ of $\alpha_{g,b}(M;\delta)$ given by Corollary \ref{cor:relative-resolutions-arcs}, together with the homotopy fibre sequences
$$(\beta_{g-i,b+i-1}(M({\bf u});\delta({\bf u})))\longrightarrow (\alpha_{g,b}(M;\delta)_i) \longrightarrow A_i(M_1;\z{\delta})$$
of Corollary \ref{cor:relative-fibrations-arcs}. For all $i\geq 1$ we have the inequality $\frac{1}{3}(2g+1)-i\leq \frac{2}{3}(g-i)$, and so, as $M({\bf u})$ is simply-connected, by inductive assumption \[H_q(\beta_{g-i,b+i-1}(M({\bf u});\delta({\bf u}))) = 0\] for $q\leq \frac{1}{3}(2g+1)-i$. When $i=0$, we have the inequality $\frac{1}{3}(2g+1)-1\leq \frac{2}{3}g$ so $H_q(\beta_{g,b-1}(M({\bf u});\delta({\bf u}))) = 0$ for $q\leq \frac{1}{3}(2g+1)-1$. In total we deduce that we have $H_q(\beta_{g-i,b+i-1}(M({\bf u});\delta({\bf u}))) = 0$ for $q\leq \frac{1}{3}(2g+1)-i$ except $(q,i)=(\frac{1}{3}(2g+1),0)$. As $\lfloor \frac{1}{3}(2g+1)\rfloor\leq g+1$, and $A_i(M_1,\z{\delta})$ is path connected, Criterion \ref{criterion3} shows that the approximate augmentations are epimorphisms for $k\leq \frac{1}{3}(2g+1)$. Similarly for (iv).

The implications (v) and (vi) will be proven in Section 
 \ref{section:triviality}. 
\end{proof}

\subsection{Stability of connected components}

In this section we establish the base case of the induction, by proving that the assertions $X_g$, $Y_g$, $A_g$ and $B_g$ hold for $g=0$. This, together with Lemma \ref{lemma:induction}, will finish the proof of Proposition \ref{prop:71}. For that, we first construct a (non-canonical) bijection between $\pi_0(\e{g,b}{M;\delta})$ and the second homology group of $M$.

Suppose that $M$ is a simply-connected manifold of dimension $d\geq 5$, and let us describe an action of the abelian group $H_2(M;\bZ)$ on the set $\pi_0\e{g,b}{M;\delta}$ of isotopy classes of surfaces of genus $g$ in $M$ with boundary condition $\delta$. Let $\hat{e} : \Sigma_{g,b} \hookrightarrow M$ be an embedding with boundary condition $\delta$, representing an element $e \in \pi_0\e{g,b}{M;\delta}$. Let $x \in \pi_2(M) \cong H_2(M;\bZ)$ be a homotopy class of maps from $S^2$ to $M$.

As $M$ has dimension at least 5, $x$ may be represented by an embedding $\hat{x} : S^2 \hookrightarrow M$ disjoint from the image of $\hat{e}$, and we can then choose an embedded path from the image of $\hat{e}$ to the image of $\hat{x}$. Forming the ambient connected sum along this path we obtain a new embedding $\hat{x} \cdot \hat{e} : \Sigma_{g,b} \hookrightarrow M$.

\begin{lemma}\label{lemma:spaces-pi0} 
The map
\begin{align*}
H_2(M;\bZ) \times \pi_0\e{g,b}{M;\delta} &\lra \pi_0\e{g,b}{M;\delta}\\
(x, e) & \longmapsto [\hat{x} \cdot \hat{e}]
\end{align*}
is well defined and gives a free and transitive action of $H_2(M;\bZ)$ on $\pi_0\e{g,b}{M;\delta}$.

If $\partial\colon H_2(M,\d;\bZ)\rightarrow H_1(\d;\bZ)$ denotes the boundary homomorphism and $[\d]\in H_1(\d;\bZ)$ denotes the fundamental class, then the map
\begin{align*}
\pi_0\e{g,b}{M;\delta} &\longrightarrow \partial^{-1}([\d])\\
[\hat{e}] &\longmapsto \hat{e}_*([\Sigma_{g,b}, \partial \Sigma_{g,b}])
\end{align*}
is an isomorphism of $H_2(M;\bZ)$-sets.
\end{lemma}

\begin{proof}

Consider the natural $\Diff(\SS_{g,b})$-equivariant inclusion
\[\varphi\colon \Emb(\SS_{g,b},M;\delta)\longrightarrow \map(\SS_{g,b},M;\delta).\]

As the dimension of $M$ is at least 5 and it is simply-connected, the main result of \cite{Haefliger} says that $\varphi$ induces a bijection 
\begin{equation}\label{eq:haefliger}
\pi_0\Emb(\SS_{g,b},M;\delta)\cong\pi_0\map(\SS_{g,b},M;\delta).
\end{equation}
Consider the cofibre sequence $S^1\stackrel{i}{\rightarrow} \SS^1_{g,b}\rightarrow \SS_{g,b}$, where $\SS^1_{g,b}$ denotes a $1$-skeleton of $\SS_{g,b}$ to which just a single $2$-cell needs to be attached to obtain $\SS_{g,b}$. The second inclusion gives the locally trivial fibration
\[\map(\SS_{g,b},M;\delta)\longrightarrow \map(\SS^1_{g,b},M;\delta),\]
whose base space is connected because $M$ is simply-connected. The fibre over a point $\phi\in \map(\SS^1_{g,b},M;\delta)$ is the space $\map(D^2,M;\phi)$ of maps from the $2$-disc $D^2$ to $M$ that restrict to $\phi\circ i$ on the boundary. 

By considering the long exact sequence on homotopy groups for this fibration, in low degrees we find that $\pi_1(\map(\SS^1_{g,b},M;\delta),\phi)$ acts on the set $\pi_0(\map(D^2,M;\phi))$ with quotient $\pi_0(\map(\SS_{g,b},M;\delta))$. 

We have the composition
\begin{equation*}
\xymatrix{
\pi_0(\map(D^2,M;\phi)) \ar@{->>}[r]& \pi_0(\map(\SS_{g,b},M;\delta)) \ar@{->>}[r] & \partial^{-1}([\delta])
}
\end{equation*}
where the source has a free transitive $\pi_2(M)$-action and the target has a free transitive $H_2(M;\bZ)$-action, and the map is equivariant with respect to the Hu\-re\-wicz homomorphism. This shows that both maps are in fact bijections, and that the induced $\pi_2(M)$-action on the set $\pi_0(\map(\SS_{g,b},M;\delta))$ is free and transitive.

Given this calculation, it is clear that the group $\Diff^+(\Sigma_{g,b})$ acts trivially on the set $\pi_0\Emb(\SS_{g,b},M;\delta))$, so there is an induced bijection 
\[\pi_0\e{g,b}{M;\delta} \lra \partial^{-1}([\delta]).\] It is then easy to see that the $H_2(M;\bZ)$-action on $\partial^{-1}([\delta])$ corresponds to the one that we constructed on $\pi_0\e{g,b}{M;\delta}$.
\end{proof}

This lemma allows us to begin the inductive proof of Proposition \ref{prop:71}, as it tells us what the zeroth homology of $\e{g,b}{M;\delta}$ is.

\begin{lemma}\label{lemma:stabpi0} If $M$ is simply-connected of dimension at least $5$, then the statements $F_0$ and $G_0$ hold. As a consequence, the statements $X_0$, $Y_0$, $A_0$ and $B_0$ hold too.
\end{lemma}
\begin{proof}
Each stabilisation map glues on a cobordism $P$ with incoming boundary $\d$ and outgoing boundary $\z{\d}$. With the notation of Lemma \ref{lemma:spaces-pi0}, adding on the 2-chain representing the relative fundamental class of $P$ defines an isomorphism of $H_2(M;\bZ)$-sets $\partial^{-1}([\d])\rightarrow \partial^{-1}([\z{\d}])$ between the inverse images of the fundamental classes $[\d]$ and $[\z{\d}]$, and hence $F_0$ and $G_0$ hold.
\end{proof}

\section{Trivial homology of approximate augmentations}\label{section:triviality}

This section is divided in two parts. In the first part we prove Lemma \ref{lemma:triviality}, which in the language of \cite{R-WResolution2015}  says that the space of embedded subsurfaces is 1-trivial. In the second part, we apply this lemma to prove assertions (v) and (vi) in Lemma \ref{lemma:induction}, hence finishing the proof of Proposition \ref{prop:71}. 

Throughout this section, we assume for simplicity that the manifold $M$ has no corners (observe that $M(u)$ does have corners). In Figures \ref{fig:triv-01}, \ref{fig:triv-02}, \ref{fig:triv-03} and the following Notation, we recall the decorations $\z{u},\y{u},u',u'',P(u)\ldots$ given in Definition \ref{df:ssresolution} and at the beginning of Section \ref{stab-maps-resolution-arcs}. 

\begin{notation} Recall that
\begin{itemize}
\item $\ell$ is a ball in $\pp M$,
\item $M_1 = M\cup (\pp M\times [0,1])$,
\item $u'''$ is an embedding of a half disc $D^+_1$ whose boundary lies in $W$ and in $\ell$,
\item $(u',u'')$ is a tubular neighbourhood of $u'''$ in the pair $(M,W)$,
\item $\vv$ is the restriction of $u'''$ to $\ppp D^+_{\frac{1}{2}}$,
\item $\ww$ is the image of  $u''\circ \vv$. It is also the intersection $u'\cap W$,
\item $\z{u}',\z{\ww},\z{v},\ldots$ are the prolongations of $u',\ww,\vv,\ldots\subset M$ to $M_1$,
\item $M(u)$ is the closure of the complement of $u'$ in $M$,
\item $M_1(\z{u})$ is the closure of the complement of $\z{u}'$ in $M$,
\item $\delta(u)$ is the boundary condition in $M$ obtained from $\delta$ by removing $\pp w$ and adding $\ppp w$. 
\item $\z{\delta}(\z{u})$ is the boundary condition in $M_1$ obtained from $\z{\delta}$ by removing $\pp \z{w}$ and adding $\ppp \z{w}$. 
\item $P(u)$ is the closure of $P\setminus \y{\ww}$.
\end{itemize}
\end{notation}

We recall some of this notation in Figures \ref{fig:triv-01}, \ref{fig:triv-02}, and \ref{fig:triv-03} below.

\begin{figure}[h]
\centering
\def\svgwidth{.6\textwidth}
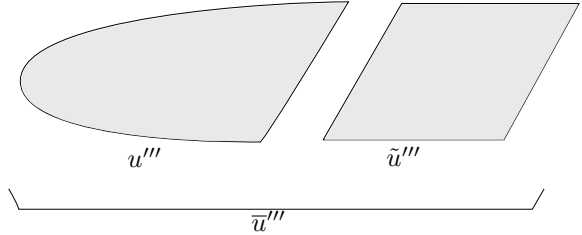
\caption{The discs $u'''\subset M$, $\y{u}'''\subset \pp M\times[0,1]$ and $\z{u}'''\subset M_1:= M\cup \pp M\times [0,1]$. The subsets $u'$, $\y{u}'$ and $\z{u}'$ are tubular neighbourhoods of them.
}
\label{fig:triv-01}
\end{figure}

\begin{figure}
\centering
\def\svgwidth{.6\textwidth}
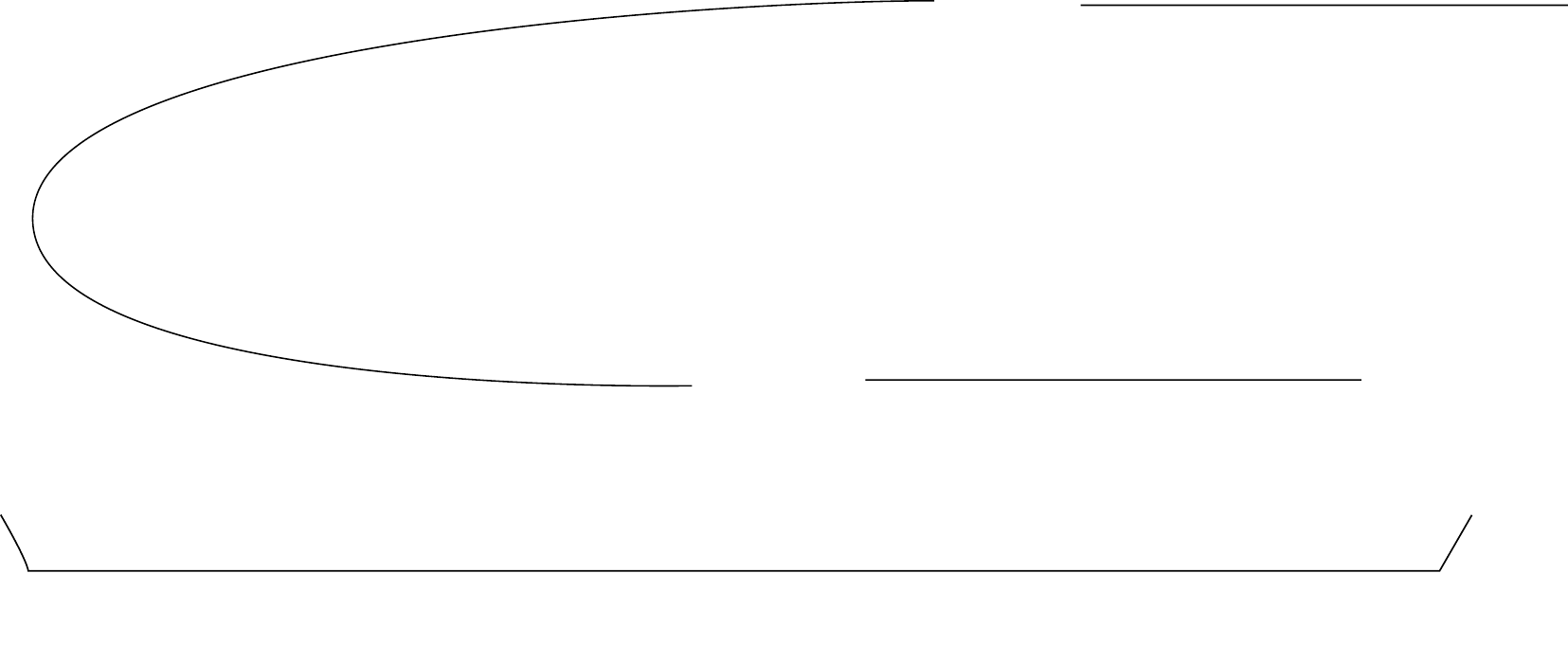
\caption{The arcs $\vv\subset W\subset M$, $\y{\vv}\subset P\subset \pp M\times[0,1]$ and $\z{\vv}\subset W\cup P\subset M_1$.
}
\label{fig:triv-02}
\end{figure}

\begin{figure}
\centering
\def\svgwidth{.6\textwidth}
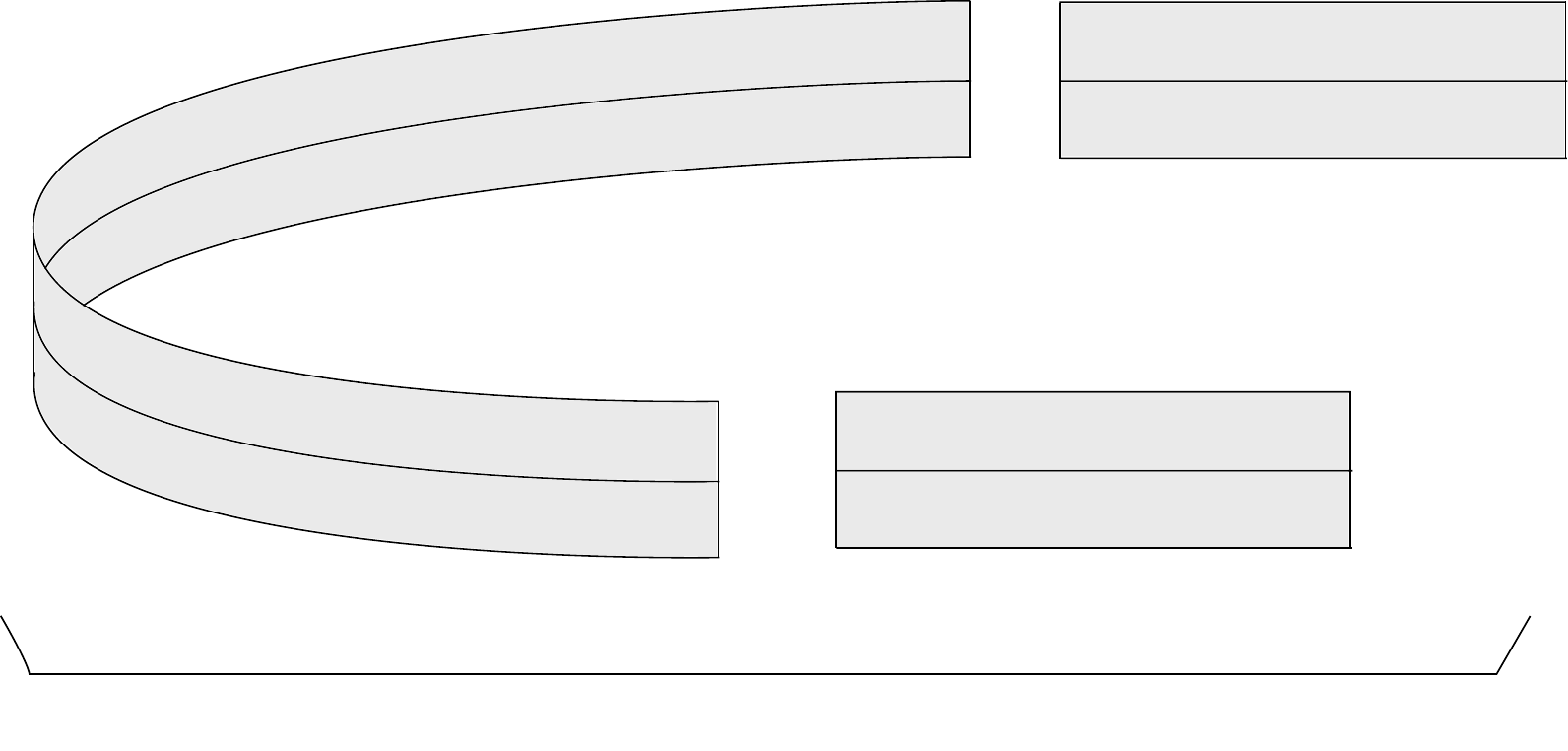
\caption{The strips $\ww\subset W\subset M$, $\y{\ww}\subset P\subset \pp M\times[0,1]$ and $\z{\ww}\subset W\cup P\subset M_1$.
}
\label{fig:triv-03}
\end{figure}

\subsection{A diagram of approximate augmentations}

Let us denote by $\ab_{g,b-1}(u)$ (\resp $\aa_{g,b}(u)$) the approximate augmentation for the resolution $\b{g,b}{M;\delta,\ell}{\bullet}$ (\resp $\h{g,b}{M;\delta}{\bullet}$) of $\e{g,b}{M;\delta}$ over a $0$-simplex $u$ (see Definitions \ref{df:ssresolution} and \ref{def:appaug1}). Then, for a stabilisation map $\alpha_{g,b}(M;\delta,\z{\delta})$, we may extend it to the following diagram, which results from gluing diagram \eqref{eq:stab-resolutions} to the big square of diagram \eqref{eq:900}:

\begin{equation}\label{eq:YYY}
\begin{gathered}
\xymatrix{
\e{g,b-1}{M(u);\delta(u)} \ar[rrrr]^{\beta_{g,b-1}(M(u);\delta(u))} \ar[d]_{\ab_{g,b-1}(u)}&&&& \e{g,b}{M_1(\z{u});\z{\delta}(\z{u})} \ar[d]^{\aa_{g,b}(u)}  \ar@{-->}[dllll]\\
\e{g,b}{M;\delta} \ar[rrrr]^{\alpha_{g,b}(M;\delta)} &&&& \e{g+1,b-1}{M_1;\z{\delta}}.
}
\end{gathered}
\end{equation}

The first result of this section is Lemma \ref{lemma:triviality} where we construct a dotted map as shown making both triangles commute up to homotopy (and enjoying certain other properties). In order to better understand the diagram, observe first that all its maps consist of gluing to the background manifold a cobordism with a surface inside, as we did in Section \ref{ss:gluingmaps}. In that notation we have: 
\begin{align*}
\alpha_{g,b}(M;\delta)        &= -\cup P                      &\text{with}&& P&\subset\ppartial M\times [0,1] &(\text{type}~I),\\ 
\beta_{g,b-1}(M(u);\delta(u)) &= -\cup P(u) &\text{with}&&P(u)&\subset \pp M(u)\times [0,1] &(\text{type}~I),\\ 
\ab_{g,b-1}(u)                &= -\cup \ww                   &\text{with}&& \ww&\subset u' &(\text{type}~II), \\
\aa_{g,b}(u)                  &= -\cup \z{\ww}  &\text{with}&&\z{\ww}&\subset \z{u}' &(\text{type}~II).
\end{align*}
\subsection{Enlarging the diagram} In practice, the only maps between spaces of submanifolds that we can define are the ones in Section \ref{ss:gluingmaps}. Since $M_1(\z{u})$ is not contained in $M$, such maps do not exist on the nose. For that reason, we now enlarge diagram \eqref{eq:YYY} by including $M$ into $M_2:=M\cup \ppartial M\times [0,2]$ and including $M_1$ into $M_3 := M\cup \ppartial M\times [0,3]$ (see also Figure \ref{fig:triv01}):
\begin{equation}\label{eq:XXX}
\begin{gathered}
\xymatrix{
\e{g,b-1}{M(u);\delta(u)} \ar[rrrr]^{\beta_{g,b-1}(M(u);\delta(u))}_{-\cup P(u)} \ar[d]_{\ab_{g,b-1}(u)}^{-\cup \ww}&&&& \e{g,b}{M_1(\z{u});\z{\delta}(\z{u})} \ar[d]_{\aa_{g,b}(u)}^{-\cup \z{\ww}}  \\
\e{g,b}{M;\delta} \ar[d]_{\ii_0(\delta)}^{-\cup \dd\times [0,2]}\ar[rrrr]^{\alpha_{g,b}(M;\delta)}_{-\cup P} &&&& \e{g+1,b-1}{M_1;\z{\delta}}\ar[d]_{\ii_1(\z{\delta})}^{-\cup \z{\delta}^0\times [1,3]}\\
\e{g,b}{M_2;\delta + 2} \ar[rrrr]^{\alpha_{g,b}(M_2)}_{-\cup (P+2)} &&&& \e{g+1,b-1}{M_3;\z{\delta} + 2},
}
\end{gathered}
\end{equation}
The maps in the bottom square are defined by 
\begin{align*}
\ii_0(\delta) &= -\cup \dd\times [0,2] &\text{with}&& \dd\times [0,2]&\subset \ppartial M\times [0,2]&(\text{type}~I),\\
\ii_1(\delta) &= -\cup \dd\times [1,3] &\text{with}&& \dd\times [1,3]&\subset \ppartial M\times [1,3]&(\text{type}~I),\\
\alpha_{g,b}(M_2) &= -\cup (P+2) &\text{with}&& P+2&\subset \ppartial M\times [2,3]&(\text{type}~I),
\end{align*}
where $P+2$ is the cobordism $P$ translated $2$ units. Finally, the new boundary conditions are:
\begin{align*}
(\delta + 2)^0 &= \dd\times \{2\} ,&\quad (\z{\delta} + 2)^0 &= \z{\delta}^0\times \{3\} \\
(\delta + 2)^1 &= \ddd\cup (\partial \dd\times [0,2]) ,&\quad (\z{\delta} + 2)^1 &= \z{\delta}^1\cup (\partial \z{\delta}^0\times [1,3]) 
\end{align*}
(if $M$ does not have corners, then $\ddd = \z{\delta}^1 = (\delta+2)^1 = (\z{\delta} + 2)^1 = \emptyset$). By stretching collars one sees that the maps $\ii_0(\delta)$ and $\ii_1(\z{\delta})$ are weak homotopy equivalences and that the bottom square commutes up to homotopy. Therefore in order to find a diagonal map in \eqref{eq:YYY} making the two triangles commute up to homotopy it is enough to find a diagonal map in \eqref{eq:XXX} making the two triangles commute up to homotopy.

\subsection{A diagonal factorisation of the enlarged diagram} Now, let $N$ be the manifold with corners $\z{u}'\cup \ppartial M\times [1,2]$. As $M_1(\z{u})\cup N = M_2$, any trivial cobordism $Q\subset N$ satisfying the boundary condition
\[\xi  := (\dd \times\{2\}) \;\cup \;  \z{\d}(\z{u})\quad \text{ inside }\quad \partial N =  (\ppartial M\times \{2\}) \cup \partial  M_1(\z{u}).\] 
defines a map
\begin{equation}\label{eq:mapQ}
\xymatrix{-\cup Q\colon \e{g,b}{M_1(\z{u});\z{\delta}(\z{u})}\ar[r]& \e{g,b}{M_2;\delta +  2}. }
\end{equation}
which is a diagonal map for the outer square in diagram \eqref{eq:XXX}. In order for this map to make the two triangles commute up to homotopy, $Q$ is required to have certain properties: a $Q$ enjoying the properties is shown to exist in the following lemma.

\begin{figure}
\centering
\def\svgwidth{\textwidth}
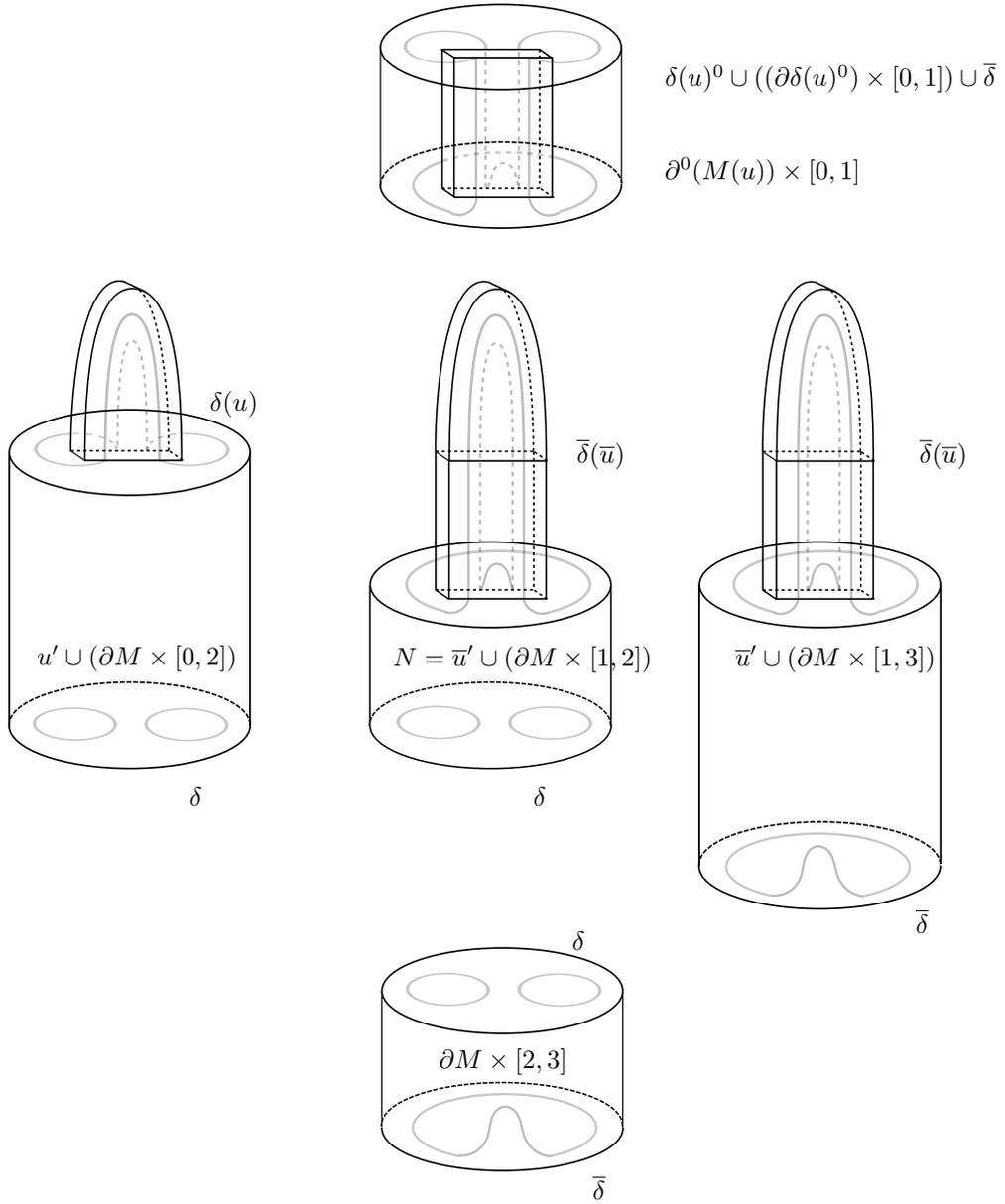
\caption{Background manifolds (in black) and boundary conditions (in grey) for the cobordisms $P(u)$, $\ww\cup (\dd\times [0,2])$,$Q$,$\z{\ww}\cup (\z{\delta}^0\times[1,3])$, $P+2$ that define the maps in diagram \eqref{eq:XXX}, with $\partial M$ a ball (so $M$ does not have corners in the picture). 
}
\label{fig:triv01}
\end{figure}
\clearpage

\begin{lemma}\label{lemma:triviality} If $\dim M\geq 5$, there is an $\ell$ as in Definition \ref{df:ssresolution} that depends only on $P$, for which there exist:
\begin{enumerate}
\item\label{lemma:triviality:1} a trivial cobordism $Q \in \E(N;\xi)$,
\item\label{lemma:triviality:2} isotopies 
\[\begin{array}{rclcl}
 P(u)\cup Q&\simeq & \ww\cup (\dd\times [0,2])&\subset& u'\cup \ppartial M\times [0,2]\\
 Q\cup (P+2) &\simeq &\z{\ww}\cup (\z{\delta}^0\times [1,3])&\subset &\z{u}'\cup\ppartial M\times [1,3]
\end{array}\]
relative to the boundaries.
\end{enumerate}
Therefore in the square \eqref{eq:YYY} a dashed diagonal map exists making both triangles commute up to homotopy.
\end{lemma}

Before entering into the proof of this lemma, we give some definitions to do with surgery on embedded surfaces which will be used in the proof.

\begin{notation}
If $E\to B$ is a vector bundle, we denote by $S(E)$ its unit sphere bundle.
\end{notation}

Choose an embedding $a\colon D^2\times \partial D^1\to D^2\times D^1$ such that $a_{|\partial D^2\times \partial D^1}$ is the identity and the image of $\{x\in D^2\mid \|x\|\geq \frac{1}{2}\} \times \partial D^1$ is contained in $\partial D^2\times D^1$.

\begin{df}\label{df:surg} Let $W\subset M$ be an oriented surface, and let $\gamma\subset W$ be an oriented embedded circle. The normal bundle $N_W\gamma$ is then oriented. As it has rank $1$, the orientation is the same as a trivialisation. Let $s_0$ be the section of $S(N_W\gamma)$ given by the trivialisation. Let $(D,s)$ be a pair consisting of an embedded oriented $2$-disc $D\subset M$ bounding $\gamma$ transverse to $W$, and a section $s$ of $S(N_MD)$ whose restriction to $S(N_W\gamma)$ is $s_0$.

A \emph{surgery datum} for $(W,(D,s))$ is an embedding of pairs $e\colon (D^2\times D^1, \partial D^2\times D^1)\to (M,W)$ such that 
\begin{enumerate}
\item the restriction of $e$ to $D^2\times \{0\}$ is an orientation preserving diffeomorphism onto $D$,
\item the canonical section of $S(N_{D^2\times D^1}D^2)$ is mapped to the section $s$ of $S(N_MD)$. 
\end{enumerate}
The \emph{ambient surgery} on $W$ along $\gamma$ by means of the pair $(D,s)$, denoted $W\natural D$ is the submanifold of $W$ obtained by removing $e((\partial D^2)\times D^1)$ from $W$ and then gluing $e\circ a(D^2\times \partial D^1)$. The orientation of $W$ determines an orientation on $W\natural D$.
\end{df}
 Any two parametrisations of $D$ differ by an element of $\Diff^+(D^2)$, which is connected. Additionally, the restriction map 
\[\Emb((D^2\times D^1,\partial D^2\times D^1),(M,W))\lra \Emb(D^2\times \{0\},\partial D^2 \times \{0\},(M,W))\]
is a fibration by Proposition \ref{prop:retractil-resolutions} and Lemma \ref{lemma:retractil-target}, and the fibre is weakly contractible. Therefore, any two surgery data of $(W,(D,s))$ are isotopic, and we deduce that: 
\begin{lemma}\label{lemma:solely} The isotopy class of $W\natural D$ is determined by $W$ and $(D,s)$.
\end{lemma}

\begin{lemma}\label{lemma:interchange} Let $B$ be a compact manifold, possibly with boundary. Let $W$ be an embedded oriented surface in $B$, not necessarily collared, with boundary condition $\delta$, and let $(\gamma,\gamma')$ be either a pair of collared arcs in $W$ with the same boundary, or a pair of curves in $W$. If $M$ is a compact manifold, possibly with boundary, contractible and of dimension at least $5$, and $\gamma$ and $\gamma'$ are non-separating (i.e., their complements are connected), then there is an isotopy $f_t\colon M\to M$ constant on the boundary of $M$ such that
\begin{enumerate}
\item $f_1(W) = W$,
\item $f_1(\gamma) = \gamma'$.
\end{enumerate}
\end{lemma}
\begin{proof}
As $\Diff^+_\partial(W)$ acts transitively on the space of non-separating arcs (curves) in $W$, there is a diffeomorphism $F$ that sends $\gamma$ to $\gamma'$. If $d$ is a boundary condition for $\Emb(W,M)$ such that $[d]=\delta$, we have that the group $\Diff_\partial^+(W)$ acts on the space $\Emb(W,M;d)$ by precomposition. As $\dim M\geq 5$, by \cite{Haefliger} (cf. \eqref{eq:haefliger}) we have that $\pi_0(\Emb(W,M;d))\cong \pi_0(\map(W,M;d))$, and the latter is trivial because $M$ is contractible. Therefore, there exists a path of embeddings from the inclusion $i\colon W\subset M$ to the embedding $i\circ F$. Using the parametrised isotopy extension theorem we promote this isotopy of embeddings to an ambient isotopy $f_t$.
\end{proof}

\begin{figure}
\centering
\def\svgwidth{.6\textwidth}
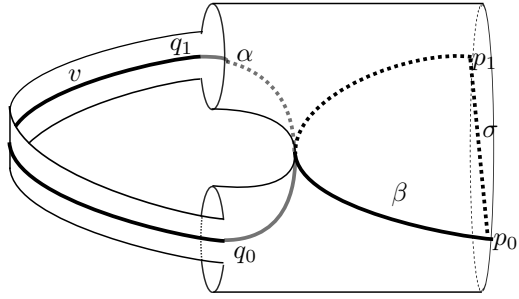
\caption{The curves $\vv$, $\alpha$, $\beta$, and $\sigma$ in $u'\cup P$.}
\label{fig:triv08}
\end{figure}

\begin{figure}
\centering
\def\svgwidth{.6\textwidth}
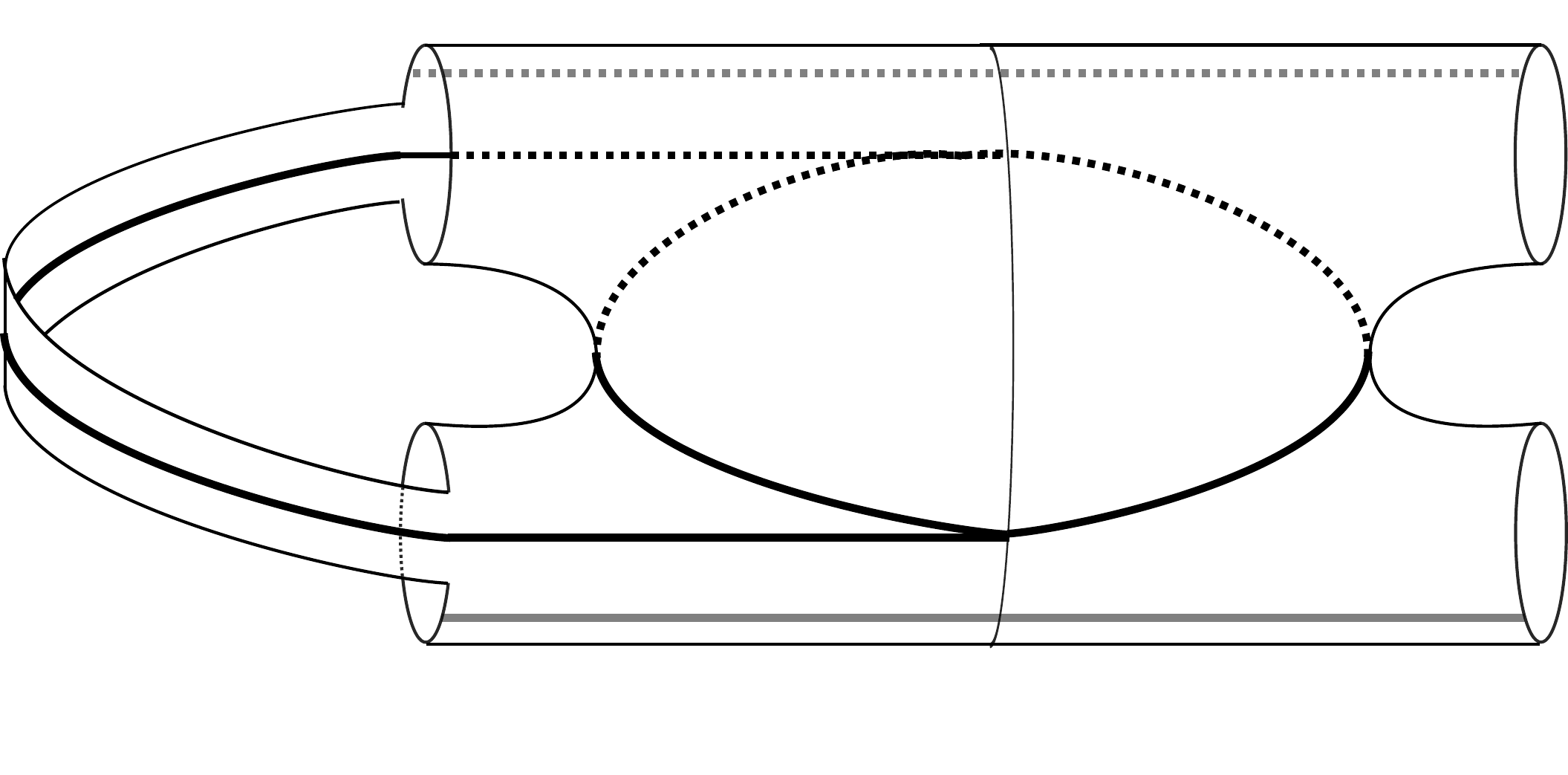
\caption{}
\label{fig:triv09}
\end{figure}

\begin{proof}[Proof of Lemma \ref{lemma:triviality}]

 Let us denote by $\varphi$ the image under $\z{u}'''$ of $\ppp D^+_{\frac{1}{2}}$ (i.e., the image of $\z{\vv}$). Let $D_{\varphi}$ be the image of $D^+_{\frac{1}{2}}$ under the embedding $\z{u}'''$, and let $\sigma$ be the image of $\pp D^+_{\frac{1}{2}}$, so that $D_{\varphi}$ is a disc in $\z{u}'$ that bounds $\varphi\cup \sigma$. Observe that the isotopy class of $\sigma$ is determined by $\ell$. Choose a section $s_\varphi$ of $S(N_{M}D_\varphi)$ that is trivial on $\varphi$ (i.e., whose restriction to $\varphi$ is the constant section with value $1$ of the trivialised bundle $S(N_W\gamma)$), and collared (i.e., as $\z{u}'''$ is collared, there is an $\epsilon$ for which $D_\varphi = \sigma\times (1-\epsilon,1]$, and a canonical identification $N_{\partial M\times (1-\epsilon,1]}{D_{\varphi}\cap \partial M\times (1-\epsilon,1]}\cong (N_\sigma \partial M)\times (1-\epsilon,1]$. Say that $s$ is collared if $s(x,t) = (s(x,0),t)$ under this identification).

Let $p_0$ and $p_1$ be the initial and final points of the path $\varphi$. As $p_0$ and $p_1$ are contained in $\ell\cap \z{\dd}$, they are points in the relevant boundary of $P$ (see Definition \ref{df:relevant}). Let $\beta$ be a path embedded in $P$, from $p_0$ to $p_1$, with a unique minimum at the unique critical point $p$ of $\pi\colon P\subset \pp M\times [0,1]\to [0,1]$, and such that $T_p\beta\subset T_pP$ is the unstable subspace. 

Now, $\ell$ may be chosen so that $\sigma$ is isotopic to $\beta$ in $\partial M\times [0,1]$, so suppose this has been done. 

Then $\beta$ is isotopic to $\varphi$, so there is a ball $B\subset u'\cup \pp M\times [0,1]$ containing them and $D_{\varphi}$. In addition, both are non-separating: as $w\cup P$ has only critical points of index $1$, the stable submanifold at the point $p$ connects both sides of the complement of $\varphi$ to the incoming boundary of $w\cup P$, which is connected. Regarding $\beta$, again both sides connect to the incoming boundary of $w\cup P$. By Lemma \ref{lemma:interchange}, there exists an isotopy $f_t$ of $u'\cup \pp M\times [0,1]$ supported on $B$ such that
\begin{enumerate}
\item $f_1(\ww\cup P) = \ww\cup P$,
\item $f_1(\varphi) = \beta$.
\end{enumerate}
Define $D_\beta = f_1(D_{\varphi})$ and $s_\beta = f_1(s_\varphi)$. We assume, without loss of generality, that $D_\beta$ has no critical points. Define $\mirror{P},\mirror{\beta},\mirror{D_\beta}\subset \partial M\times [1,2]$ to be the reflections of $P$, $\beta$ and $D_\beta$ along $\partial M\times \{1\}$, and orient $P$ compatibly with the orientation of $\z{\dd}$. Similarly, define the section $\mirror{s_\beta}$ to be the reflection of the section $s_\beta$. The union $\z{\ww}\cup \mirror{P}$ is a submanifold of $N$ satisfying the boundary condition $\xi$, but it is not a trivial cobordism.

The circle $\varphi\cup \mirror{\beta}\subset \z{\ww}\cup \mirror{P}$ is bounded by the disc $D_{\varphi}\cup \mirror{D_\beta}\subset N$, on which we have the section $s_\varphi\cup \mirror{s_\beta}$. Define 
\[Q:= (\z{\ww}\cup \mirror{P})\natural (D_{\varphi}\cup \mirror{D_\beta}).\] 

We now show that $P(u)\cup Q$ is isotopic to $\ww\cup \dd\times[1,2]$ (See Figure \ref{fig:triv09}). If we extend the isotopy $f_t$ of $u'\cup \pp M\times [0,1]$ to $u'\cup \pp M\times [0,2]$ by the identity, we have that: 
\begin{align*}
P(u)\cup Q &= P(u)\cup ((\z{\ww}\cup \mirror{P})\natural (D_{\varphi}\cup \mirror{D_\beta})) \\
&\simeq f_1^{-1}[(P(u)\cup \z{\ww}\cup \mirror{P})\natural (D_{\varphi}\cup \mirror{D_\beta})] \\
&= (f_1^{-1}[P(u)\cup \z{\ww}]\cup \mirror{P})\natural (f_1^{-1}[D_{\varphi}]\cup \mirror{D_\beta}) \\
&= (P(u)\cup \z{\ww}\cup \mirror{P})\natural (f_1^{-1}[D_{\varphi}]\cup \mirror{D_\beta}) \\
&= (\ww\cup P \cup \mirror{P})\natural (D_{\beta}\cup \mirror{D_\beta})\\
&= \ww\cup ((P \cup \mirror{P})\natural (D_{\beta}\cup \mirror{D_\beta}))
\end{align*}
But $\pi\colon P \cup \mirror{P}\subset \pp M\times[0,2]\to [0,2]$ has only two critical points which are cancelled by the surgery along $D_\beta\cup \mirror{D_\beta}$. Therefore $(P \cup \mirror{P})\natural (D_{\beta}\cup \mirror{D_\beta})$ has no critical points and is therefore a trivial cobordism. In addition, the part of this cobordism lying in $\pp M\times [1,2]$ is a reflection of the part lying in $\pp M\times [0,1]$, and therefore the union is homotopic to $\dd\times[0,2]$, and therefore isotopic to it too (because $\dim M\geq 5$, using again \cite{Haefliger} as in \eqref{eq:haefliger}). This gives the first isotopy of statement \ref{lemma:triviality:2} of the lemma.

Since $g_1(P(u)\cup Q)$ has a unique critical point and so does $P(u)$, it follows that $Q$ is a trivial cobordism. Hence $Q$ is as in statement \ref{lemma:triviality:1} of the lemma.

To show that $Q\cup (P+2)$ is isotopic to $\z{\ww}\cup \z{\delta}^0\times [1,3]$, we proceed as follows: Let $q_0$ and $q_1$ be the initial and final points of the path $v$.  Let $\alpha$ be a path embedded in $P$, from $q_0$ to $q_1$, with a unique minimum at the unique critical point $p$ of $\pi\colon P\subset \pp M\times [0,1]\to [0,1]$, and such that $T_p\alpha\subset T_pP$ is the stable subspace. 

Let $D_\alpha\subset \pp M\times [0,1]$ be a collared half $2$-disc transverse to $P$, bounding $\alpha\cup \rho$ for some path $\rho \subset \pp M$, and let $s_\alpha$ be a section of $S(N_M D_\alpha)$ trivialised over $\alpha$ and collared (as in the first paragraph of this proof). Assume that the intersection of $D_\alpha$ with $D_\beta$ consists only on the point $p$. The discs $D_{\varphi}\cup \mirror{D_\beta}$ and $\mirror{D_\alpha}\cup (D_\alpha +2)$ intersect in a single point, namely the reflection of $p$ in $\mirror{P}$, therefore there is a ball $B$ that contains both discs. In addition, both $\varphi\cup \beta$ and $\alpha\cup (\alpha +2)$ are non-separating in $\z{w}\cup \mirror{P}\cup (P+2)$: each curve is transverse to the other and they meet in a single point, therefore the possible two sides of the complement of each curve are connected by the other curve. By Lemma \ref{lemma:interchange}, there is an isotopy $h_t$ of $\z{u}'\cup \pp M\times [1,3]$ supported on $B$ such that:
\begin{enumerate}
\item $h_1(\z{w}\cup \mirror{P}\cup (P+2)) = \z{w}\cup \mirror{P}\cup (P+2)$ and
\item $h_1(\varphi\cup \mirror{\beta}) = \mirror{\alpha}\cup (\alpha+2)$
\end{enumerate}
Because $h_1(D_{\varphi}\cup \mirror{D_\beta})$ and $\mirror{D_\alpha}\cup (D_\alpha+2)$ are contained in the ball $B$ they are isotopic, so there exists another isotopy $h'_t$ of $\z{u}'\cup \pp M\times [1,3]$ supported on $B$ such that:
\begin{enumerate}
\item $h'_1(\z{w}\cup \mirror{P}\cup (P+2)) = \z{w}\cup \mirror{P}\cup (P+2)$ and
\item $h'_1(h_1(D_{\varphi}\cup \mirror{D_\beta})) = \mirror{D_\alpha}\cup (D_\alpha+2)$
\end{enumerate}
Therefore, endowing $\mirror{D_\alpha}\cup (D_\alpha+2)$ with the section $h'_1(h_1(s_\varphi\cup \mirror{s_\beta}))$, we have:
\begin{align*}
Q\cup (P+2) &= (\z{\ww}\cup \mirror{P})\natural(D_{\varphi}\cup \mirror{D_\beta}) \cup (P+2) \\
&\simeq h'_1\circ h_1[(\z{\ww}\cup \mirror{P})\natural(D_{\varphi}\cup \mirror{D_\beta}) \cup (P+2)] \\
&= (\z{\ww}\cup \mirror{P}\cup (P+2))\natural(h_1[D_{\varphi}\cup \mirror{D_\beta}]) \\
&= (\z{\ww}\cup \mirror{P}\cup (P+2))\natural(\mirror{D_\alpha}\cup (D_\alpha+2)) \\
&= \z{\ww}\cup ((\mirror{P}\cup (P+2))\natural(\mirror{D_\alpha}\cup (D_\alpha+2))) 
\end{align*}
But $\pi\colon \mirror{P} \cup (P +2)\subset \pp M\times[1,3]\to [1,3]$ has only two critical points which are cancelled by the surgery along $\mirror{D_\alpha}\cup (D_\alpha+2)$. Therefore $(\mirror{P} \cup (P+2))\natural (\mirror{D_\alpha}\cup (D_\alpha+2))$ has no critical points and is therefore a trivial cobordism. In addition, the part of this cobordism lying in $\pp M\times [2,3]$ is a reflection of the part lying in $\pp M\times [1,2]$, and therefore the union is homotopic to $\z{\dd}\times[1,3]$, and therefore isotopic (again, using \cite{Haefliger} as in \eqref{eq:haefliger}). This gives the second isotopy of statement \ref{lemma:triviality:2} of the lemma.
\end{proof}

\subsection{Zero in homology}\label{section:zero-in-homology} During this section, if $A\rightarrow X$ is a map, we will denote by $(X,A)$ its mapping cone. We will use the letter $\Sigma$ for unreduced suspension, and write $CX = [0,1] \times X / \{1\} \times X$.

In this section, we give a slightly different treatment of the problem of $1$-triviality as approached in the literature. The difference is that using the methods of \cite{R-WResolution2015} and \cite{Palmer:oriented}, one can prove that the composition $h\circ \Sigma \aa_{g-1,b}(u_0,u_1)$ in the first display of Proposition \ref{prop:22} is zero in homology, whereas we prove that it is in fact null-homotopic. This is done via the following lemma, which has the additional advantage of avoiding the use of maps that only exist at the level of homology.
 
\begin{lemma}\label{lemma:factor} If
\[\xymatrix{A\ar[r]^i\ar[d]^g& X \ar[d]^f \\ A'\ar[r]^j &X'}\]
is a map of pairs and there is a map $t\colon X\rightarrow A'$ making the bottom triangle commute up to a homotopy $H\colon f\simeq jt$, then the induced map between mapping cones $(f,g)\colon (X,A)\rightarrow (X',A')$ factors as $(X,A)\stackrel{p}{\rightarrow} CA \cup_i CX \stackrel{h}{\rightarrow} (X',A')$, where $p$ comes from the Puppe sequence. In addition, if there is also a homotopy $G\colon g\simeq ti$, then the composite $CA \cup_i CX \stackrel{h}{\rightarrow} (X',A')\stackrel{p'}{\rightarrow} CA' \cup_j CX'$ is nullhomotopic.
\end{lemma}
\begin{proof}
The map $h\colon CA\cup_i CX\rightarrow CA'\cup_j X'$ is given by 
\begin{align*}
h(a,s) &= (g(a),s) \in CA' &\text{if}~& (a,s)\in CA \\
h(b,s) &= H(b,2s) \in X' &\text{if}~ &(b,s)\in CX ~\text{and}~ 0\leq s\leq 1/2 \\
h(b,s) &= (t(b),2s-1)\in CA' &\text{if}~ &(b,s)\in CX ~\text{and}~ 1/2\leq s\leq 1,
\end{align*}
and it restricts to $(f,g)$ in the mapping cone $CA\cup_i X$, hence $hp = (f,g)$. 

For the second part, let $C_{\frac{1}{2}}Y = \{(y,s)\in CY\mid 0\leq s\leq 1/2\}$, notice that
$$(CA \cup_i CX) / C_{\frac{1}{2}}X \cong \Sigma A \vee \Sigma X,$$
and consider the diagram

\footnotesize
\[\begin{gathered}
\xymatrix@C=.4cm{
CA \cup CA = \Sigma A \ar[d]^{\mathrm{Id} \cup Ci}_\simeq\\
CA\cup_i CX \ar[r]^h\ar[d]^{\text{collapse } C_{\frac{1}{2}}X} & CA'\cup_j X'\ar[r]^{p'} & CA'\cup_j CX' \ar[rr]_-{\simeq}^-{\text{collapse }CX'}&& \Sigma A' \\
(CA\cup_i CX)/C_{\frac{1}{2}}X \cong \Sigma A \vee \Sigma X \ar[rrrr]^{\Sigma g \vee \Sigma t}&&&& \Sigma A'\vee \Sigma A'\ar[u]^\nabla
}
\end{gathered}\]
\normalsize
which is easily checked to commute. As $ti$ is homotopic to $g$, the lower composition is homotopic to $\nabla \circ (\Sigma g \vee \Sigma g) \circ \vee$, i.e.\ $\Sigma g - \Sigma g$, so it is nullhomotopic, as required.
\end{proof}

 We now return to \eqref{eq:YYY}, where we had chosen a $0$-simplex $u_0:=u \in \b{g,b}{M;\delta}{0}$. Suppose that we have another $0$-simplex $u_1\in \h{g,b-1}{M(u_0);\delta(u_0)}{0}$ (we can also consider $u_1$ as a point in $\b{g,b}{M;\delta}{0}$ and $u_0$ as a point in $\h{g,b-1}{M(u_1);\delta(u_1)}{0}$). We now construct the following diagram:

{\footnotesize
\[\xymatrix@C=.4cm{
\ar @{} [dr] |{(4)}
\e{g-1,b}{M(u_0,u_1)}\ar[r] \ar[d]_{\aa_{g-1,b}(u_0)} & \e{g,b-1}{M(u_0,u_1)} \ar[r]\ar[d]^{\ab_{g,b-1}(u_0)} & \alpha_{g-1,b}(M(u_0,u_1))\ar[r]\ar[d] & \Sigma\e{g-1,b}{M(u_0,u_1)}\ar[d]. \\ 
\ar @{} [dr] |{(1)}
\e{g,b-1}{M(u_0)}\ar[r] \ar[d]_{\ab_{g,b-1}(u_0)} & \e{g,b}{M(u_0)} \ar[r]\ar[d]^{\aa_{g,b}(u_0)} & \beta_{g,b-1}(M(u_0))\ar[r]^p\ar[d] & \Sigma\e{g,b-1}{M(u_0)} \ar[dl]^{h} \\
\ar @{} [dr] |{(2)}
\e{g,b}{M}\ar[r]& \e{g+1,b-1}{M} \ar[r] & \alpha_{g,b}(M)& \\
\ar @{} [dr] |{(3)}
\e{g,b-1}{M(u_1)}\ar[r] \ar[u]^{\ab_{g,b-1}(u_1)} & \e{g,b}{M(u_1)} \ar[r]\ar[u]_{\aa_{g,b}(u_1)} & \beta_{g,b-1}(M(u_1))\ar[u]\ar[r]^{p'} & \Sigma\e{g,b-1}{M(u_1)}\ar[ul]_{h'}\\
\e{g-1,b}{M(u_0,u_1)}\ar[r] \ar[u]^{\aa_{g-1,b}(u_0)} & \e{g,b-1}{M(u_0,u_1)} \ar[r]\ar[u]_{\ab_{g,b-1}(u_0)} & \alpha_{g-1,b}(M(u_0,u_1))\ar[r]\ar[u] & \Sigma\e{g-1,b}{M(u_0,u_1)} \ar[ul]_{h''}.\\ 
}\]
}

The third line of the diagram is the Puppe sequence for the stabilisation map $\alpha_{g,b}(M;\delta, \z{\delta})$. The second and fourth lines are the Puppe sequences for the approximate augmentations corresponding to the data $u_0$ and $u_1$ respectively. The first and fifth lines are the Puppe sequences for the approximate augmentation of $\OO{g,b}{M(u_0);\delta(u_0)}{\bullet}$ over the $0$-simplex $u_1$ and for the approximate augmentation of $\OO{g,b}{M(u_1);\delta(u_1)}{\bullet}$ over the $0$-simplex $u_0$. Note that the first and fifth lines are the same.

We first apply Lemma \ref{lemma:triviality} to the $0$-simplex $u_0 \in \OO{g,b}{M(u_1);\delta(u_1)}{0}$ in order to provide a cobordism
\[Q\subset N=\z{u}_0'\cup \pp M(u_1)\times [1,2]\]
and a diagonal map $-\cup Q$ that commutes up to homotopy in the square $(3)$. The cobordism
\[Q' := Q\cup \y{\ww}_1\subset N'=\z{u}_0'\cup \pp M\times [1,2]\] 
provides a diagonal map $-\cup Q'$ for the square $(1)$. As the isotopies constructed in that lemma for $Q$ where the identity on the boundary of $N$, they extend through the constant isotopy on $N'$, moreover, by construction of $Q$ and $Q'$ we have:

\begin{lemma}\label{lemma:commutative} The composition of the diagonal map of (3) with $\ab_{g-1,b}(u_1)$ is the same as the composition of $\aa_{g-1,b}(u_0)$ with the diagonal map of (1). The same holds for the homotopies. 
\end{lemma}
We continue by choosing any diagonal map and homotopies for the square $(2)$ using again Lemma \ref{lemma:triviality}, and construct the maps $p$, $h$, $p'$, $h'$, $p''$ and $h''$ applying the first part of Lemma 6.6 to the squares (1), (2) and (3).

The following finishes the proof of parts $(v)$ and $(vi)$ of Lemma \ref{lemma:induction}.
\begin{proposition}\label{prop:22} 
 If $X_{g-1}$ holds, then the map $(\beta_{g,b-1}(M(u_0)))\rightarrow (\alpha_{g,b}(M))$
induces the zero homomorphism in homology degrees $\leq \frac{1}{3}(2g+2)$. If $Y_{g-1}$ holds, then the map $(\alpha_{g,b-1}(M(u,v)))\rightarrow (\beta_{g,b}(M))$
induces the zero homomorphism in homology degrees $\leq \frac{1}{3}(2g+1)$.
\end{proposition}
\begin{proof}
We find homotopies
\[h\circ \Sigma\aa_{g-1,b}(u_0,u_1) \simeq (\aa_{g,b}(u_1),\ab_{g,b-1}(u_1))\circ h'' \simeq h'\circ p'\circ h''\simeq *\]
by first applying Lemma \ref{lemma:commutative} and then applying each part of Lemma \ref{lemma:factor}. Since $X_{g-1}$ holds, the map $\Sigma\aa_{g-1,b}(u_0,u_1)$ induces an epimorphism in homology degrees $\leq \frac{1}{3}(2(g-1)+1)+1$ (although the map $\aa_{g-1,b}(u_0,u_1)$ is not a map of type $\alpha_{g,b}(M(u_0,u_1)$, it is isotopic to such a map after rounding the corners of $M$), hence, in order to have that $h\circ \Sigma\aa_{g-1,b}(u_0,u_1)\simeq *$, the map $h$ must induce the zero homomorphism in those degrees, and so must $hp$, which is the map in question again by Lemma \ref{lemma:factor}. The second part is proven similarly, by rewriting all of this section.
\end{proof}

\section{Homological stability for closed surfaces}\label{last}\label{section:lastboundary}
We have to prove the last assertion of Theorem \ref{thm:stab}, as well as the injectivity in homology of the maps of type $\beta$ for which one of the newly created boundaries is contractible in $\partial M$. That $\beta_{g,b}(M;\delta)$ induces a monomorphism in homology in this case and that $\gamma_{g,b}(M;\delta)$ induces an epimorphism in homology if $b\geq 2$ is a consequence of the fact (cf. Remark \ref{remark-gamma-beta}) that for each such $\beta_{g,b}(M;\delta)$ (resp. each $\gamma_{g,b}(M;\delta)$), there is a $\gamma_{g,b+1}(M;\delta')$ (\resp $\beta_{g,b-1}(M;\delta')$) and homotopy retractions
\[\gamma_{g,b+1}(M;\delta')\beta_{g,b}(M;\delta)\simeq \Id,\quad \beta_{g,b-1}(M;\delta')\gamma_{g,b}(M;\delta)\simeq \Id.\]
Moreover, by Proposition \ref{prop:71} the $\beta$-maps induce isomorphisms in homology in degrees $\leq \frac{2}{3}g-1$ and an epimorphism in the next degree, but since the $\beta$ maps are monomorphisms, it follows that they also induce isomorphisms in homology up to degree $\frac{2}{3}g$. Finally, this implies that $\gamma_{g,b}(M;\delta)$ is an isomorphism in those degrees too. 

Therefore, it only remains to prove the third assertion of Theorem \ref{thm:stab} when $b=1$. This is the purpose of this section.

\subsection{Resolutions and fibrations}
Consider the space $\e{g,b}{M;\delta}$, and let $\ell\subset \pp M$ be a subset diffeomorphic to a ball disjoint from $\delta$. There is a semi-simplicial space $\p{g,b}{M;\delta,\ell}{\bullet}$ whose $i$-simplices are tuples $(W,p_0,\ldots,p_i)$, where  $p_j=(p_j',p_j'',p_j''')$ and
\begin{enumerate}
\item $W\in \e{g,b}{M;\delta}$;
\item\label{cond1} $p_j'''\colon ([0,1],\{1/2\})\rightarrow (M,W)$ is an embedding of pairs with $p_j'''(0)\in \ell$ and~$p_j'''(1)\in \mathring{M}$;
\item $(p_j',p_j'')$ is a closed tubular neighbourhood of $p_j'''([0,1])$ in the pair $(M,W)$;
\item the neighbourhoods $p_0',\ldots,p_i'$ are disjoint.
\end{enumerate}
The $j$th face map forgets $p_j$ and there is an augmentation map to $\e{g,b}{M;\delta}$ that forgets all the $p_j$. We topologise the space of $i$-simplices as a subset of $\e{g,b}{M;\delta}\times \TubEmb(I\times[i],M;q,q_N)$, where
\[q(x) = q_N(x) = \ell ~\text{if}~ x = 0,\qquad q(x) = q_N(x) = M ~\text{if}~ x\neq 0.\]

\begin{proposition}\label{prop:resolution-paths} If $M$ is connected and of dimension at least $3$, then the semi-simplicial space $\p{g,b}{M;\delta}{\bullet}$ is a resolution of $\e{g,b}{M;\delta}$. 
\end{proposition}
\begin{proof}
 The space $\e{g,b}{M;\delta}$ is $\Diff_\partial(M)$-locally retractile by Proposition \ref{prop:retractil-embeddings}, and therefore it is also $\Diff(M;\delta,\ell)$-locally retractile. For each $i$, the augmentation map \[\epsilon_i\colon \p{g,b}{M;\delta}{i}\lra \e{g,b}{M;\delta}\] is $\Diff(M;\delta,\ell)$-equivariant for all $i$, therefore it is also a locally trivial fibration by Lemma \ref{lemma:retractil-target}. As a consequence, the semi-simplicial fibre $\Fib_W(\epsilon_\bullet)$ is homotopy equivalent to the homotopy fibre of $\vert\epsilon_\bullet\vert$ by Criterion \ref{criterion1}. The space of $i$-simplices of the semi-simplicial fibre is $\TubEmb((I\times [i],\{1/2\}\times [i]),(M,W);q,q_N)$. Let us define the semi-simplicial space $P(W,M)_\bullet$ whose space of $i$-simplices is 
$$\Emb((I\times [i],\{1/2\}\times [i]),(M,W);q),$$ and the face maps are given by forgetting embeddings.
Forgetting the tubular neighbourhoods yields a map
\[\label{ejemplo}r_\bullet\colon \Fib_W(\epsilon_\bullet)\longrightarrow P(W,M)_\bullet\]
that is levelwise $\Diff(M;W,\ell)$-equivariant onto the space $P(W,M)_\bullet$, which is levelwise $\Diff(M;W,\ell)$-locally retractile by Corollary \ref{cor:retractil-resolutions}, hence this map is a levelwise fibration by Lemma \ref{lemma:retractil-target}. The fibre of the map $r_i$ over an $i$-simplex ${\bf p} = (p_0''',\ldots,p_i''')$ is the space $\Tub({\bf p}'''(I),(M,W);q_N)$, which is contractible by Lemma \ref{lemma:tubular-contractible}.

The semi-simplicial space $P(W,M)_\bullet$ is a topological flag complex, and we will apply Criterion \ref{criterion2} to prove that it is contractible. As $M$ is connected, for each tuple $(W,p_0'''),\ldots,(W,p_{i-1}''')$ of $0$-simplices over a surface $W$ there is another $0$-simplex $(W,p_i''')$ over $W$ orthogonal to them all, by general position. Hence it is contractible by Criterion \ref{criterion2}.
\end{proof}

Let $B_i(M;\ell)$ be the set of tuples $(p_0,\ldots,p_i)$ with $p_j = (p_j',p_j'',p_j''')$, where $p_j'''$ is an embedding of an interval in $M$, $p_j'$ is a tubular neighbourhood of $p_j''$ in $M$ and $p_j''$ is the restriction of $p_j'$ to some vector subspace $L_j\subset \Ncl_M p'''_j(1/2)$ of dimension $2$. Moreover, we require that $p_j'$ be disjoint from $p_k'$. This space is in canonical bijection with $\TubEmb_{2,\{1/2\}\times [i]}(I\times [i],M;q,q_N)$, and we use this bijection to topologise it.

There is a map 
\[\p{g,b}{M;\delta,\ell}{i}\longrightarrow B_i(M;\ell)\]
that sends a tuple $(W,p_0,\ldots,p_i)$ to the tuple $(p_0,\ldots,p_i)$. 

\begin{proposition}\label{prop:fibrations-paths} For a point ${\bf p} \in B_i(M;\ell)$, with the notation of Section \emph{\ref{ss:gluingmaps}}, there is a homotopy fibre sequence
\[\e{g,b+i+1}{M({\bf p});\delta({\bf p})} \longrightarrow \p{g,b}{M;\delta}{i}\lra B_i(M;\ell).\] 
 \end{proposition}
\begin{proof}
The map is $\Diff(M;\delta,\ell)$-equivariant and $B_i(M;\ell)$ is $\Diff(M;\delta,\ell)$-locally retractile by Proposition \ref{prop:retractil-fibrations}, therefore this map is a locally trivial fibration by Lemma \ref{lemma:retractil-target}. The fibre over a point ${\bf  p}$ is the space of surfaces $W$ in $M$ that meet the tubular neighbourhoods $p_j'$ in the image of $p_j''$. This space is canonically homeomorphic to the space $\e{g,b+i+1}{M({\bf p});\delta(\bf p)}$.
\end{proof}

\subsection{Stabilisation maps between resolutions} In this section we will show how to extend the stabilisation map $\gamma_{g,b}(M;\delta)$ to a map between resolutions
\[\xymatrix{
\p{g,b}{M;\delta,\ell}{i}  \ar@{-->}[rr]\ar[d]                           &&   \p{g,b-1}{M_1;\z{\delta},\z{\ell}}{i} \ar[d]   \\
\e{g,b}{M,\delta}    \ar[rr]^{\gamma_{g,b}(M_1;\delta,\z{\d})} &&   \e{g,b-1}{M_1,\z{\delta}}.
}\]
To define the maps $\gamma_{g,b}(M;\delta,\z{\d})\colon \e{g,b}{M;\delta}\rightarrow \e{g,b-1}{M_1,\z{\d}}$, we joined each surface with a cobordism $P$ in $\pp M\times I$. We will assume, without loss of generality, that 
\begin{enumerate}
\item $\z{\ell} = \ell\times\{1\};$
\item $(\ell\times I)\cap P = \emptyset$.
\end{enumerate}
As in previous constructions, we define $\y{p}_j' = \pp p_j'\times I$ and $\z{p}_j' = p'_j\cup \y{p}'_j$ and similarly $\y{p}'_j$ and $\z{p}'_j$. There is a map $\gamma_{g,b}(M;\delta,\z{\d})_i$ making the diagram commute, that sends a tuple $(W,{\bf p})$ to the tuple $(W\cup P,{\bf \z{p}})$. These maps commute with the face maps and with the augmentation maps, so they define a map  of semi-simplicial spaces (the \emph{resolution of $\gamma_{g,b}(M;\delta,\z{\delta})$}),
\[\gamma_{g,b}(M;\delta,\z{\delta})_{\bullet}\colon \p{g,b}{M;\delta}{\bullet}\longrightarrow \p{g,b-1}{M_1;\z{\delta}}{\bullet},\]
extending $\gamma_{g,b}(M;\delta,\z{\delta})$.

\begin{corollary}[To Proposition \ref{prop:resolution-paths}]\label{relative-resolution-paths} The pair $(\gamma_{g,b}(M)_\bullet,\delta,\z{\delta})$ together with the natural augmentation to the pair $\gamma_{g,b}(M;\delta,\z{\delta})$ is a resolution.
\end{corollary}
The diagram
\[
\xymatrix{
\p{g,b}{M;\delta}{i}\ar[rrr]^-{\gamma_{g,b}(M;\delta,\z{\delta})_i} \ar[d]&&& \p{g,b-1}{M;\z{\delta}}{i}\ar[d] \\
B_i(M;\ell)\ar[rrr]^-{{\bf p}\mapsto {\bf\z{p}}} &&& B_i(M_1,\ell)
}
\]
is an extension of the homotopy equivalence $B_i(M;\ell)\rightarrow B_i(M_1;\z{\ell})$. Hence we obtain a well-defined map on the homotopy fibres over the points ${\bf p}$ and ${\bf \z{p}}$ of the fibrations of Proposition \ref{prop:fibrations-paths},
\[\e{g,b+i+1}{M({\bf p});\d({\bf p})}\longrightarrow \e{g,b+i}{M({\bf \z{p}});\z{\delta}({\bf \z{p}})},\]
obtained by gluing the cobordism $P$ to each surface. This is a map of type $\gamma_{g,b+i+1}(M({\bf p});\delta({\bf p}),\z{\delta}({\bf \z{p}}))$. 

\begin{corollary}[To Proposition \ref{prop:fibrations-paths}]\label{relative-fibrations-paths} There is a relative homotopy fibre sequence
\[(\gamma_{g,b+i+1}(M({\bf p});\delta({\bf p}),\z{\delta}({\bf \z{p}})))\longrightarrow (\gamma_{g,b}(M;\delta,\z{\delta})_i) \longrightarrow B_i(M;\ell).\]
\end{corollary}

\subsection{Homological stability} 

The following proposition finishes the proof of Theorem \ref{thm:stab}.

\begin{proposition}\label{prop:106} Let $M$ be a simply-connected manifold of dimension at least $5$ with non-empty boundary, and $\delta$ be a boundary condition. Then, for any boundary condition~$\z{\delta}$,
\begin{enumerate}
\item for every map $\gamma_{g,b}(M;\delta,\z{\delta})$ we have $H_k(\gamma_{g,b}(M;\delta,\z{\delta})) = 0$ for $k\leq \frac{2}{3}g+1$;
\item every map $\beta_{g,b}(M;\delta,\z{\delta})$ for which one of the newly created components of $\z{\delta}$ is contractible in $\partial^0M$ induces a monomorphism in all homology degrees;
\item every map $\gamma_{g,b}(M;\delta,\z{\delta})$ for which there is another component of $\delta$ in the same component of $\partial^0M$ as the one which is closed induces an epimorphism in all homology degrees.
\end{enumerate}
\end{proposition}

\begin{remark}\label{remark-gamma-beta} Consider a stabilisation map $\gamma_{g,b}(M;\delta,\z{\d})$, which is given by closing off one of the boundaries $b$ of $\delta$ (which must necessarily be nullhomotopic in $\partial^0M$). If $\delta$ has another boundary component $b'$ in the same component of $\partial^0 M$ as $b$, then there exists a stabilisation map $\beta_{g,b-1}(M;\delta_0,\delta)$ creating the boundaries $b$ and $b'$. In this case we may enlarge collars, and we have the composition
\[
\xymatrix{
\e{g,b-1}{M;\delta_0}\ar[rr]^-{\beta_{g,b-1}(M;\delta_0,\delta)} && \e{g,b}{M_1;\delta}\ar[rr]^-{\gamma_{g,b}(M_1;\delta,\z{\delta})} && \e{g,b-1}{M_2;\z{\delta}}
}
\]
which is homotopic to a stabilisation map which takes the union with a cylinder inside $\partial^0M \times [0,2]$. This map may not be homotopic to the identity ---the cylinder may be embedded in a non-trivial way--- but it is a homotopy equivalence (as we may find an inverse cylinder), so the map $\gamma_{g,b}(M_1;\delta,\z{\delta})$ is split surjective in homology. By the same argument, any map $\beta_{g,b}(M_1;\delta,\z{\delta})$ which creates a boundary which is nullhomotopic in $\partial^0M$ is split injective in homology.
\end{remark}

\begin{proof}[Proof of Proposition \ref{prop:106}]
We have already shown the last two statements above. Regarding the first statement, suppose first that there is another component of $\delta$ in the same component of $\partial^0M$ as the one which is closed by $\gamma_{g,b}(M;\delta,\z{\delta})$, and choose a $\beta_{g,b-1}(M;\delta_0,\delta)$ as in Remark \ref{remark-gamma-beta}. By Proposition \ref{prop:71}, we know that the map $\beta_{g,b-1}(M;\delta_0,\delta)$ induces an epimorphism in homology degrees $\leq \frac{2}{3}g$, and it also induces a monomorphism in all degrees: thus it induces an isomorphism in degrees $\leq \frac{2}{3}g$. As $\gamma_{g,b}(M;\delta,\z{\delta})$ is a left inverse to it, this also induces an isomorphism in these degrees. Hence $H_k(\gamma_{g,b}(M;\delta,\z{\delta})) = 0$ for $k\leq \frac{1}{3}(2g+3)$, as $\gamma_{g,b}(M;\delta,\z{\delta})$ induces an epimorphism in all degrees.

Now suppose that there is no other component of $\delta$ in the same component of $\partial^0M$ as the one which is closed by $\gamma_{g,b}(M;\delta,\z{\delta})$. We choose a ball $\ell \subset \partial^0 M$ and form the resolution of $\gamma_{g,b}(M;\delta,\z{\delta})$ given by Corollary \ref{relative-resolution-paths}. Using Corollary \ref{relative-fibrations-paths} to identify the space of $i$-simplices in this resolution, the pair $(\gamma_{g,b+i+1}(M({\bf p});\delta({\bf p}))$ is a map of type $\gamma$ for surfaces with (after rounding the corners of $M$) at least $i+1$ extra boundary components of $\delta({\bf p})$ in the component of $\partial M$ containing the boundary which is closed off, so the discussion above applies and shows that $H_k(\gamma_{g,b+i+1}(M({\bf p});\delta({\bf p}))) = 0$ for $k\leq \frac{1}{3}(2g+3)$. Applying the second result of Criterion \ref{criterion3} to this resolution gives the result.
\end{proof}

\section{Stable homology of the space of surfaces in a manifold with boundary}\label{section:stablehomology}
In these last three sections we prove Theorem \ref{thm:MainH}. In Section \ref{section:stablehomology} we prove the theorem in the case where $M$ has non-empty boundary, except for the proof of a proposition (Proposition \ref{prop:Surgery}), which we defer to the Section \ref{section:surgery}. We show how to deduce Theorem \ref{thm:Main} for manifolds without boundary in Section \ref{section:closed}. In these three sections we only work with manifolds and manifolds with boundary, but not manifolds with corners, as we did in the previous sections.

This section gravitates around a group completion argument that takes place in Proposition \ref{prop:group-completion}. Roughly speaking, the space of all compact connected oriented surfaces $\coprod_{g,\delta}\e{g,b}{M;\delta}$ in $M$ is a module over the cobordism category of cobordisms in $\partial M\times I$, and the group completion will invert the operation ``gluing a torus in $\partial M\times I$''. We will compare the homotopy type of these modules with the homotopy type of certain spaces of sections in order to deduce Theorem \ref{thm:MainH} for background manifolds with non-empty boundary.

\subsection{Spaces of manifolds and scanning maps}

Let $M$ be a smooth manifold of dimension $d$, possibly with boundary. Recall from \cite[Definition~2.1]{GR-W} and \cite[Section~3]{R-WEmbedded} that the set $\Psi(M)$ of all smooth oriented $2$-dimensional submanifolds of $\mathring{M}$ which are closed as subsets of $M$ can be endowed with a topology. 

More generally, for any real vector space $V$ we can define the space $\Psi(V)$ of smooth oriented $2$-dimensional manifolds in $V$. If $\langle-,-\rangle$ is an inner product on $V$, there is a space $\Ss(V)$ as in page \pageref{def:ssv} in the introduction, and an inclusion
\begin{equation}\label{eq:SVinc}
i : \Ss(V) \lra \Psi(V)
\end{equation}
given by sending a pair $(L \in \Gr_2^+(V), v \in L^\perp)$ to the oriented 2-manifold $v + L \subset V$, and sending the point at infinity to the empty manifold.

\begin{proposition}[\cite{GR-W}]\label{prop:psiR}
The inclusion $i$ is a weak homotopy equivalence.
\end{proposition}

Let $\gg$ be a Riemannian metric on $M$, not necessarily complete. There is an associated partially defined exponential map $\exp \colon TM \dashrightarrow M$. The \emph{injectivity radius} of $\gg$ at $p \in M$ is the supremum of the real numbers $r \in (0,\infty)$ such that $\exp$ is defined on $T_pM$ on vectors of length $< r$, and $\exp$ is injective when restricted to the open ball of radius $r$ in $T_pM$.

Let ${a}\colon M\rightarrow (0,\infty)$ be a smooth map whose value at each point is strictly less than the injectivity radius of the metric $\gg$ at that point ---such functions exist by a partition of unity argument. If $V$ is an inner product space, define an endomorphism $h$ of $V$ by $v \mapsto \left(\frac{1}{\pi}\arctg||v||\right)v$. Let $\exp_a\colon TM\rightarrow M$ be the composition of the endomorphism of $TM$ given by $v\mapsto a(p)h(v)$ if $v\in T_pM$ and the exponential map.

Let $\Psi(TM)$ denote the space of pairs $(p,W)$ with $p\in M$ and $W\in \Psi(T_pM)$, i.e., the space obtained by performing the construction $\Psi(-)$ fibrewise to $TM$. There is a map 
\[\Psi(M)\times M\longrightarrow \Psi(TM),\]
given by $(W,p)\mapsto ((\exp_{a}\vert_{T_pM})^{-1}(W) \subset T_pM)$, whose adjoint
\[s_a\colon \Psi(M)\longrightarrow \Gamma(\Psi(TM)\rightarrow M),\]
a map to the space of sections of the bundle $\Psi(TM)\rightarrow M$, is called the \emph{non-affine scanning map}.

\begin{proposition}[\cite{R-WEmbedded}]\label{prop:embedded-61}  If $M$ has no compact components, then the non-affine scanning map $s_a$ is a weak homotopy equivalence.
\end{proposition}

\subsection{Scanning maps with boundary conditions}\label{sec:ScanningWithBdy}

We will also often need scanning maps when $M$ has a boundary, and surfaces are required to satisfy a boundary condition, as in Section \ref{ss:spaces-of-manifolds}. We formalise this as follows. 

Let $M$ be a manifold with boundary, $c \colon (-1,0] \times \partial M \to M$ be a collar, and $\xi \subset \partial M$ be a compact oriented 1-manifold. Write $M(\infty) = M \cup_{\partial M}([0,\infty) \times\partial M)$ for the manifold obtained by attaching an infinite collar to $M$. Let also
$$\Psi(M;\xi) := \{W \in \Psi(M(\infty)) \,\,\vert\,\, W \cap ((-1,\infty) \times\partial M) = (-1,\infty) \times \xi\}.$$
By choosing a Riemannian metric $\gg$ on $M(\infty)$ (which is a product on $(-1,\infty) \times\partial M$) and a function $a$ as above, we obtain a scanning map $s_a$ for $M(\infty)$. If the function $a$ is chosen so that
$$\mathrm{exp}_a(TM(\infty)\vert_{[0,\infty) \times\partial M}) \subset (-1,\infty) \times\partial M,$$
then for $W \in \Psi(M;\xi)$ the section $s_a(W)$ is a product when restricted to $[0,\infty) \times\partial M$, and is independent of $W$. By a slight abuse of notation, we call this product section $s_a([0,\infty)\times \xi)$, and write $\Gamma(\Psi(TM) \to M ;s_a(\xi))$ for the space of sections of $\Psi(TM)\to M$ which agree with $s_a([0,\infty)\times \xi)\vert_{\partial M}$ over $\partial M$. In this case there is a scanning map
\begin{equation}\label{eq:ScanningMap}
s_a \colon \Psi(M;\xi) \lra \Gamma(\Psi(TM) \to M ;s_a(\xi))
\end{equation}
by construction. As in Proposition \ref{prop:embedded-61}, if $M$ has no compact components then this scanning map is  a weak homotopy equivalence.

\subsection{Adding tails to \boldmath$M$}

Suppose that $M$ is a \emph{compact} manifold with collared boundary, and let $N, L\subset \partial M$ be open codimension 0 submanifolds, with $L$ diffeomorphic to a ball.

\begin{df}
We define the following subspaces of $\partial M \times [0,\infty)$:
$$N_{[a,b]} := N \times [a,b], \quad\quad L_{[a,b]} := L \times [a,b],$$
and we also write $N_{[0,\infty)} = N \times [0,\infty)$ and $N_a = N_{[a,a]}$, and similarly for $L$.

We then write
$$M_{a,b} := M \cup N_{[0,a]} \cup L_{[0,b]},$$
and let $M_{a, \infty}$, $M_{\infty, b}$ or $M_{\infty, \infty}$ have their obvious meaning.
\end{df}

Note that the boundary component $N_a \subset M_{a,b}$ has a canonical collar inside $((-1,0] \times \partial M) \cup_{\partial M} (N \times [0,\infty))$; similarly for the boundary component $L_b \subset M_{a,b}$.

For $\delta \subset N$ and $\xi \subset L$ compact oriented 1-manifolds, we define
$$\Psi(M_{a,b};\delta,\xi) := \Psi(M_{a,b};\delta \cup \xi) \subset \Psi(M_{\infty, \infty}),$$
as in Section \ref{sec:ScanningWithBdy}. A careful examination of the topology of $\Psi(M_{\infty,\infty})$ shows that $\Psi(M_{a,b};\delta,\xi)$ is homeomorphic to the disjoint union $\coprod_{[\SS]}{\E(\SS,M_{a,b};\delta\cup\xi)}$ where $[\SS]$ runs along a set of compact oriented surfaces with boundary diffeomorphic to $\delta\cup\xi$, one in each diffeomorphism class.

\subsection{Semi-simplicial models}

In order to show that the map \eqref{eq:ScanningMap} induces an isomorphism in homology in a range of degrees, we will pass through certain auxiliary semi-simplicial spaces.

\begin{df}
For $\xi \subset L_b$ a boundary condition, let $D(M_{\infty,b}; \xi)_p$ be the set of tuples $(a_0, a_1, \ldots, a_p, W)$ where
\begin{enumerate}
\item $0 < a_0 < a_1 < \cdots < a_p$ are real numbers;
\item $W \in \Psi(M_{\infty,b};\xi)$ is a surface satisfying the boundary condition $\xi$, and the $a_i$ are regular values for the projection $p_W \colon W \cap N_{[0,\infty)} \to [0,\infty)$.
\end{enumerate}
We give it the subspace topology from $(\bR^\delta)^{p+1} \times \Psi(M_{\infty,b};\xi)$. The collection of all the spaces $D(M_{\infty,b}; \xi)_p$ for $p\geq 0$ forms a semi-simplicial space, where the $j$th face map is given by forgetting $a_j$, and it is augmented over $\Psi(M_{\infty,b};\xi)$.
\end{df}

\begin{df}
Let ${D}(N_{(0,\infty)})_p$ be the set of tuples $(a_0, a_1, \ldots, a_p, W)$ where
\begin{enumerate}
\item $0 < a_0 < a_1 < \cdots < a_p$ are real numbers;
\item $W \in \Psi(N_{(0,\infty)})$ and the $a_i$ are regular values for the projection $p_W \colon W \to [0,\infty)$.
\end{enumerate}
We topologise this as a subspace of $(\bR^\delta)^{p+1} \times \Psi(N_{(0,\infty)})$. The collection of all these spaces forms a semi-simplicial space where the $j$th face map forgets $a_j$. It is \emph{not} augmented.

Let $\widehat{D}(N_{(0,\infty)})_p$ be the quotient space of ${D}(N_{(0,\infty)})_p$ by the relation
$$(a_0, a_1, \ldots, a_p, W) \sim (a'_0, a'_1, \ldots, a'_p, W')$$
if $a_j=a'_j$ for all $j$ and $p_W^{-1}([a_0, \infty)) = p_{W'}^{-1}([a_0, \infty))$. These form again a semi-simplicial space by forgetting the $a_j$.
\end{df}

There is a semi-simplicial map
$$\pi \colon D(M_{\infty,b}; \xi)_\bullet \lra \widehat{D}(N_{(0,\infty)})_\bullet$$
given by sending a tuple $(a_0, a_1, \ldots, a_p, W)$ to $[a_0, a_1, \ldots, a_p, W \cap N_{(0, \infty)}]$, which factors through the quotient map $r\colon {D}(N_{(0,\infty)})_\bullet \to \widehat{D}(N_{(0,\infty)})_\bullet$. In addition to these semi-simplicial spaces, we require another pair with stricter requirements.

\begin{df}\label{defn:Dpartial}
Let $D_\partial(M_{\infty,b}; \xi)_\bullet \subset D(M_{\infty,b}; \xi)_\bullet$ be the sub-semi-simplicial space where, in addition,
\begin{enumerate}
\item $W \cap M_{a_0, b}$ is connected, and
\item each pair $(p_W^{-1}[a_i, a_{i+1}], p_W^{-1}(a_i))$ is connected.
\end{enumerate}

Similarly, let $\widehat{D}_\partial(N_{(0,\infty)})_\bullet \subset \widehat{D}(N_{(0,\infty)})_\bullet$ be the sub-semi-simplicial space where, in addition, each pair $(p_W^{-1}[a_i, a_{i+1}], p_W^{-1}(a_i))$ is connected. As before, there is a semi-simplicial map $\pi_\partial \colon D_\partial(M_{\infty,b}; \xi)_\bullet \to \widehat{D}_\partial(N_{(0,\infty)})_\bullet$ given by restriction.
\end{df}

If $\Sigma \subset L_{[b,c]}$ is a surface satisfying the boundary condition $\xi \subset L_b$ and the boundary condition $\xi' \subset L_c$, we obtain a semi-simplicial map
$$- \cup \Sigma \colon D(M_{\infty,b}; \xi)_\bullet \lra D(M_{\infty,c}; \xi')_\bullet$$
over $\pi$, and if $(\Sigma, \Sigma \cap L_b)$ is connected then we also obtain a semi-simplicial map
$$- \cup \Sigma \colon D_\partial(M_{\infty,b}; \xi)_\bullet \lra D_\partial(M_{\infty,c}; \xi')_\bullet$$
over $\pi_\partial$.

\subsection{Proof of Theorem \ref{thm:MainH} when \boldmath$\partial M \neq \emptyset$}

Let us choose once and for all a surface $\Sigma \subset L \times [0,3]$ which satisfies the boundary condition $\xi \subset L$ at both ends (with respect to the obvious collars), is connected, and has positive genus. We define
$$D(M_{\infty,\infty}; \xi)_\bullet := \colim_{b \to \infty} D(M_{\infty,b}; \xi)_\bullet,$$
where the colimit is formed using the maps 
\[-\cup \Sigma \colon D(M_{\infty,b}; \xi)_\bullet \lra D(M_{\infty,b+3}; \xi)_\bullet.\] 
We define $D_\partial(M_{\infty,\infty}; \xi)_\bullet$ in the same way. Similarly, we define
$$\Psi(M_{\infty, \infty};\xi) := \colim_{b\to\infty} \Psi(M_{\infty, b};\xi),$$
where the maps in the colimit are again given by union with $\Sigma$.

There is a commutative diagram
\begin{equation}\label{eq:FundamentalDig}
\begin{gathered}
\xymatrix{
D_\partial(M_{\infty,\infty}; \xi)_\bullet \ar[r]\ar[d]^{\pi_\partial} & D(M_{\infty,\infty}; \xi)_\bullet \ar[d]^\pi \ar@{=}[r] & D(M_{\infty,\infty}; \xi)_\bullet \ar[d] \ar[r]^{\epsilon_\bullet} & \Psi(M_{\infty, \infty};\xi) \ar[d]\\
\widehat{D}_\partial(N_{(0,\infty)})_\bullet \ar[r] & \widehat{D}(N_{(0,\infty)})_\bullet & {D}(N_{(0,\infty)})_\bullet \ar[l]_{r_\bullet} \ar[r]^{\epsilon_\bullet}& \Psi(N_{(0,\infty)})
} 
\end{gathered}
\end{equation}
which we will use to compare the leftmost and rightmost vertical maps after geometric realisation. The first step in doing so is the following.

\begin{lemma}
The map $r_\bullet$ is a levelwise weak homotopy equivalence, and the two augmentation maps labelled $\epsilon_\bullet$ are weak homotopy equivalences after geometric realisation.
\end{lemma}
\begin{proof}
The map $r_\bullet$ can be treated with the techniques of \cite[Theorem 3.9]{GR-W}, and the two augmentation maps can be treated with the techniques of \cite[Theorem 3.10]{GR-W}.
\end{proof}

The second step in comparing the leftmost and rightmost vertical maps of \eqref{eq:FundamentalDig} is to show that the unlabelled horizontal maps are weak homotopy equivalences after geometric realisation. This is much more complicated and it is deferred to Section \ref{section:surgery}, although we state the result here.

\begin{proposition}\label{prop:Surgery}
The maps
$$\vert D_\partial(M_{\infty,\infty}; \xi)_\bullet\vert \lra \vert D(M_{\infty,\infty}; \xi)_\bullet\vert \quad\text{and}\quad \vert \widehat{D}_\partial(N_{(0,\infty)})_\bullet\vert \lra \vert \widehat{D}(N_{(0,\infty)})_\bullet\vert$$
are weak homotopy equivalences.
\end{proposition}

Before moving on to the proof of this proposition, let us show how we will apply it. We choose a Riemannian metric $\gg$ on $M_{\infty,\infty}$, an $a_0 \in (0,\infty)$, and a function $a \colon M_{\infty,\infty} \to (0,\infty)$ bounded above by the injectivity radius, and so that $\exp_a(TM_{\infty,\infty}\vert_{N_{[a_0,\infty)}}) \subset N_{(0,\infty)}$. The non-affine scanning map gives the following commutative diagram:
\begin{equation}\label{eq:FundamentalDig2}
\begin{gathered}
\xymatrix{
\Psi(M_{\infty, b};\xi) \ar[r]^-{s_a} \ar[d]& \Gamma(\Psi(TM_{\infty, b}) \to M_{\infty, b}; s_a(\xi)) \ar[d]^{\Pi_{a_0, b}} \\
\Psi(N_{(0,\infty)}) \ar[r]^-{s_a}& \Gamma(\Psi(TN_{[a_0,\infty)}) \to N_{[a_0,\infty)}),
}
\end{gathered}
\end{equation}
where both vertical maps are given by restriction. By Proposition \ref{prop:embedded-61}, the two non-affine scanning maps are weak homotopy equivalences. (For the lower one, we must use that the restriction map
$$\rho \colon \Gamma(\Psi(TN_{(0,\infty)}) \to N_{(0,\infty)}) \lra \Gamma(\Psi(TN_{[a_0,\infty)}) \to N_{[a_0,\infty)})$$
is an equivalence, and that if we choose a different function $a'$ bounded above by the injectivity radius of $\gg\vert_{N_{(0,\infty)}}$, then the functions $s_a$ and $\rho \circ s_{a'}$ are homotopic.)

Finally, as $N_{[a_0,\infty)} \hookrightarrow M_{\infty, b}$ is a cofibration, the rightmost vertical map is a fibration, so its homotopy fibre over a section $f$ is equivalent to
$$\Gamma(\Psi(TM_{a_0, b}) \to M_{0, b}; f\vert_{N_{a_0}}, s_a(\xi)).$$
The following group completion argument lets us understand the homotopy fibre of $\vert \pi_\partial\vert$. Recall that a map $f\colon X\to Y$ is a \emph{homology fibration} if for each point $y\in Y$, the natural map $\fib(y)\to \hofib(y)$ to the homotopy fibre is a homology equivalence, that is, induces isomorphisms in homology groups.
 
\begin{proposition}\label{prop:group-completion}
If $x=(a_0, W) \in \widehat{D}_\partial(N_{(0,\infty)})_0$ then the fibre\mmnote{fc: I think this is homotopy equivalent to the actual fibre. The actual fibre has as boundary condition the jet of $W$ at $p_W^{-1}(a_0)$.\\orw: ok, but I think saying this is fine.} of $\vert \pi_\partial\vert$ over $x$ is
$$F(a_0, W) := \colim_{b\to\infty} \left(\coprod_{g \geq 0} \mathcal{E}_{g, c}(M_{a_0, b} ; p_W^{-1}(a_0), \xi)\right),$$
where the colimit is formed by $- \cup \Sigma$, and $c$ denotes the number of components of $p_W^{-1}(a_0) \cup \xi$. Furthermore, the map $\vert \pi_\partial\vert$ is a homology fibration. 
\end{proposition}
\begin{proof}
Identifying the fibre is elementary. To show that $\vert \pi_\partial\vert$ is a homology fibration we wish to apply \cite[Proposition 4]{McDuff-Segal}. To do this, we observe that $(\pi_\partial)_p$ is a fibration, and that its fibre over $[a_0, a_1, \ldots, a_p, W]$ is $F(a_0, W)$. Thus face maps $d_i$ for $i > 0$ induce homeomorphisms on fibres, but the face map $d_0$ induces a map
$$\colim_{b\to\infty} \left(\coprod_{g \geq 0} \mathcal{E}_{g, c}(M_{a_0, b} ; p_W^{-1}(a_0), \xi)\right) \lra \colim_{b\to\infty} \left(\coprod_{g \geq 0} \mathcal{E}_{g, c'}(M_{a_1, b} ; p_W^{-1}(a_1), \xi)\right)$$
on fibres, given by union with the cobordism $p_W^{-1}([a_0, a_1])$. As this cobordism is connected relative to $p_W^{-1}(a_0)$, union with it may be expressed as a composition of maps of type $\alpha$, $\beta$ and $\gamma$, so by Theorem \ref{thm:stab} the induced map on homology is an isomorphism.
\end{proof}

In all, taking geometric realisation and the colimit of diagrams \eqref{eq:FundamentalDig} and \eqref{eq:FundamentalDig2} over stabilisation of the top row by $- \cup \Sigma$, we obtain a diagram where all horizontal maps are homotopy equivalences. A choice of point $(a_0, W) \in D(N_{(0,\infty)})_0$ such that $p_W^{-1}(a_0)=\emptyset$ gives a compatible collection of basepoints in all the spaces on the bottom row, and we obtain a zig-zag of weak homotopy equivalences between the homotopy fibres of all the vertical maps, taken at this compatible collection of basepoints. In particular, we obtain a zig-zag of homology equivalences between the actual fibres of of $\vert \pi_\partial\vert$ and $\Pi_\infty$,
\begin{equation}\label{eq:colim1}
\colim_{b\to\infty} \left(\coprod_{g \geq 0} \mathcal{E}_{g, c}(M_{a_0, b} ; \emptyset, \xi)\right)
\end{equation}
and
\begin{equation}\label{eq:colim2}
\colim_{b\to\infty} \left( \Gamma(\Psi(TM_{a_0, b}) \to M_{a_0, b}; s_a(\emptyset), s_a(\xi))\right).
\end{equation}

\begin{lemma}\label{lem:SecSpaceEq}
The stabilisation maps between the spaces of sections 
$$\Gamma(\Psi(TM_{a_0, b}) \to M_{a_0, b}; s_a(\emptyset), s_a(\xi))$$
are homotopy equivalences.
\end{lemma}
\begin{proof}
The stabilisation map is given by union with the section
$$s_a(\Sigma) \in \Gamma_c(\Psi(T(L \times [0,1])) \to L \times [0,1];s_a(\xi),s_a(\xi)) =: X$$
obtained by scanning the surface $\Sigma$. The space $X$ is a homotopy associative $H$-space, by concatenating intervals and reparametrising. As $L$ was chosen to be diffeomorphic to $\bR^{d-1}$, we may choose such a diffeomorphism; this identifies $X$ with
$$\mathrm{map}_c(\bR^{d-1} \times [0,1], \Psi(\bR^d);s_a(\xi),s_a(\xi)) \simeq \Omega_{s_a(\xi)} (\Omega^{d-1}\Psi(\bR^d))$$
as an $H$-group. In particular, $\pi_0(X)$ is a group. Thus there is a section $f$ such that $s_a(\Sigma) \cdot f$ is homotopic to the constant section $s_a(\xi) \times [0,1]$, but then union with the section $f$ gives a homotopy inverse to the stabilisation map.
\end{proof}

\begin{corollary}
There is a bijection
$$\pi_0(\Gamma(\Psi(TM_{a_0, b}) \to M_{a_0, b}; s_a(\emptyset), s_a(\xi))) \cong \bZ \times H_2(M;\bZ).$$
\end{corollary}
\begin{proof}
The set of path components of \eqref{eq:colim1} is isomorphic to $\bZ \times H_2(M;\bZ)$, by Lemma \ref{lemma:spaces-pi0}.
\end{proof}

In Section \ref{sec:PathComp} we give a concrete description of this bijection. Combining the homology equivalence between \eqref{eq:colim1} and \eqref{eq:colim2}, Lemma \ref{lem:SecSpaceEq}, and Theorem \ref{thm:stab}, we see that the scanning map
$$\mathcal{E}_{g, c}(M_{a_0, b} ; \emptyset, \xi) \lra \Gamma(\Psi(TM_{a_0, b}) \to M_{a_0, b}; s_a(\emptyset), s_a(\xi))$$
is a homology isomorphism in degrees $\leq \tfrac{2}{3}(g-1)$. (We have used the fact that $L$ is contractible, so when we write $\Sigma \subset L \times [0,3]$ as the composition of $\alpha$ and $\beta$ maps, the $\beta$ maps are always gluing on a pair of pants with nullhomotopic outgoing boundary, so Theorem \ref{thm:stab} \ref{thm:stab:it:2} gives a stability range $ \leq \tfrac{2}{3}g$ for gluing $\beta$ maps.) 

Extending surfaces and sections cylindrically from $M$ to $M_{a_0, b}$ gives a commutative square
\begin{equation*}
\xymatrix{
\mathcal{E}_{g, c}(M ; \xi) \ar[r]\ar[d]& \Gamma(\Psi(TM) \to M; s_a(\xi)) \ar[d]\\
\mathcal{E}_{g, c}(M_{a_0, b} ; \emptyset, \xi) \ar[r] & \Gamma(\Psi(TM_{a_0, b}) \to M_{a_0, b}; s_a(\emptyset), s_a(\xi))
}
\end{equation*}
where the vertical maps are clearly homotopy equivalences; this proves the first part of Theorem \ref{thm:MainH}. The second part of Theorem \ref{thm:MainH}, \emph{in the case where the manifold $M$ has non-empty boundary}, follows from the commutative square
\begin{equation*}
\xymatrix{
\mathcal{E}_{g, 1}(M ; \xi) \ar[r]\ar[d]^{\gamma_{g,1}}& \Gamma(\Psi(TM) \to M; s_a(\xi)) \ar[d]\\
\mathcal{E}_{g}(M_1) \ar[r] & \Gamma(\Psi(TM_1) \to M_1;s_a(\emptyset))
}
\end{equation*}
where $\xi \subset \partial M$ is a single nullhomotopic circle, ${\gamma_{g,1}}$ is the map that glues on a collar $[0,1] \times \partial M$ containing a disc, and the right-hand map is given by union with the section obtained by scanning the disc. The right-hand map is an equivalence by an argument analogous to that of Lemma \ref{lem:SecSpaceEq}, and the left-hand map is an isomorphism in homology in degrees $ \leq \tfrac{2}{3}g$ by Theorem \ref{thm:stab}. This finishes the proof of Theorem \ref{thm:MainH} in the case where the manifold $M$ has non-empty boundary. In Section \ref{section:closed} we show how to deduce Theorem \ref{thm:MainH} in the case where $M$ has empty boundary.
\section{Surgery}\label{section:surgery}
In this section we prove Proposition \ref{prop:Surgery}. We will prove in detail that the map
\begin{equation}\label{eq:SurgeryProp1}
\vert D_\partial(M_{\infty,\infty}; \xi)_\bullet\vert \lra \vert D(M_{\infty,\infty}; \xi)_\bullet\vert 
\end{equation}
is a weak homotopy equivalence, and then briefly explain the changes in the argument to show that
\begin{equation}\label{eq:SurgeryProp2}
\vert \widehat{D}_\partial(N_{(0,\infty)})_\bullet\vert \lra \vert \widehat{D}(N_{(0,\infty)})_\bullet\vert
\end{equation}
is a weak homotopy equivalence. 

The proof of these results uses a parametrised surgery move similar to that of \cite{GMTW} and \cite{GR-W2}. This will be used in two ways: to make properties (i) and (ii) of Definition \ref{defn:Dpartial} hold. Making property (ii) hold is analogous to the ``positive boundary subcategory'' theorem of \cite{GMTW}, but property (i) does not have an analogue.

At a technical level, we construct the ``do surgery'' maps differently to \cite{GR-W2}. There, these maps are constructed between geometric realisations of semi-simplicial spaces in terms of barycentric coordinates, and need to be glued together very carefully. Instead, ours will be semi-simplicial maps and easy to define. The idea is quite general, and can for example also be used with \cite{GR-W2}.

We first introduce two more auxiliary semi-simplicial spaces.

\begin{df} Define a semi-simplicial space $\Dpn{M_{\infty,b}}{\xi}$ whose space of $i$-simplices is the space of tuples $(W,a_0,\ldots,a_i)$ such that 
\begin{enumerate}
\item $0 < a_0 < a_1 < \cdots < a_i \in \bR$;
\item $W\in \Psi(M_{\infty, b};\xi)$;
\item each $a_j$ is either a regular value of $p_W \colon W \cap N_{[0,\infty)} \to [0,\infty)$, or $p_W^{-1}(a_j)$ contains only Morse critical points of index at least 1. We denote $\delta_j = p_{W}^{-1}(a_j)$;
\item for each $j$, the map $\pi_0(\delta_j)\rightarrow \pi_0(W\cap N_{[a_j,a_{j+1})})$ induced by the inclusion is a surjection;
\item $W \cap (M \cup N_{[0,a_0)} \cup L_{[0,b]})$ is path connected.
\end{enumerate}
Similarly, we let $\Dn{M_{\infty, b}}{\xi}$ have as $i$-simplices those tuples $(W,a_0,\ldots,a_i)$ which satisfy just the first three conditions above. In both cases, the simplices are topologised as a subspace of $(\bR^\delta)^{i+1} \times \Psi(M_{\infty, b};\xi)$ and the face maps are given by forgetting the $a_j$.
\end{df}

\begin{lemma}\label{lem:natural}
The inclusions
$$|\Dp{M_{\infty, b}}{\xi}|\longrightarrow |\Dpn{M_{\infty, b}}{\xi}| \quad\text{and}\quad |\D{M_{\infty, b}}{\xi}|\longrightarrow |\Dn{M_{\infty, b}}{\xi}|$$
are weak homotopy equivalances.
\end{lemma}
\begin{proof}
The argument is the same in both cases; to be specific, we treat the first case. Let ${\mathcal J}_{\bullet,\bullet}$ be the bi-semi-simplicial space whose $(i,j)$-simplices consist of the tuples $(W,a_0,\ldots,a_i,b_0,\ldots,b_j)$ such that $(W,a_0,\ldots,a_i)$ is an $i$-simplex in $\Dp{M_{\infty, b}}{\xi}$ and $(W,a_0,\ldots,a_i,b_0,\ldots,b_j)$ is an $(i+j+1)$-simplex in $\Dpn{M_{\infty, b}}{\xi}$.

The $(p,\bullet)$-face map forgets the value $a_p$ and the $(\bullet,q)$-face map forgets the value $b_q$. It has an augmentation $\epsilon_{-,\bullet}$ to $\Dpn{M_{\infty, b}}{\xi}$ given by forgetting all the values $a_0,\ldots,a_i$ and an augmentation $\epsilon_{\bullet,-}$ to $\Dp{M_{\infty, b}}{\xi}$ given by forgetting all the values $b_0,\ldots,b_i$. The triangle
\begin{equation*}
\xymatrix{
 & \vert{\mathcal J}_{\bullet,\bullet}\vert \ar[ld]_{\vert\epsilon_{\bullet,-}\vert}\ar[rd]^{\vert\epsilon_{-,\bullet}\vert}\\
\vert\Dp{M_{\infty, b}}{\xi}\vert \ar[rr]& & \vert\Dpn{M_{\infty, b}}{\xi}\vert
}
\end{equation*}
commutes up to homotopy, by construction. 

The augmentation maps have local sections. We try to define a section of $\epsilon_{i,-} \colon \mathcal{J}_{i,0} \to D_\partial(M_{\infty,b};\xi)_i$ through the point $(W,a_0,\ldots,a_i,b_0)$ on the open neighbourhood $U$ of $(W,a_0,\ldots,a_i)$ consisting of those $W'$ such that $a_0,\ldots, a_i$ are still regular values and $b_0$ contains only Morse critical points of index at least 1, by the formula $(W', a_0, \ldots, a_i) \mapsto (W', a_0, \ldots, a_i, b_0)$. To see that this defines a section, we must check that $p_{W'}^{-1}([a_j, a_{j+1}))$ and $p_{W'}^{-1}([a_i, b_0))$ all satisfy the connectivity requirement (iv). The first case is immediate: as the $a_j$ remain regular values, $p_{W'}^{-1}([a_j, a_{j+1})) \cong p_{W}^{-1}([a_j, a_{j+1}))$ and 
\[W \cap (M \cup N_{[0,a_0)} \cup L_{[0,b]}) \cong W' \cap (M \cup N_{[0,a_0)} \cup L_{[0,b]}).\]
In the second case, $p_{W'}^{-1}([a_i, b_0))$ differs from $p_{W}^{-1}([a_i, b_0))$ by adding 1- or 2-handles, but this does not change the connectivity property with respect to the lower boundary. We show that the augmentation map $\epsilon_{-,j}$ has local sections in a similar (but easier) way.

The fibre $F_\bullet$ of $\epsilon_{\bullet,-}$ over $(W,a_0,\ldots,a_i)$ has $p$-simplices those tuples of real numbers $(b_0, \ldots, b_p)$ such that $(W,a_0,\ldots,a_i,b_0, \ldots, b_p)$ is a simplex of $\Dpn{M_{\infty, b}}{\xi}$, i.e., $p_W^{-1}(b_j)$ contains only Morse critical points of index at least 1, and $p_W^{-1}([a_i, b_0))$ and each $p_W^{-1}([b_j, b_{j+1}))$ are connected relative to its lower boundary. These conditions only involve pairs of $b_j$'s, so this is a topological flag complex (whose topology is discrete). Given a finite collection $b_1,\ldots, b_n$ of elements of $F_0$, we may choose $a_i < c < \min(b_j)$ such that $[a_i, c]$ consists of regular values of $p_W$. Then $c$ is also in $F_0$, and $(c, b_j) \in F_1$ for each $b_j$. It follows from Criterion \ref{criterion2} (and Remark \ref{rem:criterion2}) that $\vert\epsilon_{\bullet,-}\vert$ is a weak homotopy equivalence.

The fibre $F'_\bullet$ of $\epsilon_{-,\bullet}$ over $(W,b_0,\ldots,b_j)$ has $p$-simplices those tuples of real numbers $(a_0, \ldots, a_p)$ which are regular values of $p_W$, such that 
\[(W,a_0,\ldots,a_p,b_0, \ldots, b_j)\]
is a simplex of $\Dpn{M_{\infty, b}}{\xi}$, which is again seen to be a topological flag complex (whose topology is discrete). For a finite collection $a_1, \ldots, a_n$ of elements of $F'_0$, choose $\max(a_j) < c < b_0$ such that $[c, b_0)$ consists of regular values of $p_W$. Then $c$ is also in $F'_0$, and we claim that each $(a_j, c)$ is a 1-simplex of $F'_\bullet$, i.e., that $p_W^{-1}([a_j, c))$ is path connected relative to $p_W^{-1}(a_j)$. To see this, first note that $p_W^{-1}([a_j, b_0))$ is path connected relative to $p_W^{-1}(a_j)$ by assumption, so there is a path from any point of $p_W^{-1}([a_j, c))$ to $p_W^{-1}(a_j)$ inside of $p_W^{-1}([a_j, b_0))$, but as $[c,b_0)$ consists of regular values $p_W^{-1}([c,b_0))$ is a cylinder, so this path may be homotoped into $p_W^{-1}([a_j, c))$ relative to its ends. It follows from Criterion \ref{criterion2} that $\vert\epsilon_{-,\bullet}\vert$ is a weak homotopy equivalence.
\end{proof}

\subsection{Local surgery move}
Let $w=(W,a_0,\ldots,a_i)$ be a simplex in $\Dn{M_{\infty, b}}{\xi}$. We first construct a path from this $i$-simplex to an $i$-simplex $w'=(W',a_0,\ldots,a_i)$ in $\Dpn{M_{\infty, b}}{\xi}$. In particular, this will prove that the inclusion $\Dpn{M_{\infty, b}}{\xi}\rightarrow \Dn{M_{\infty, b}}{\xi}$ is levelwise $0$-connected. In the last section we use this path to show that it is in fact a homotopy equivalence after geometric realisation.

Let $R(w)=\{W_1,\ldots,W_k\}$ be the set of connected components of $W\cap M_{a_0, b}$, and let $W_0$ be the connected component that contains $\xi$. Define
\[P_{a,b}(W) = \{\omega\in \pi_0(p_W^{-1}[a,b))\mid a\notin p_W(\omega)\},\qquad P(w)=\bigcup_{k=0}^{i-1} P_{a_k,a_{k+1}}(W).\]
Observe that \emph{$w$ is in $\Dpn{M}{\xi}$ if and only if $R(w)\cup P(w)=\emptyset$.} We define the following subsets of $\bR^3$:
\begin{align*}
T' &= (\{0\}\times [-3,3])\cup ((0,5]\times \{0\})\subset \bR^2\subset \bR^3 \\
T &= \{(x,y,z)\in \bR^3\mid d(T',(x,y,z))<1,\: |x|\leq 3,\: y< 5\}
\end{align*}
and let $x_1,x_2 \colon T \to \bR$ be the first and second coordinate functions.

\begin{df}\label{def:local-surgery} Let $w=(W,a_0,\ldots,a_i)$ be an $i$-simplex in $\Dn{M}{\xi}$. A \emph{local surgery datum for $w$} is a pair $Q = (\Lambda, e)$ where $\Lambda$ is a set and $e\colon \Lambda \times T\rightarrow M_{\infty,b}$ is a closed embedding, whose restriction $e_{|\{\lambda\}\times T}$ we denote by $e_\lambda$, such that:
\begin{enumerate}
\item\label{1}  $e^{-1}(W\cap M_{a_i, b}) = \Lambda \times (T\cap x_2^{-1}(\{-3,3\}));$
\item\label{2} $(\mathrm{Id}_\Lambda \times x_1)(e^{-1}(W\cap N_{[a_i,\infty)})) \subset \Lambda\times (4,5)$;
\item\label{3} for each $\lambda \in \Lambda$, $e_\lambda(x_2^{-1}(-3))\subset W_0\cap M_{a_0, b}$;
\item\label{4} for each $\omega \in P(w)\cup R(w)$, there is a $\lambda\in \Lambda$ such that $e_\lambda(x_2^{-1}(3))\subset \omega$;
\item\label{5} $\lim_{x\rightarrow 5} e_\lambda(x,y,z) = \infty$ for all $\lambda\in \Lambda$ and all $(y,z)$ such that $\sqrt{y^2+z^2}<1$;
\item\label{6} for each $\lambda\in \Lambda$ and for each $j= 0,\ldots,i$ there is an $\epsilon>0$ such that for all $a\in (a_j-\epsilon,a_j+\epsilon)$, either $x_2e^{-1}_\lambda(N_{a_j}) \in (-2,-1)$ or $x_1e^{-1}_\lambda(N_{a_j}) \in (2,3)$.
\end{enumerate}
\end{df}

\begin{proposition}\label{prop:local-surgery} A local surgery datum $Q$ for an $i$-simplex $w$ of $\Dn{M}{\xi}$ determines a path $\Phi_Q(t)$ that starts at $w$ and ends in an $i$-simplex of $\Dpn{M_{\infty,b}}{\xi}$. 
\end{proposition}
\begin{proof}
Consider the $1$-parameter family of diffeomorphisms of $$Y= T\cup \{(x,y,z)\in \bR^3\mid x\leq 5, ||(y,z)||<1\}$$ given by
\[h_t(x,y,z) =\left\{ 
\begin{array}{ll} 
(y,x+(x-3)te^{2-\frac{1}{1-||(y,z)||^2}}),z)& {\rm if}~||(y,z)||<1, x\leq 3, \\ 
(x,y,z)& {\rm otherwise}.
\end{array}\right.\]
The properties of this family which we will use are the following: 
\begin{enumerate}
\item\label{21} $h_0$ is the identity;
\item\label{22} if $x\in (4,5)$ and $||(y,z)||<1/\sqrt{2}$, then $x_1h_1^{-1}(x,y,z)\in (3,4)$;
\item\label{23} $h_t$ is the identity on $T\cap x_1^{-1}([-\infty,3))$;
\item\label{24} $h_t$ extends to $\bR^3$ with the identity outside $T$.
\end{enumerate}
The family $h_t$ induces a $1$-parameter family of maps
\[H_t\colon \Psi(T)\longrightarrow \Psi(T)\]
given by sending a submanifold $W$ to $h_t(W)\cap T$. From the first property of $h_t$ it follows that $H_0$ is the identity. In Figures \ref{s03} and \ref{s04} we give a picture of the action of $H_t$ on the dark disc at the bottom of Figure \ref{s02}.

\begin{figure}[h]
\centering
\subfloat[]{\label{s02}\includegraphics[width=3.9cm,angle=90]{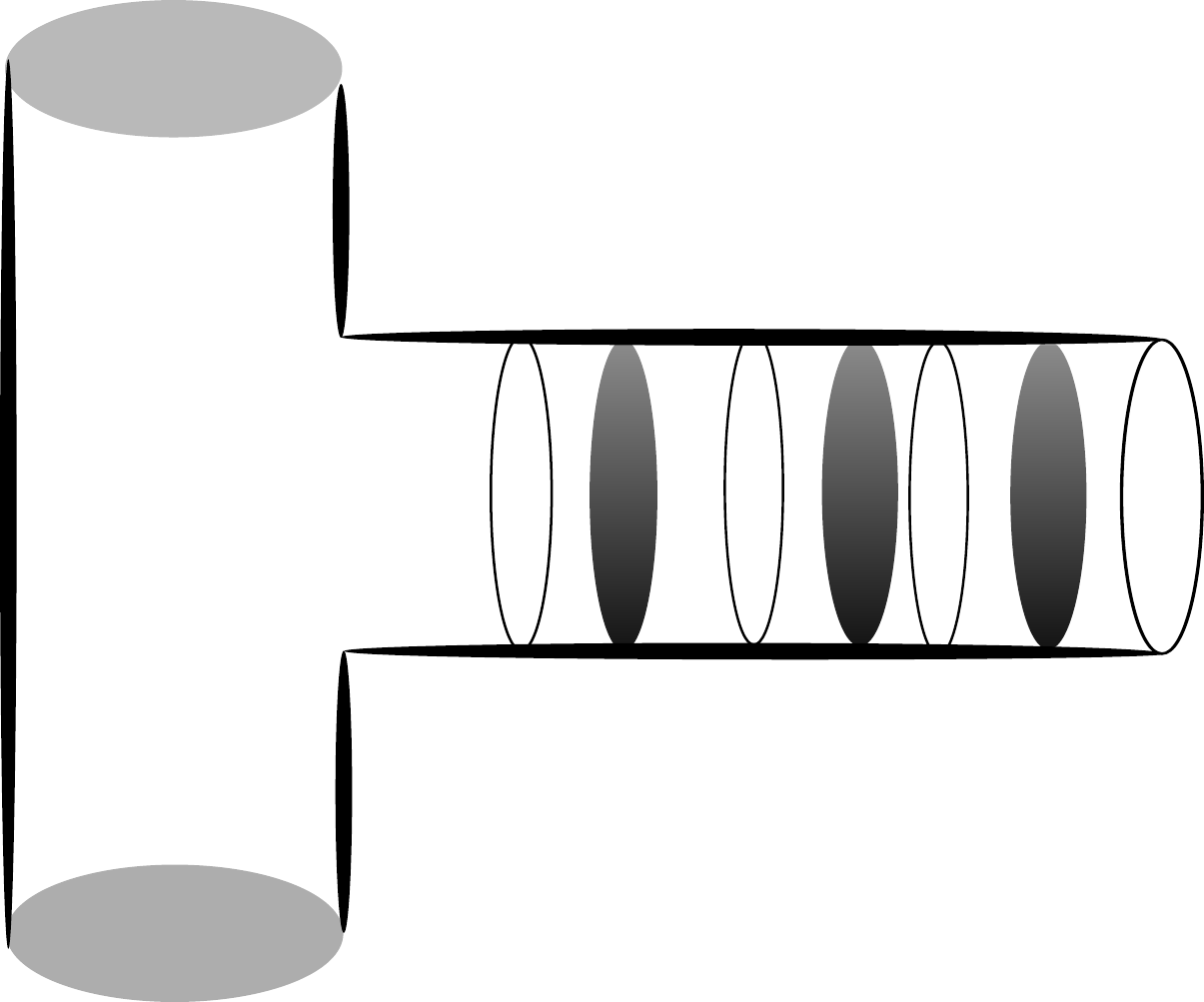}}\hspace{1cm}
\subfloat[]{\label{s03}\includegraphics[width=3.9cm,angle=90]{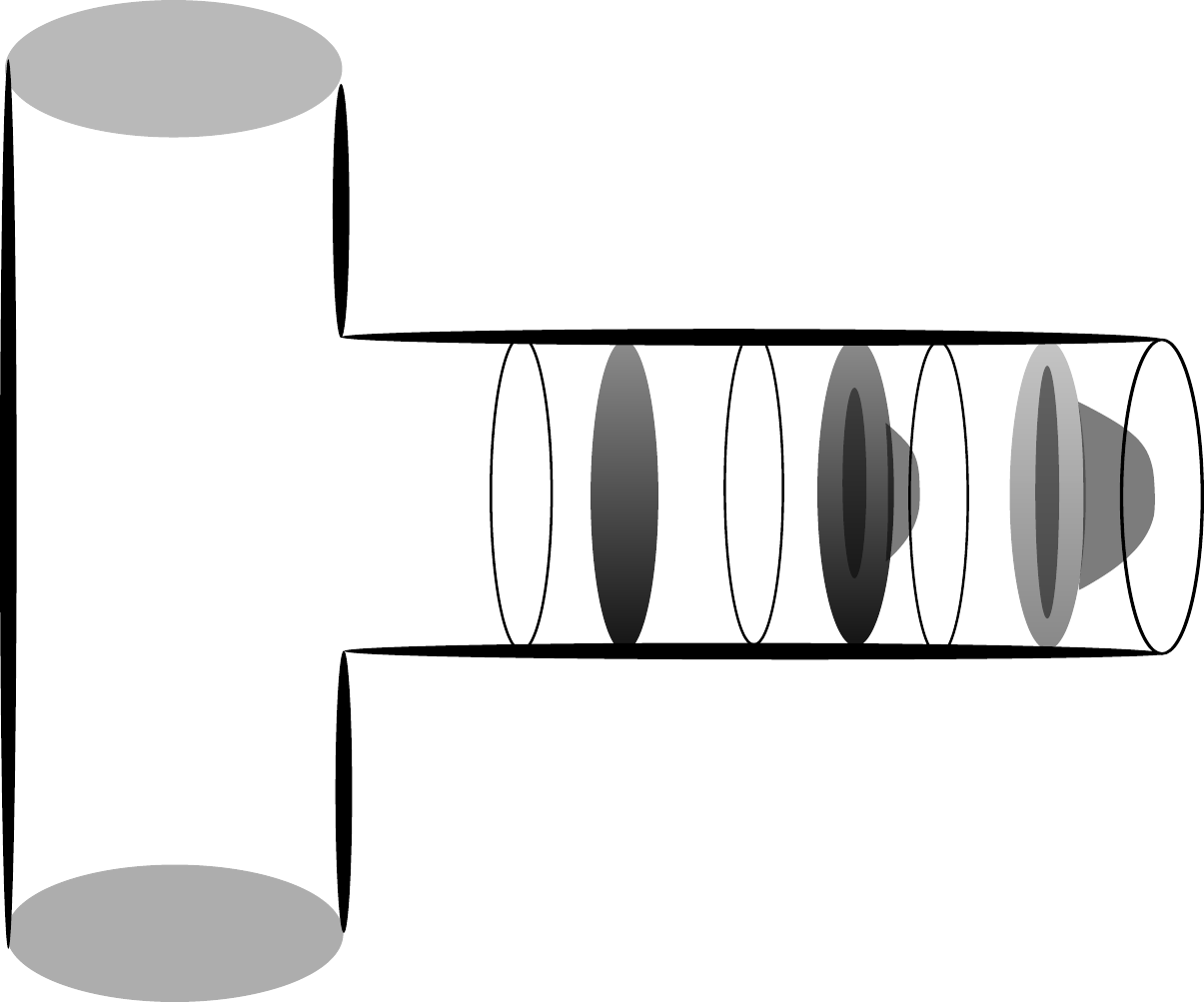}}\\\vspace{1cm}
\subfloat[]{\label{s04}\includegraphics[width=3.9cm,angle=90]{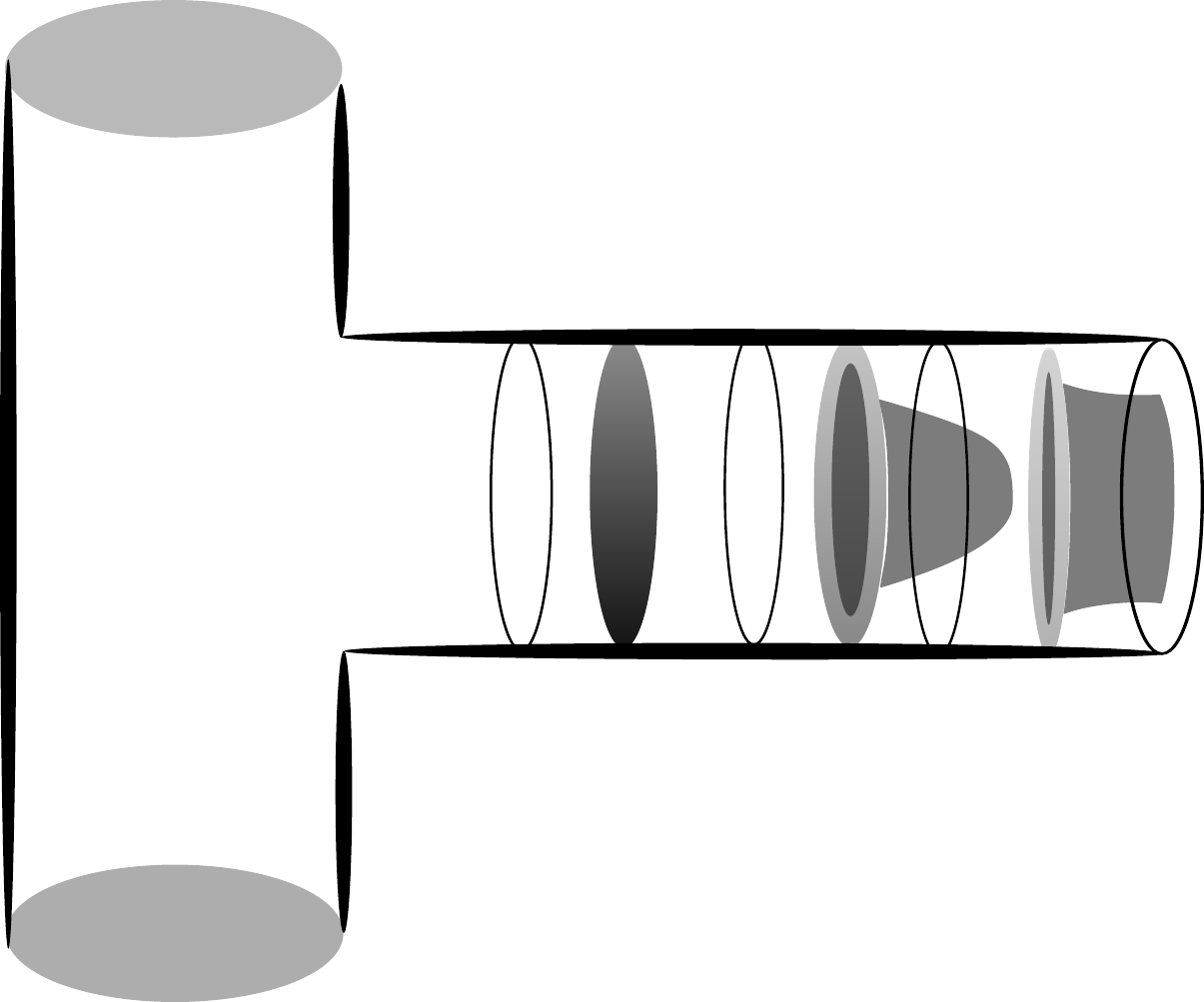}}
\caption{The effect of the family $H_t$ in the surgery movement on discs in $x_1^{-1}((2,3))$, $x_1^{-1}((3,4))$ and $x_1^{-1}((4,5))$.}
\label{figure:surgerynew0}
\end{figure}

Consider now the path $\eta$ in $\Psi(T)$ given in Figure \ref{figure:surgerynew} that starts with the surface $x_2^{-1}(\{-3,3\})$, which is the disjoint union of two open balls in $T$. It pushes both balls to infinity, joins the balls there and then pulls them backwards. In Figure \ref{s01}, a picture at time $0$ is given. The three vertical circles represent the balls $x_1^{-1}(2)$, $x_1^{-1}(3)$ and $x_1^{-1}(4)$ and the horizontal circle represents the ball $x_2^{-1}(-1)$. The planes in the figure will be given an interpretation later. In Figures \ref{s05}, \ref{s06} and \ref{s07}, the ball is pushed to infinity, and in Figures \ref{s08} and \ref{s09} the surface returns in the shape of a (non-compact) pair of pants. The main properties of this movement are the following:
\begin{enumerate}
	\item\label{10} $\eta(0) = x_2^{-1}(\{-3,3\})$;
	\item\label{11} all the values in $(1,2)$ are regular values or Morse critical values of index $2$ for the restriction of $x_2$ to $\eta(t)$;
	\item\label{12} all the values in $(2,3)$ are regular values for the restriction of $x_1$ to $\eta(t)$ or Morse critical values of Morse index $1$ or $2$ (the former possibility happens only in the step from \ref{s08} to \ref{s09});
	\item\label{13} $\eta(t)\cap x_1^{-1}((4,5)) \subset \{(x,y,z)\in T\mid ||(y,z)||<1/\sqrt{2},\, y\in (4,5)\}$;
	\item\label{14} in the surface $\eta(1)$, the circles $x_2^{-1}(\{-3,3\})\cap \eta(1)$ are in the same connected component. 
\end{enumerate}

\begin{figure}
\centering
\subfloat[]{\label{s01}\includegraphics[width=3.9cm,angle=90]{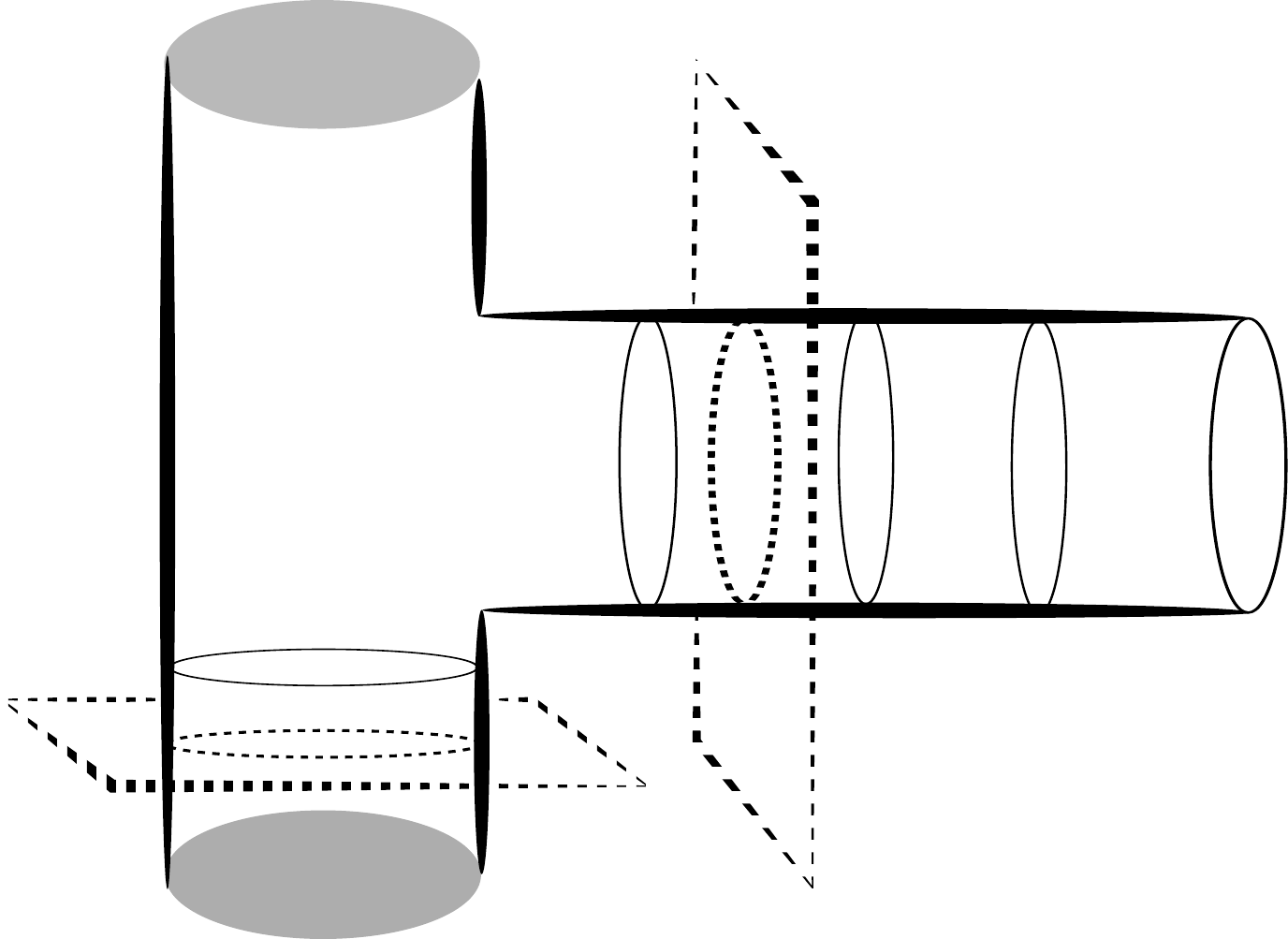}}
\subfloat[]{\label{s05}\includegraphics[width=3.9cm,angle=90]{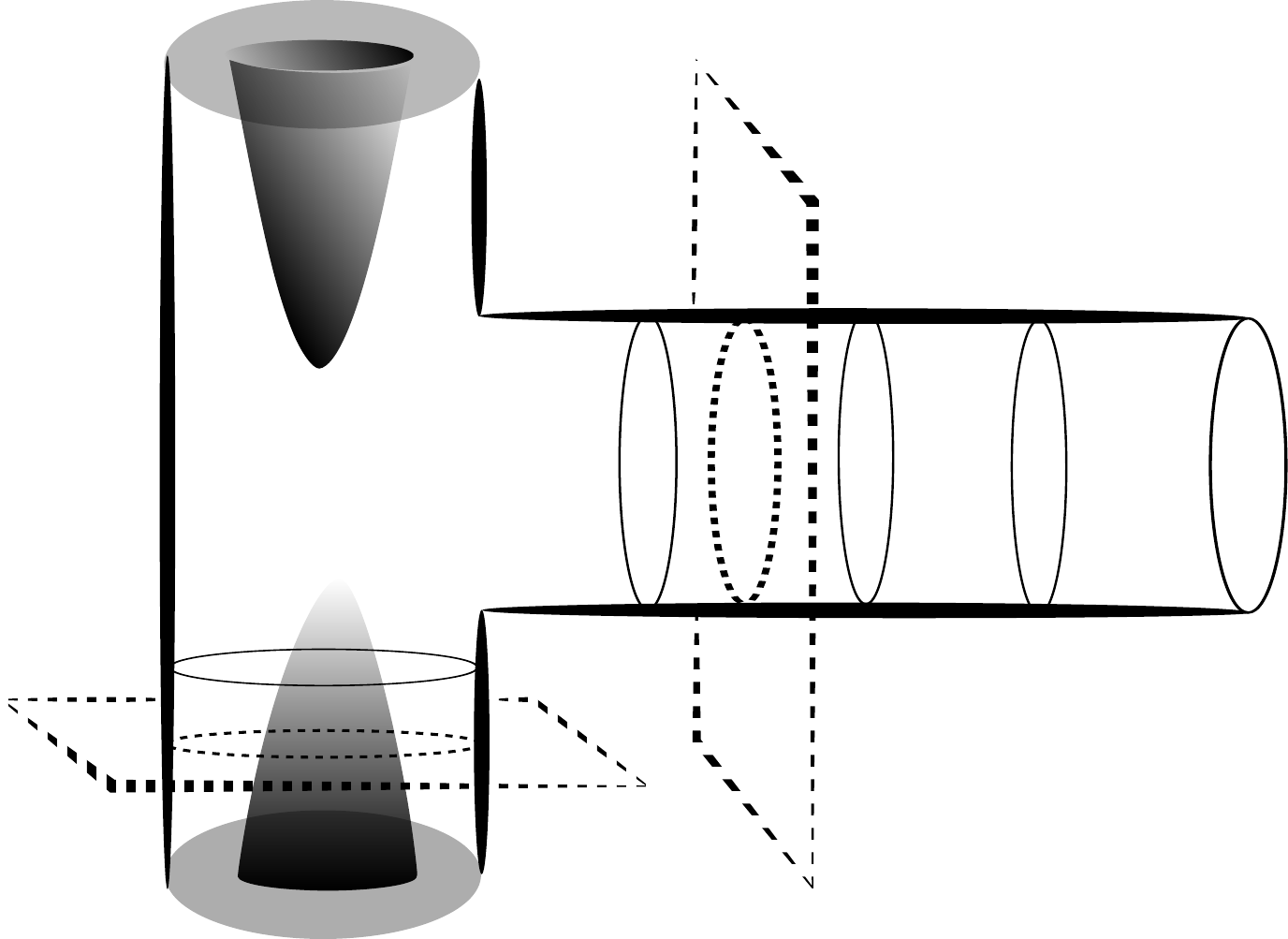}}\\
\subfloat[]{\label{s06}\includegraphics[width=3.9cm,angle=90]{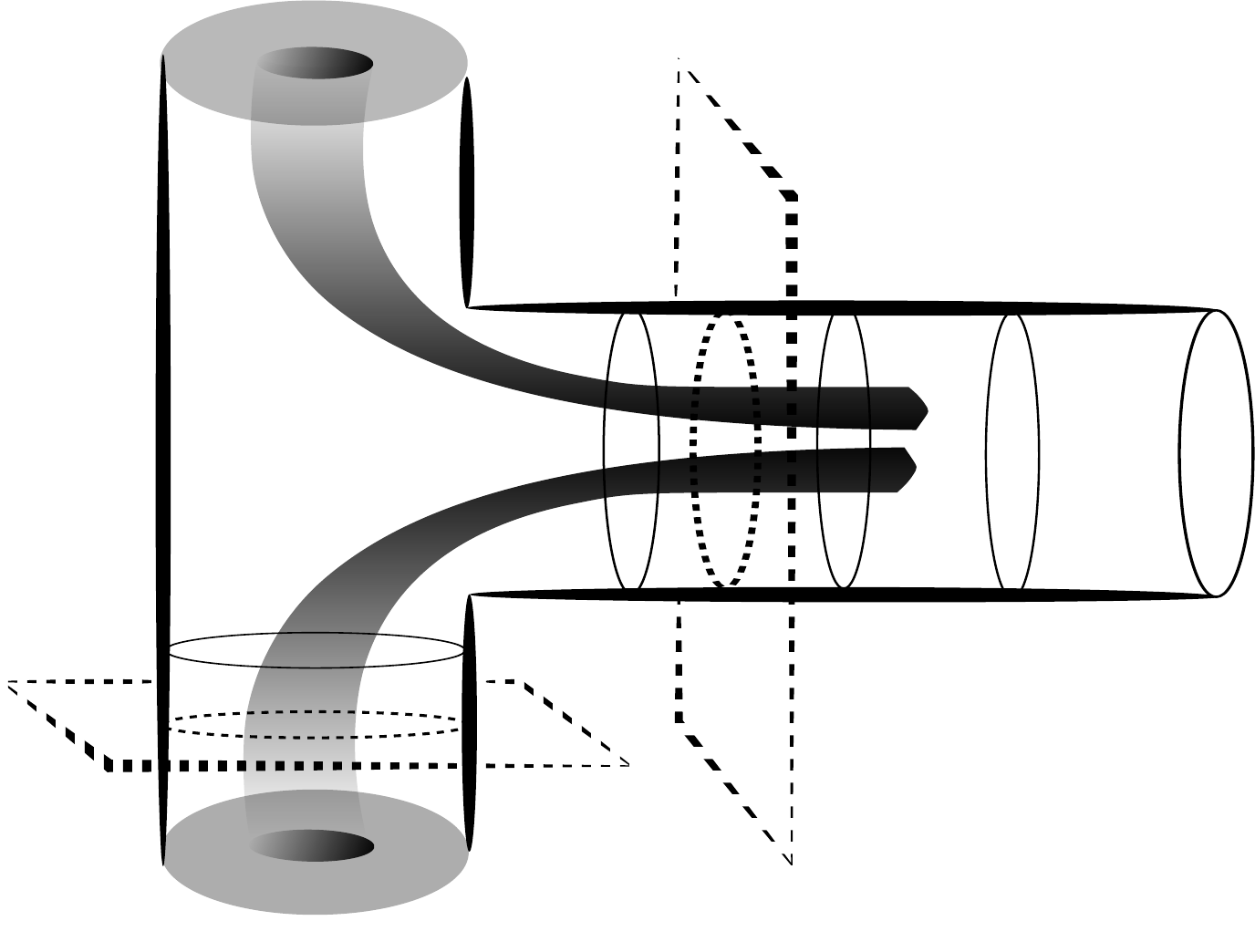}}
\subfloat[]{\label{s07}\includegraphics[width=3.9cm,angle=90]{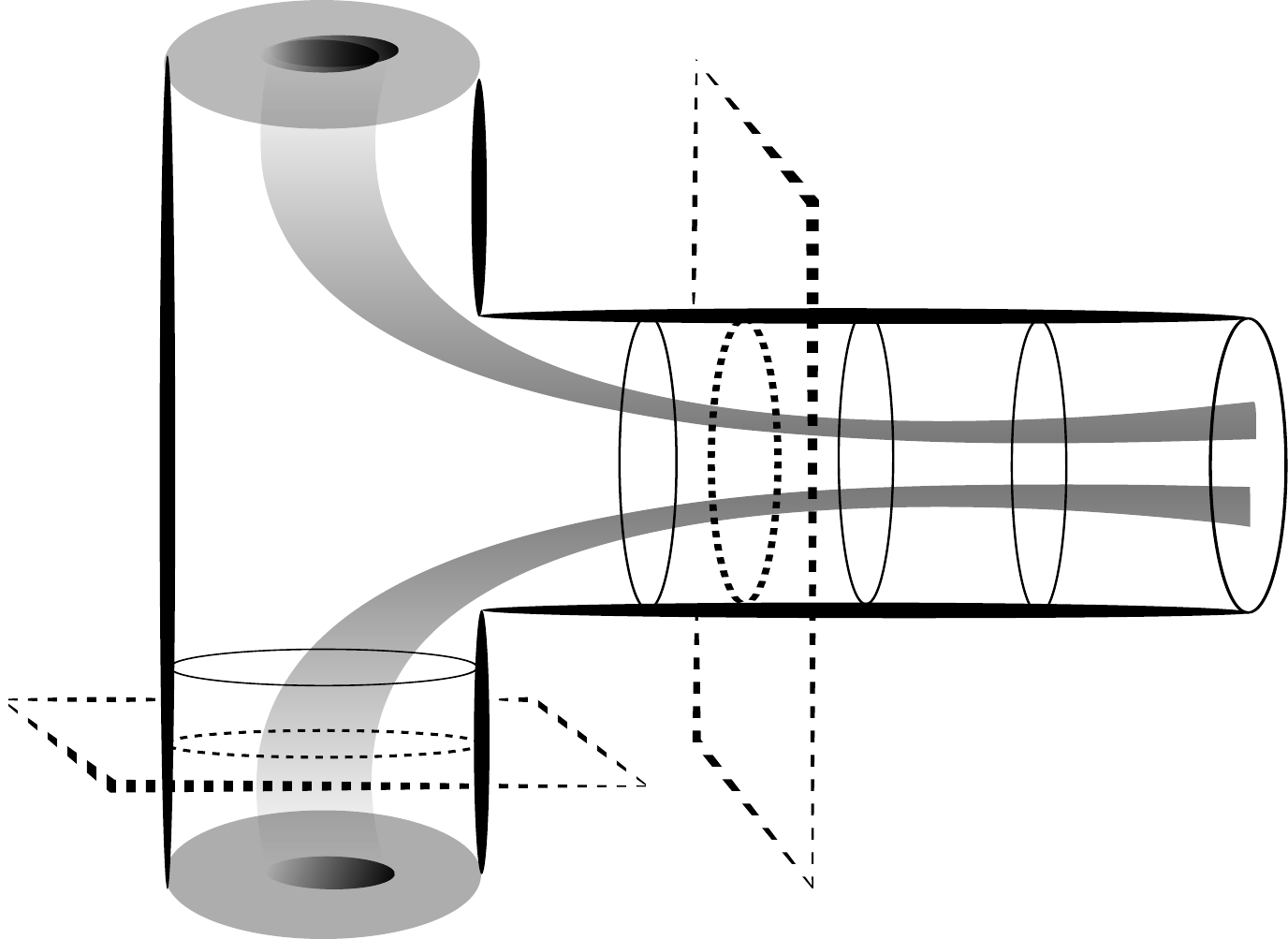}}\\
\subfloat[]{\label{s08}\includegraphics[width=3.9cm,angle=90]{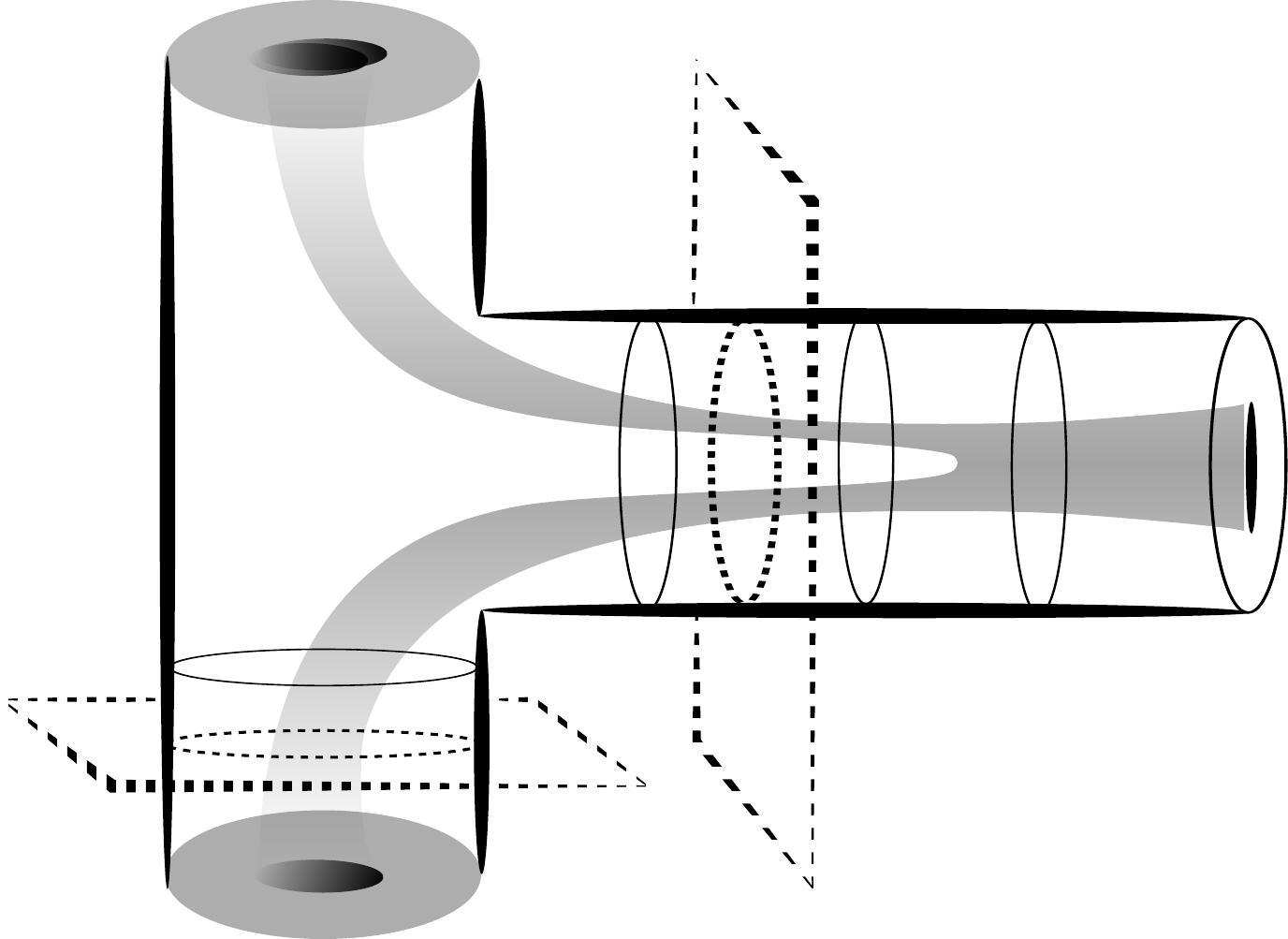}}
\subfloat[]{\label{s09}\includegraphics[width=3.9cm,angle=90]{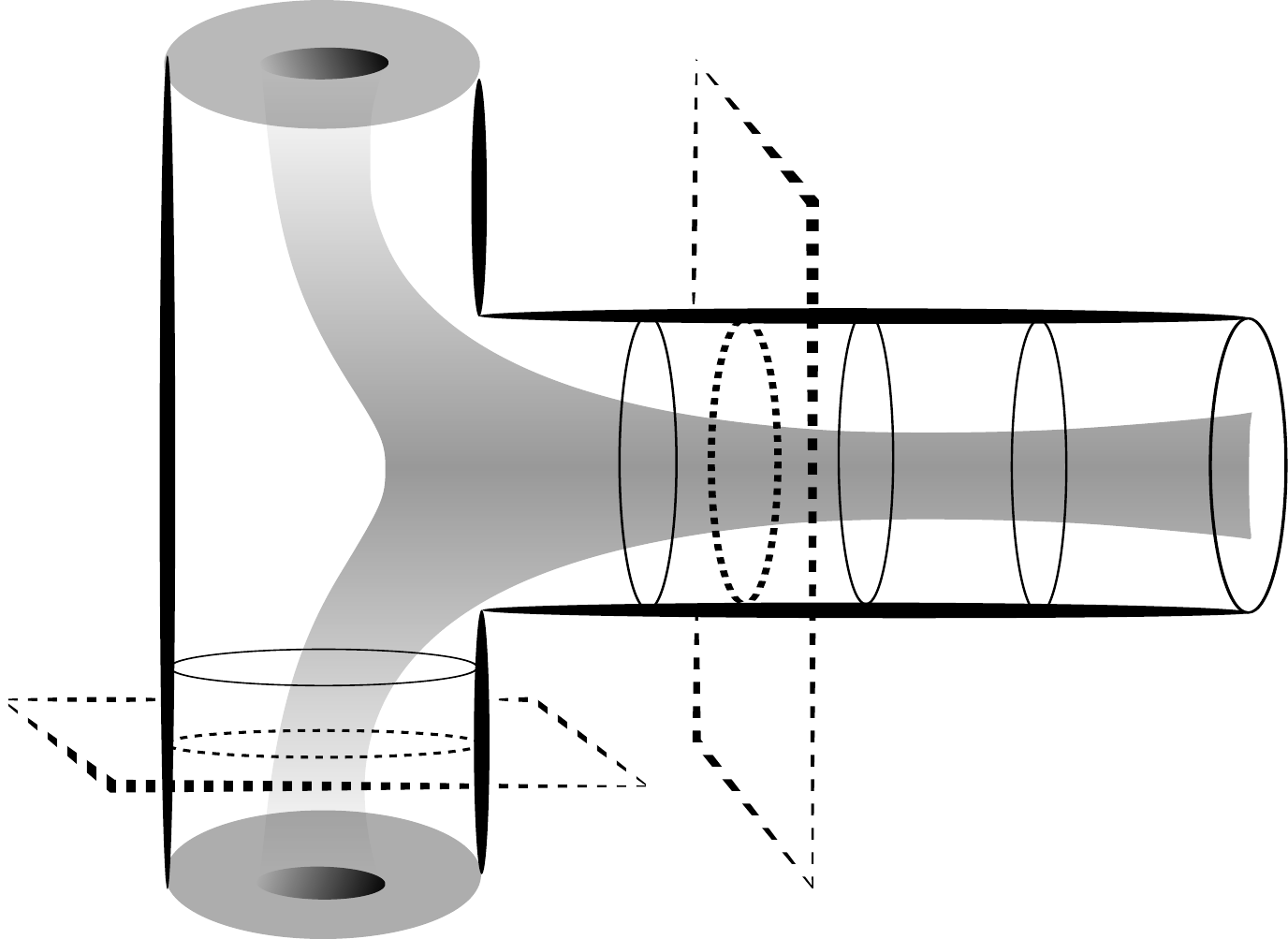}}
\caption{The path $\eta$ in the surgery movement. The shadowed surface at the top of \eqref{s01} is $\gamma(t)$, starting with $\gamma(0) = x_2^{-1}(\{-3,3\})$. The dotted planes are $e^{-1}(N_{a_j})$, and are still planes because of condition \ref{6} of the local surgery data.}
\label{figure:surgerynew}
\end{figure}

Now, let $V\in \Psi(T)$ be the union of the balls $V_0 = x_2^{-1}(\{-3,3\})$ and some surface $V_1\subset \{(x,y,z)\in T\mid x\in (4,5)\}$. We define a path $\phi_V\colon I\rightarrow \Psi(T)$ as 
\[\phi_V(t) = \begin{cases} H_{2t}(V) & \text{if } t\in [0,1/2],\\ H_1(V_1)\cup \eta(2t-1)& \text{if } t\in [1/2,1].\end{cases}\]
property \ref{23} of $h_t$ assures that both paths glue well: $H_1(V) = H_1(V_1)\cup V_0 = H_1(V_1)\cup \eta(0)$. Property \ref{22} for $h_t$ and property \ref{13} for $\eta$ assure that the union $H_1(V_1)\cup \eta(2t-1)$ is a union of disjoint surfaces, hence a surface. Hence the path is well-defined. We will use the following properties of this path:
\begin{enumerate}
	\item\label{30} $\phi_{V}(0) = V$ by property \ref{21} of $h_t$;
	\item\label{31} $\phi_{V}(1)\cap x_2^{-1}(\{-3,3\})$ is connected, by property \ref{14} of $\eta$.
\end{enumerate}

If we are given a set $V_\Lambda$ of surfaces $V_\lambda\subset \Psi(\{\lambda\}\times T)$ indexed by $\lambda$, we denote by $\phi_{V_\Lambda}(t)$ the result of performing $\phi_{V_\lambda}(t)$ in each $\lambda\times T$.

Now suppose we are given a surgery datum $Q$  for $w$, and let us define a path $\Phi_Q$ in $\Dn{M_{\infty,b}}{\xi}$ starting at $w= (W,a_0,\ldots,a_i)$ as 
\[\Phi_Q(t) = (W_Q(t),a_0,\ldots,a_i),\]
 where
\[W_Q(t)\cap e = e\phi_{e^{-1}(W)}(t), \qquad W_Q(t)\setminus e = W\setminus e.\]
There are five things to check for each $\lambda\in \Lambda$ in order to verify that this path is well-defined. First, that $e^{-1}_\lambda(W)$ is the union of $V_0$ and some surface $V_1$ as above is granted by conditions \ref{1} and \ref{2} of the surgery data, hence $\phi_{e^{-1}_\lambda(W)}$ is well-defined. Second, that $\Phi_Q(0) = w$ follows from property \ref{30} of $\phi_V$. Third, that the union of the two pieces of $W_Q(t)$ is indeed a surface is guaranteed by property \ref{24} of $h_t$. Fourth: as described, the embedding $e_\lambda$ does not induce a map $\{\lambda\}\times \Psi(T)\rightarrow \Psi(M_{\infty,b})$. Condition \ref{5} of the surgery data and properties \ref{23} and \ref{24} of $h_t$ grant that the precomposition $I\rightarrow \Psi(T)\rightarrow \Psi(M_\infty,b)$ with $\phi_{e^{-1}_\lambda(W)}$ is continuous. In other words, they grant that the surface $W_Q(t)\subset M_{\infty,b}$ is closed in $M(\infty)$ and that $W\cap M_{a,b}$ is compact. Fifth, that $(a_0,\ldots,a_i)$ are regular values or Morse critical points of index $1$ or $2$ is a consequence of properties \ref{11} and \ref{12} of the path $\eta$, together with the following consequences of conditions \ref{5} and \ref{6} of the surgery data:
\begin{enumerate}
\item If $x_2e^{-1}_\lambda(a_j)\in (-2,-1)$, then $p_{W_Q(t)}^{-1}(a_j) = x_2^{-1}(b_j)$ for some $b_j\in (1,2)$.
\item If $x_1e^{-1}_\lambda(a_j)\in (2,3)$, then $p_{W_Q(t)}^{-1}(a_j) = x_1^{-1}(b_j)$ for some $b_j\in (2,3)$.
\item $\frac{\partial}{\partial x_1} p_We_\lambda(x,y,z) > 0$ if $y\in (-2,-1)$.
\item $\frac{\partial}{\partial x_2} p_We_\lambda(x,y,z) < 0$ if $x\in (2,3)$.
\end{enumerate}

Finally, from conditions \ref{2} and \ref{3} in the definition of surgery datum and property \ref{31} of $\phi_{V}$, it follows that $P(\Phi_Q(1))\cup R(\Phi_Q(1))$ is the empty set, hence $\Phi_Q(1)\in \Dn{M_{\infty,b}}{\xi}$. 
\qedhere
\end{proof}

\begin{remark} This move is a simplified version of the one used in \cite{GMTW}. The one used there is more powerful and extends to surfaces with any tangential structure. Sadly, that move needs to push parts of the surface to both $+\infty$ and $-\infty$, while here we are only allowed to push things to $+\infty$.
\end{remark}

\subsection{Global surgery move}

We will now construct a bi-semi-simplicial space $\H_{\bullet,\bullet}$ with an augmentation to $\Dn{M_{\infty, b}}{\xi}$ which, over each simplex of $\Dn{M_{\infty, b}}{\xi}$, consists of certain tuples of local surgery data. This will allow us to compare it to $\Dpn{M_{\infty, b}}{\xi}$ by ``doing surgery'' in an appropriate way.

\begin{df} Let $\H_{\bullet,\bullet}$ be the bi-semi-simplicial space whose space of $(i,j)$-simplices is the space of tuples $(w,Q_0,\ldots,Q_j,s_0,\ldots,s_j)$ where 
\begin{enumerate}
\item $w$ is an $i$-simplex in $\Dn{M_{\infty, b}}{\xi}$;
\item each $Q_q$ is a local surgery datum for $w$;
\item\label{gs:3} the embeddings in $Q_0,\ldots, Q_j$ are pairwise disjoint;
\item $(s_0,\ldots,s_j)\in [0,1]^{j+1}$.
\end{enumerate}
The $(p,\bullet)$th face map forgets the regular value $a_p\in w$ and the $(\bullet,q)$th face map is
$$\partial_{\bullet,q}(w,Q_0,\ldots, Q_i,s_0,\ldots,s_i) = (\Phi_{Q_q}(s_q), Q_0,\ldots,\hat{Q}_q,\ldots,Q_i,s_0,\ldots,\hat{s}_q,\ldots,s_i).$$
\end{df}
There is an augmentation map $\epsilon_{\bullet,\bullet}$ to $\Dn{M_{\infty, b}}{\xi}$ given by performing the surgery $Q_q$ on $w$ up to time $s_q$ for all $q$ and forgetting all the surgery data. Let $\H_{\bullet,\bullet}^1$ be the bi-semi-simplicial subspace of those simplices such that $s_0=\ldots=s_j=1$. Note that by Proposition \ref{prop:local-surgery} the restriction $\epsilon_{\bullet,\bullet}^1$ of $\epsilon_{\bullet,\bullet}$ to this subspace gives an augmentation onto $\Dpn{M_{\infty, b}}{\xi}$ and the following diagram commutes:
\begin{equation}\label{eq:225}
\begin{gathered}
\xymatrix{
\H^1_{\bullet, \bullet}\ar[r]\ar[d]^{\epsilon_{\bullet,\bullet}^1} &\H_{\bullet, \bullet}\ar[d]^{\epsilon_{\bullet,\bullet}} \\
\Dpn{M_{\infty, b}}{\xi} \ar[r] & \Dn{M_{\infty, b}}{\xi}.
}
\end{gathered}
\end{equation}

\begin{proposition}\label{prop:surgery} If $M$ has dimension at least $4$, the inclusion of $\H_{\bullet,\bullet}^1$ into $\H_{\bullet,\bullet}$ and the augmentation maps are weak homotopy equivalences after geometric realisation.
\end{proposition}

The first part of Proposition \ref{prop:Surgery} now follows from the commutative diagram
\begin{equation*}
\begin{gathered}
\xymatrix{
\vert \Dp{M_{\infty, b}}{\xi}\vert  \ar[d]\ar[r] & \vert \D{M_{\infty, b}}{\xi}\vert  \ar[d]\\
\vert \Dpn{M_{\infty, b}}{\xi}\vert  \ar[r] & \vert \Dn{M_{\infty, b}}{\xi}\vert ,
}
\end{gathered}
\end{equation*}
after taking the limit when $b\to\infty$, since the vertical maps are equivalences by Lemma \ref{lem:natural} and the lower map is an equivalence by \eqref{eq:225} and Proposition \ref{prop:surgery}. As we remarked earlier, the second part of Proposition \ref{prop:Surgery} is proved similarly.

\begin{proof}[Proof of Proposition \emph{\ref{prop:surgery}}] It is clear that the inclusion $\H^1_{\bullet,\bullet}\rightarrow \H_{\bullet,\bullet}$ is a levelwise equivalence. To see that the augmentation map $\epsilon_{\bullet,\bullet}^1$ is a homotopy equivalence after geometric realisation, we notice that the augmented semi-simplicial space $\epsilon_{i,\bullet}^1: \H^1_{i,\bullet} \to D_\partial^\natural(M_{\infty,b};\xi)_i$ has a simplicial contraction, by adding the empty surgery data.

For the map $\epsilon_{\bullet,\bullet}$, let $\H^0_{\bullet,\bullet}$ be the semi-simplicial subspace of $\H_{\bullet,\bullet}$ where the simplices are required to have all $s_i$ equal to 0, and let $\H'_{\bullet,\bullet}$ be the semi-simplicial space defined as $\H_{\bullet,\bullet}^0$, but replacing condition \ref{gs:3} in the definition of $\H_{\bullet,\bullet}$ by 
\begin{itemize}
	\item[(iii$^\prime$)] The restrictions of the embeddings $e_q$ in each $Q_q = (\Lambda_q,e_q)$ to the subspace $\Lambda\times T'\subset \Lambda\times T$ are pairwise disjoint.
\end{itemize} 
Notice that $\H^0_{\bullet,\bullet}\subset \H'_{\bullet,\bullet}$ and the following diagram is commutative:
\[\xymatrix{
\H'_{\bullet,\bullet} \ar[dr]_{\epsilon_{\bullet,\bullet}'} & \H^0_{\bullet,\bullet}\ar[d]^{\epsilon_{\bullet,\bullet}^0}\ar[l]\ar[r] & \H_{\bullet,\bullet}\ar[dl]^{\epsilon_{\bullet,\bullet}} \\
&\Dn{M_{\infty, b}}{\xi}.&
}\]
We next prove that the following statements are true, concluding that the augmentation map $\epsilon_{\bullet,\bullet}$ for $\H_{\bullet,\bullet}$ is a homotopy equivalence after geometric realisation, hence finishing the proof of this proposition.
\begin{enumerate}[(i)]
	\item The inclusion of $\H^0_{\bullet,\bullet}$ into $\H_{\bullet,\bullet}$ is a levelwise homotopy equivalence.
	\item  The inclusion of $\H^0_{\bullet,\bullet}$ into $\H'_{\bullet,\bullet}$ is a levelwise homotopy equivalence.
	\item The augmentation map $\epsilon_{\bullet,\bullet}'$ is a homotopy equivalence.

\end{enumerate}

Statement (i) is clear. For statement (ii), we will prove that the inclusion $\H^0_{\bullet,\bullet}\rightarrow \H'_{\bullet,\bullet}$ is a levelwise weak homotopy equivalence. Consider the deformation $h\colon \H'_{i,j}\times (0,1]\rightarrow \H'_{i,j}$ that sends a tuple $(w,Q_0,\ldots,Q_i)$ to the tuple $(w,h_t(Q_0),\ldots,h_t(Q_i))$, where $h_t(Q_q) = (\Lambda_q,h_t(e_q))$ and $h_t(e_q)(x,y,z) = h_t(tx,y,tz)$. Under this deformation, any point eventually ends up, and stays, in the subspace $\H^0_{i,j}$. If $f \colon (D^n, S^{n-1}) \to (\H'_{i,j}, \H^0_{i,j})$ represents a relative homotopy class, then because $D^n$ is compact the map $h(-, t)\circ f$ has image in $\H^0_{i,j}$ for some $t$, so the homotopy class of $f$ is trivial.

For statement (iii), we notice that $\H'_{i,\bullet}\rightarrow \DDn(M_{\infty, b};\xi)_i$ is an augmented topological flag complex, so we may apply Criterion \ref{criterion2} to show that it is a weak homotopy equivalence. Then $\epsilon'_{\bullet,\bullet}$ will be a levelwise equivalence in the $i$-direction, hence a weak homotopy equivalence after realisation.

We will prove in Lemma \ref{lemma:local-sections} that the augmentation map is surjective and has local sections. Moreover, given $w \in D_\partial^\natural(M_{\infty,b};\xi)_i$ and a non-empty finite collection $(w,Q_0),\ldots,(w,Q_j)$ of $(i,0)$-simplices over $w$, as the dimension of $M$ is $>2$, we can perturb the restriction $e_{0|\Lambda\times T'}$ of $e_0\in Q_0$ to be disjoint from $Q_0,\ldots,Q_j$, and any extension $e_{j+1}$ to $\Lambda\times T$ of this perturbation will define a $0$-simplex orthogonal to the given ones. \qedhere
\end{proof}

\begin{lemma}\label{lemma:local-sections} The augmentation map $\epsilon_{i,0}' \colon \H'_{i,0}\rightarrow \DDn(M_{\infty, b};\xi)_i$ is surjective and has local sections.
\end{lemma}
\begin{proof}
First we show that $\epsilon_{i,0}'$ is surjective: if $w\in D_\partial^\natural(M_{\infty, b};\xi)_i$, let $\Lambda = P(w)\cup R(w)$. As $M$ is connected, it is clear that we may take a smooth map $e'\colon \Lambda\times T'\rightarrow M_{\infty,b}$ satisfying the restriction of conditions \ref{1}, \ref{2}, \ref{3}, \ref{4} and \ref{5} of the local surgery data to $T'$, except that of being an embedding and that of being disjoint from $W$ outside $e(x_1^{-1}((4,5)))$, $e(0,-3,0)$ and $e(0,3,0)$. As the dimension of $M$ is at least $4$, a small perturbation makes it satisfy the latter properties. Again, as the dimension is greater than $3$, we may thicken the embedding $e'$ to an embedding $e\colon \Lambda\times T\rightarrow M_{\infty,b}$ that satisfies all conditions except \ref{6}, and we may deform $e$ to satisfy this last condition.

Next, we show that $\epsilon_{i,0}'$ has local sections. Let $(w,(\Lambda, e)) \in \H'_{i,0}$. We need to find a neighbourhood $U$ of $w$ in $\DDn(M_{\infty, b};\xi)_i$ and a section $s\colon U\rightarrow \H'_{i,0}$ so that $s(w) = (w,(\Lambda, e))$. Write $w=(W, b_0, \ldots, b_i)$ and choose a regular value $a > b_i$ of $p_W$. Let $U$ be an open neighbourhood of $w$ in $\DDn(M_{\infty, b};\xi)_i$ for which $a$ remains regular. The space
$$E := \{((W', b_0, \ldots, b_i), x \in W' \cap M_{a,b}) \in U \times M_{a,b}\}$$
over $U$ is a fibre bundle, and so it is locally trivial. Choosing a trivialisation on a smaller neighbourhood $U'$ of $w$, we obtain a map
$$\psi\colon U' \lra \Emb(W\cap M_{a,b}, M_{\infty,b}),$$
and using the $\Diff_c(M_{\infty,b})$-locally retractile property of $\Emb(W\cap M_{a,b}, M_{\infty,b})$ we obtain an even smaller neighbourhood $U''$ and a map
$$\phi \colon U'' \lra \Diff_c(M_{\infty,b})$$
such that $\phi(W', b_0, \ldots, b_i)(W \cap M_{a,b}) = W' \cap M_{a,b}$ for $(W', b_0, \ldots, b_i) \in U''$.

We now attempt to define a section $s \colon U'' \to \H'_{i,0}$ by
$$s(W', b_0, \ldots, b_i) = ((W', b_0, \ldots, b_i), (\Lambda, \phi(W', b_0, \ldots, b_i)\circ e)).$$
To check that this is indeed a section, we must verify the six properties of Definition \ref{def:local-surgery} for these data. Properties \ref{1} and \ref{3} are immediate from the fact that inside $M_{a,b}$ the data $(W', (\Lambda, \phi(W', b_0, \ldots, b_i)\circ e))$ agree with the data $(W, Q)$ modified by a diffeomorphism of $M_{\infty,b}$. Property \ref{5} is automatic, and property \ref{2} holds at the point $w$ and is an open condition, so it also holds on some neighbourhood $w \in U''' \subset U''$.  Property \ref{6} holds after perhaps shrinking $U''$, as then the diffeomorphisms $\phi(U'')$ may be assumed to be supported away from $e \cap p^{-1}(\{b_0, \ldots, b_i\})$.

This leaves property \ref{4}, which follows from the important observation that if $w'$ is sufficiently close to $w$, then $P(w')$ and $R(w')$ can only be \emph{smaller} than $P(w)$ and $R(w)$; i.e., the amount of surgery we must do to obtain suitably connected surfaces is upper semi-continuous. More precisely, if $w'$ is sufficiently close to $w$ then $W' \cap (M \cup N_{[0,b_0)} \cup L_{[0,b]})$ is obtained from $W \cap (M \cup N_{[0,b_0)} \cup L_{[0,b]})$ by attaching 1- and 2-handles at $b_0$, and $p_{W'}^{-1}([b_i, b_{i+1}))$ is obtained from $p_{W'}^{-1}([b_i, b_{i+1}))$ by attaching 1- and 2-handles at $b_{i+1}$, or subtracting 1- and 2-handles at $b_i$, neither of which change the required connectivity properties.\qedhere
\end{proof}

\section{Stable homology of the space of surfaces in a closed manifold}\label{section:closed}
In this section we prove Theorem \ref{thm:MainH} for manifolds $M$ with empty boundary. Before doing so, we briefly study the set of path components of the space of sections $\Gamma_c(\mathcal{S}(TM) \to M)$ for such manifolds. We fix a complete Riemannian metric $\gg$ on $M$.

\subsection{Path components of \boldmath $\Gamma_c(\Ss(TM) \to M)$}\label{sec:PathComp}

The space $\Ss(TM)$ is a bundle of Thom spaces over $M$, with fibre over $p\in M$ given by $\mathrm{Th}(\gamma_2^\perp \to \Gr_2^+(T_pM))$, the Thom space of the orthogonal complement to the tautological bundle over the Grassmannian of oriented 2-planes in $T_pM$. Similarly, we can form the bundle of Grassmannians $q\colon \Gr_2^+(TM) \to M$, which comes equipped with a bundle injection $\gamma_2 \hookrightarrow q^*TM$ from the tautological bundle to the pullback of the tangent bundle of $M$. We let $\gamma_2^\perp \to \Gr_2^+(TM)$ denote the orthogonal complement to $\gamma_2$ in $q^*TM$.

There is a map
$$c\colon \Ss(TM) \lra \mathrm{Th}(\gamma_2^\perp \to \Gr_2^+(TM))$$
given by identifying all the points at infinity. If we choose an orientation of $TM$ there is an induced orientation of $\gamma_2^\perp$, hence a Thom class
$$u \in H^{d-2}(\mathrm{Th}(\gamma_2^\perp \to \Gr_2^+(TM));\bZ).$$
There is also an Euler class $e=e(\gamma_2) \in H^2(\Gr_2^+(TM);\bZ)$, and so a class
$$u\cdot e \in H^{d}(\mathrm{Th}(\gamma_2^\perp \to \Gr_2^+(TM));\bZ).$$
By abuse of notation, we use the names $u$ and $u \cdot e$ for the cohomology classes on $\Ss(TM)$ given by $c^*(u)$ and $c^*(u \cdot e)$ respectively.

There are maps
\[\begin{array}{rcccl}
\pi\colon \Gamma_c(\Ss(TM) \to M) &\longrightarrow& H_c^{d-2}(M;\bZ) &\longrightarrow & H_2(M;\bZ) \\
s &\longmapsto& s^*(u) &\longmapsto& \pi(s)\\[0.2cm]
\chi\colon \Gamma_c(\Ss(TM) \to M) &\longrightarrow& H_c^d(M;\bZ) &\longrightarrow & H_0(M;\bZ) \\
s &\longmapsto& s^*(u\cdot e) &\longmapsto& \chi(s)
\end{array}\]
obtained by pulling back the classes $e$ or $u\cdot e$ along a section, and then applying Poincar{\'e} duality.

\begin{lemma}
If $M$ is connected, then under the scanning map 
\[\s \colon \mathcal{E}_g(M) \lra \Gamma_c(\Ss(TM) \to M)\]
 we have 
\begin{align*}
\pi(\s([f \colon \Sigma_g \hookrightarrow M])) &= f_*([\Sigma_g]) \in H_2(M;\bZ)\\
\chi(\s([f \colon \Sigma_g \hookrightarrow M])) &= 2-2g \in \bZ = H_0(M;\bZ).
\end{align*}
\end{lemma}
\begin{proof}
The cohomology class $u \in H^{d-2}(\Ss(TM);\bZ)$ is Poincar{\'e} dual to the class of the submanifold $\Gr_2^+(TM) \subset \Ss(TM)$, so if $s$ is a (suitably transverse) section, then $s^*(u)$ is Poincar{\'e} dual to the submanifold $s^{-1}(\Gr_2^+(TM))$. 

The cohomology class $u\cdot e \in H^{d}(\Ss(TM);\bZ)$ is Poincar{\'e} dual to the class of the submanifold $Z \subset \Gr_2^+(TM) \subset \Ss(TM)$, which is the zero set of a transverse section of $\gamma_2 \to \Gr_2^+(TM)$. Thus if $s$ is a (suitably transverse) section, then the class $s^*(u \cdot e)$ is Poincar{\'e} dual to the set of zeroes of a section of $Ts^{-1}(\Gr_2^+(TM))$ which is transverse to the zero section. The latter is $\chi(s^{-1}(\Gr_2^+(TM)))$ by the Poincar{\'e}--Hopf theorem.

The map obtained by scanning an embedded submanifold $f(\Sigma_g)$ is suitably transverse, and $s^{-1}(\Gr_2^+(TM)) = f(\Sigma_g)$, so the claimed identities hold.
\end{proof}

\begin{proposition}\label{prop:sections-components}
If $M$ is connected, so $H_0(M;\bZ) = \bZ$, then the map $\chi$ takes values in $2\bZ$. If $M$ is simply-connected and of dimension $d \geq 5$, then the map
$$\chi \times \pi \colon \pi_0(\Gamma_c(\Ss(TM) \to M)) \lra 2\bZ \times H_2(M;\bZ)$$
is a bijection.
\end{proposition}
\begin{proof}
The space of compactly supported sections is the space of compactly supported lifts along $p\colon \Ss(TM)\rightarrow M$ of the identity map of $M$. We will use the notation $F = \Ss(\bR^d) = \Th(\gamma_2^\perp \to \Gr_2^+(\bR^d))$ for the fibre of the map $p$, and suppose for simplicity that $M$ is compact.

The map $\Gr_2^+(\bR^d) \to \Gr_2^+(\bR^\infty)$ induces an isomorphism in cohomology in degrees $\leq d-1$, so
$$\bZ[e(\gamma_2)] \lra H^*(\Gr_2^+(\bR^d);\bZ)$$
is an isomorphism in degrees $\leq d-1$, and hence
$$u \cdot \bZ[e(\gamma_2)] \lra \widetilde{H}^*(\Ss(\bR^d);\bZ)$$
is an isomorphism in degrees $\leq 2d-3$. As there are cohomology classes $u \cdot e^i \in H^*(\Ss(TM), M;\bZ)$ restricting to $u \cdot e(\gamma_2)^i$ on the fibre, the bundle $p$ satisfies the conditions of the (relative) Leray--Hirsch theorem in degrees $\leq 2d-3$, so
$$H^*(M;\bZ) \otimes (u \cdot \bZ[e(\gamma_2)]) \lra H^*(\Ss(TM), M;\bZ)$$
is an isomorphism in this range of degrees.

Let us show that $\chi$ takes even values. As $q^*TM = \gamma_2 \oplus \gamma_2^\perp$, we calculate in the $\bF_2$-cohomology of $\Th(\gamma_2^\perp \rigtharrow\Gr_2^+(TM))$%
$$\mathrm{Sq}^2(u) = u \cdot w_2(\gamma_2^\perp) = u \cdot (w_2(\gamma_2) + q^*w_2(M)) = u \cdot e + u \cdot q^*w_2(M)$$
and so pulling back via $c$ we have
$$\mathrm{Sq}^2(u) = u \cdot e + u \cdot p^*w_2(M) \in H^{d}(\Ss(TM);\bF_2).$$
Thus for any section $s$ we have $\mathrm{Sq}^2(s^*u) = s^*(u \cdot e) + s^*u \cdot w_2(M)$ in the $\bF_2$-cohomology of $M$. However $\mathrm{Sq}^2(s^*u) = v_2(M) \cdot s^*u = w_2(M) \cdot s^*u$ as $M$ is simply-connected, and so $s^*(u \cdot e) = 0 \in H^d(M;\bF_2)$. Thus $s^*(u \cdot e) \in H^d(M;\bZ)=\bZ$ is even, as claimed.

We will be required to know $\pi_k(\Ss(\bR^d))$ for $k \leq d$. By considering the cohomology calculation above in degrees $\leq 2d-3$, we see that as long as $d \geq 4$ then $\Ss(\bR^d)$ has a cell structure whose $(d+1)$-skeleton $X$ consists of a $(d-2)$-cell and a $d$-cell. Because
$$\Sq^2(u) = u \cdot w_2(\gamma_2^\perp) = u \cdot w_2(\gamma_2) \neq 0,$$
we see that the $d$-cell is attached along a non-trivial map $S^{d-1} \to S^{d-2}$, which must be the Hopf map as long as $d \geq 5$. Thus $X \simeq \Sigma^{d-4}\bC\bP^2$, and it remains to calculate the homotopy groups of this space in degrees $\leq d$. By the Blakers--Massey theorem, the map of pairs $\pi_k(S^{d-2}, S^{d-1}) \to \pi_k(X,*)$ is an isomorphism for $k\leq 2d-5$, so for $k\leq d$ as we have assumed that $d \geq 5$. Calculating by means of the known stable homotopy groups of spheres in this range shows that
$$\pi_{d-2}(\Ss(\bR^d)) \cong \bZ, \quad\quad \pi_{d-1}(\Ss(\bR^d)) =0 ,\quad\quad \pi_{d}(\Ss(\bR^d)) \cong \bZ,$$
and also that the Hurewicz map is injective in these degrees. (When $d=5$ we must use that $\pi_5(S^3)=\bZ/2\langle \eta^2\rangle$, even though it is not in the stable range; this may be found in Toda's book \cite{Toda}.)

Let $s_0$ and $s_1$ be two sections of $p$ which have the same value of the invariants $\pi$ and $\chi$, and let us show that they are fibrewise homotopic. We obtain a diagram
\begin{equation}\label{eq:lifting}
\begin{gathered}
\xymatrix{
\{0,1\} \times M \ar[r]^-{s_0 \cup s_1} \ar[d]& \Ss(TM) \ar[d]^p \ar[r]^-{p \times u} & M \times K(\bZ, d-2) \ar[d]^-{\mathrm{proj}}\\
[0,1] \times M \ar@{-->}[ru] \ar[r]^-{\mathrm{proj}} & M \ar@{=}[r]& M
}
\end{gathered}
\end{equation}
and we must supply the dashed arrow. By obstruction theory, the first possible obstruction lies in
$$H^{d-1}([0,1]\times M, \{0,1\} \times M ; \pi_{d-2}(\Ss(\bR^d))) \cong H^{d-2}(M; \bZ)$$
and it must be $\pi(s_0)-\pi(s_1)$, as it agrees with the first possible obstruction for the (trivial) right-hand fibration in \eqref{eq:lifting}. But we have assumed that $\pi(s_0)-\pi(s_1)$ is zero, so there is no obstruction at this stage. The next possible obstruction lies in
$$H^{d+1}([0,1]\times M, \{0,1\} \times M ; \pi_{d}(\Ss(\bR^d))) \cong H^{d}(M; \bZ),$$
and by comparing it with the trivial bundle $M \times K(\bZ,d) \to M$ via $p \times (u \cdot e)$, as above, and using the injectivity of the Hurewicz map, we see that this obstruction vanishes if and only if $\chi(s_0)-\chi(s_1)$ does; we have assumed this. As $M$ has dimension $d$, there are no higher obstructions to constructing the dotted map, which gives a fibrewise homotopy between the two sections.

\end{proof}

\subsection{A scanning map independent of the metric}\label{section:closed2}

In order to prove Theorem \ref{thm:MainH} for background manifolds without boundary, it is very inconvenient that the source $\E^\tau(\SS_g,M)$ of the scanning map (see page \pageref{def:ssv}) depends on the metric so does not admit an action of the diffeomorphism group of $M$, and as a result, we cannot use the techniques in Section \ref{section:retractile} to prove that certain maps are fibrations. We present here an alternative definition of the scanning map which is defined for surfaces in smooth manifolds (as opposed to surfaces in Riemannian manifolds).

Let $\en{g}{M}$ be the space of pairs $(W,u)$, where $W\in \e{g}{M}$ and $u\colon NW\to M$ is a tubular neighbourhood, topologised as a quotient of $\TubEmb(\SS_{g,b},M;\delta)$ by the action of $\Diff(\SS_{g,b})$. It has an action of $\Diff(M)$ in the obvious way (cf. Section \ref{ss:thickened-embeddings}).

\begin{lemma}\label{lemma:retractil-enu} The space $\en{g}{M}$ is $\Diff(M)$-locally retractile.
\end{lemma}
\begin{proof}
As in the proof of Lemma \ref{lemma:111}, one can produce a local trivialisation of the quotient map using a $\Diff(M)$-local retraction of $\e{g,b}{M;\delta}$ and therefore prove that the quotient map is a locally trivial fibration, and, since the source of the quotient map is $\Diff(M)$-locally retractile, the target is also $\Diff(M)$-locally retractile by Lemma \ref{lemma:retractil-source}.
\end{proof}

For a vector space $V$, define $\Sss(V)$ to be the Thom space of the tautological bundle $\gamma_{n-2}(V)\to \Gr_{n-2}^+(V)$ over the Grassmannian of oriented $2$-planes in $V$. For a vector bundle $E$, let $\Sss(E)$ be the result of applying fibrewise this construction.

We then define a map $M \times \en{g}{M} \to \Gamma_c(\Sss(TM)\to M;\infty)$ by 
\begin{align*}
(p, W, u) & \longmapsto 
\begin{cases}
\infty_p \in \Sss(T_p M) & \text{if } \begin{array}{l} p \nin u(NW),\end{array}\\
\left(D_vu_{q}(T_vN_qW), D_vu_q(-v)\right) \subset \gamma_{n-2}(T_pM) & \text{if } 
\end{cases}
\end{align*}
where $D_vu_q\colon T_{v}N_qW\to T_pM$ is the differential of the restriction of the tubular neighbourhood to the subspace $N_qW\subset NW$ and we consider $-v\in N_q(M)$ as lying in $T_vN_qM$ using the canonical isomorphism $N_qM\cong T_vN_qM$. The adjoint to this map,
\begin{equation*}
\sss \colon \en{g}{M} \lra \Gamma_c(\Sss(TM) \to M;\infty),
\end{equation*}
is the (smooth) \emph{scanning map}. The forgetful map $\en{g}{M} \to \e{g}{M}$ is a locally trivial fibration because it is $\Diff_c(M)$-equivariant and $\e{g}{M}$ is $\Diff_c(M)$-locally retractile. As its fibre is contractible by Lemma \ref{lemma:tubular-contractible}, it follows that the map is homotopy equivalence. 

If a metric on $M$ is given, we can canonically identify $\Sss(TM)$ with $\Ss(TM)$, and the diagram given by the Riemannian and the smooth scanning maps
\[\xymatrix{
\Gamma_c(\Ss(TM)\to M) \ar[rr]^\cong && \Gamma_c(\Sss(TM)\to M) \\
\E^{\tau}(\SS_g,M) \ar[rr]^\simeq \ar[u]^{\s}&& \en{g}{M} \ar[u]^{\sss}
}\]
commutes on the nose, hence the scanning map $\sss$ is a perfect replacement of the scanning map $\s$.

We denote by $\infty_p$ the point at infinity of $\Sss(T_pM)$ and write $\infty = \bigcup_p{\infty_p}$, and we define the support of $f\in \Gmc{M}$, denoted $\supp f$, to be the closure of $M\backslash f^{-1}(\infty)$.

\subsection{Proof of Theorem \ref{thm:MainH} when \boldmath$\partial M = \emptyset$}\label{section:closed3}

\begin{df}\label{defn:DiscRes}
Let $\g{g}{M}_\bullet$ be the semi-simplicial space whose  $i$-simplices are tuples $(W,u;d_0,\ldots,d_i)$, where
\begin{enumerate}[(i)]
\item\label{it:101} $(W,u)\in \en{g}{M}$;
\item\label{it:100} $d_0,\ldots,d_i$ are disjoint embeddings of the closed unit disc into~$M$; 
\item\label{it:102} $d_j(0)$ does not belong to the image of $u$, for all $j$. 
\end{enumerate}
The semi-simplicial structure is as usual given by forgetting data, which gives a semi-simplicial space augmented over $\en{g}{M}$.
\end{df}

\begin{proposition} If the dimension of $M$ is at least $3$, then $\g{g}{M}_\bullet$ is a resolution of~$\en{g}{M}$.
\end{proposition}
\begin{proof}
Let $G_\bullet$ be the semi-simplicial space constructed similarly to the above, with $i$-simplices consisting of those tuples $(W,u;d_0,\ldots,d_i)$ such that condition (\ref{it:101}) above holds, as well as
\begin{enumerate}[(i$^\prime$)]
\setcounter{enumi}{1}
\item $d_0,\ldots,d_i \colon D^d \hookrightarrow M$ are embeddings of the closed unit disc into $M$ such that the $d_j(0)$ are distinct;
\item $d_j(0) \cap W = \emptyset$, for all $j$.
\end{enumerate}
There is an inclusion $\g{g}{M}_\bullet\hookrightarrow G_\bullet$, which is a levelwise weak homotopy equivalence, by shrinking the discs and the tubular neighbourhood $u$. Now $G_\bullet$ is an augmented topological flag complex over $\e{g}{M}$, so we apply Criterion \ref{criterion2}. The augmentation map is a fibration by Corollary \ref{Cerf-fibration}, hence has local sections, and given any finite (possibly empty) collection $(W,u,d_0),\ldots,(W,u,d_i)$ of $0$-simplices over $(W,u)$, the complement $M\setminus( W \cup \bigcup d_j(0))$ is a non-empty manifold of dimension at least $3$, so there is an embedding $d_{i+1}$ of a closed $d$-ball into it. Then $(W,u,d_{i+1})$ is a $0$-simplex orthogonal to all the former $0$-simplices. 
\end{proof}
\begin{proposition} There are fibrations
\[\en{g}{M\setminus \cup d_j(0)}\longrightarrow \g{g}{M}_i\longrightarrow C_i(M) =: \Emb(\{0,1,\ldots,i\} \times D^d, M)\]
where the fibre is taken over the point $(d_0,\ldots,d_i)$.
\end{proposition}
\begin{proof}
This is a consequence of Corollary \ref{Cerf-fibration}.
\end{proof}

In the notation of the last section, we let $\Gmg{M}{g}$ denote the collection of path components $\chi^{-1}(2-2g)$. Thus it consists of those sections which have ``formal genus $g$''.

\begin{df}
Let $\f{g}{M}_\bullet$ be the semi-simplicial space whose $i$-simplices are tuples $(f,(d_0,h_0),\ldots,(d_i,h_i))$, where 
\begin{enumerate}[(i)]
\item\label{it:201} $f\in \Gmg{M}{g}$;
\item\label{it:202} $d_0,\ldots,d_i \colon D^d \hookrightarrow M$ are disjoint embeddings of the closed unit disc of dimension $d$ into~$M$;
\item\label{04}\label{it:203} $h_0,\ldots,h_i \colon [0,1]\times M \to \Sss(TM)$ are homotopies of sections such that 
$$h_j(0,-) = f(-),\qquad d_j(0)\notin \supp h_j(1,-),$$ and the homotopy $h_j$ is constant outside of the set $d_j(D^d)$.
\end{enumerate}
The $j$th face map forgets $(d_j,h_j)$, and forgetting everything but $f$ gives an augmentation to the space $\Gmg{M}{g}$.
\end{df}

\begin{proposition} If $M$ has dimension at least $3$, then $\f{g}{M}_\bullet$ is a resolution of $\Gmg{M}{g}$.
\end{proposition}

\begin{proof}Let us define $F_g(M)_\bullet$ as the semi-simplicial space whose $i$-simplices are tuples $(f,(d_0,h_0),\ldots,(d_i,h_i))$ such that conditions (\ref{it:201}) and (\ref{it:203}) above hold and condition (\ref{it:202}) is replaced by
\begin{enumerate}[(i$^\prime$)]
\setcounter{enumi}{1}
	\item $d_0,\ldots,d_i\colon D^d\hookrightarrow M$ are embeddings such that the $d_j(0)$ are distinct,
\end{enumerate}
and whose face maps are given by forgetting data, and it has an augmentation to $\Gmg{M}{g}$ that forgets everything but $f$. There is an obvious semi-simplicial inclusion $\f{g}{M}_\bullet\hookrightarrow F_g(M)_\bullet$ over $\Gmg{M}{g}$, and the lemma will follow from the following statements:
\begin{enumerate}[(1)]
	\item the semi-simplicial inclusion is a levelwise weak homotopy equivalence and
	\item the augmentation of $F_g(M)_\bullet$ is a weak homotopy equivalence. 
\end{enumerate}
For the first statement, take a smooth function $\lambda\colon [0,\infty)\rightarrow [0,1]$ with 
\[\lambda([1,\infty)) = 0,\; \lambda([0,1/2])=1.\]
Consider the following deformation $H_s\colon F_g(M)_i\times (0,1]\rightarrow F_g(M)_i$ (which restricts to a deformation of $\f{g}{M}_i$):
\[
H_s(f) = f,\; H_s(d_j)(y) = d_j(sy), \;
H_s(h_j)(t,x) = 
\begin{cases} 
h_j(\lambda(||sy||)t,sy) &\mathrm{if}~x = d_j(sy) \\
x &\mathrm{otherwise}.
\end{cases}
\]
Under this deformation, every $i$-simplex eventually ends up, and stays, in the subspace $\f{g}{M}_i$. If $f \colon (D^n, S^{n-1}) \to (F_g(M)_i, \f{g}{M}_i)$ represents a relative homotopy class, then because $D^n$ is compact the map $h(-, t)\circ f$ has image in $\f{g}{M}_i$ for some $t$, so the homotopy class of $f$ is trivial.

For the second statement, note that $F_g(M)_\bullet$ is a topological flag complex augmented over $\e{g}{M}$ whose augmentation is a fibration by Lemmas \ref{cor:retractil-embedded} and \ref{lemma:retractil-target}. Given a possibly empty finite collection of $0$-simplices 
\[(f,d_0,h_0),\ldots,(f,d_i,h_i)\]
 over $f$, we may find an embedding of a disc $d_{i+1}$ such that $d_{i+1}(0)$ is different from the points $d_0(0),\ldots,d_i(0)$. We may also find a homotopy $h_{i+1}$ satisfying condition (\ref{it:203}) for the embedding $d_{i+1}$ and the section $f$, because the space $\Sss(T_{d_{i+1}(0)}M)$ is path connected. Hence the conditions of Criterion \ref{criterion2} hold, so the augmentation for $F_g(M)_\bullet$ is a weak homotopy equivalence.
\end{proof}

\begin{proposition}\label{prop:fibrations-sections} There are homotopy fibrations
\[\Gmg{M\setminus \cup d_j(0)}{g}\longrightarrow \f{g}{M}_i\longrightarrow C_i(M),\]
where the fibre is taken over the point $(d_0,\ldots,d_i)$.
\end{proposition}
\begin{proof}
The space $C_i(M)$ is $\Diff_\partial(M)$-locally retractile by Lemma \ref{prop:retractil-embeddings}, and the map is equivariant for the action of $\Diff_\partial(M)$; hence, by Lemma \ref{lemma:retractil-target}, this is a locally trivial fibration. The fibre is the space $\Fib_i$ of tuples $(f,h_0,\ldots,h_i)$ where $f\in \Gmg{M}{g}$ and $h_j$ is a homotopy of $f$ supported in $d_j$ such that $h_j(1,-)$ is a section supported away $d_j(0)$. Since the homotopies $h_j$ have disjoint support, we may compose them. There is a homotopy
\begin{align*}
H\colon I\times \Fib_i&\longrightarrow  \Fib_i\\
(t,(f,(h_0,\ldots,h_i)))&\longmapsto (H_t(f),H_t(h_0),\ldots,H_t(h_i))
\end{align*}
where 
\begin{align*}
H_t(f)(-) &= h_0(t,-)\circ\ldots\circ h_i(t,-) \\
H_t(h_j)(s,-) &= h_j(t+s(1-t),-).
\end{align*}
This homotopy deformation retracts $\Fib_i$ into the subspace $Y$ of those tuples $(f,h_0,\ldots,h_i)$ such that $d_j(0)\notin\supp h_j$ and $h_j$ is the constant homotopy. Finally, there is a map
\[Y\longrightarrow \Gmg{M\setminus \cup d_j(0)}{g}\]
given by sending $(f,h_0,\ldots,h_i)$ (recall that these homotopies are all constant) to $f_{|M\setminus \cup d_j(0)}$, and this map is a homeomorphism.
\end{proof}

By condition (\ref{it:102}) of Definition \ref{defn:DiscRes}, the scanning map
\[\sss\colon \en{g}{M}\longrightarrow \Gmg{M}{g}\]
extends to a semi-simplicial map $\sss_\bullet\colon \g{g}{M}_\bullet\to \f{g}{M}_\bullet$ given on $i$-simplices by sending each tuple $(W,u,d_0,\ldots,d_i)$ to the tuple $(\sss(W,u),(d_0,\Id),\ldots,(d_i,\Id)$, where $\Id$ denotes the constant homotopy.

\begin{proposition} The resolution $\sss_\bullet$ of the scanning map is a levelwise homology equivalence in degrees $\leq \frac{1}{3}(2g-2)$. Hence the scanning map is also a homology equivalence in those degrees.
\end{proposition}
\begin{proof}
The induced map on the space of $i$-simplices is a map of fibrations over $C_i(M)$, and the induced map on fibres is
\[\sss_i\colon \en{g}{M\setminus \cup d_j(0)}\longrightarrow \Gmg{M\setminus \cup d_j(0)}{g}.\]
 As $\sss_i$ is a scanning map, Theorem \ref{thm:MainH} for surfaces in a manifold with boundary (which was proven in the preceding two sections) asserts that $\sss_i$ is a homology equivalence in degrees $\leq \frac{1}{3}(2g-2)$. Note that although $M\setminus \cup d_j(0)$ does not have boundary, it does admit a boundary.
\end{proof}

\bibliographystyle{amsalpha}
\bibliography{biblio}

\end{document}